\documentclass[11pt,reqno]{amsart}
\usepackage[utf8]{inputenc}
\usepackage{amsthm,amssymb,esint,bm,mathtools}
\usepackage{dsfont}
\usepackage[margin=2cm]{geometry}
\usepackage[dvipsnames]{xcolor}
\usepackage{hyperref,hypcap}
\hypersetup{colorlinks=true,linkcolor=blue,citecolor=blue,bookmarksdepth=3}
\usepackage{enumitem}
\usepackage{comment}
\usepackage{cancel}
\usepackage{mathtools}
\usepackage{longfbox}
\usepackage{mathrsfs}
\usepackage{pgf, tikz}
\usepackage{mathtools}
\usepackage{graphicx}
\usepackage{bigints}

\hypersetup{
    colorlinks,
    citecolor=MidnightBlue,
    filecolor=black,
    linkcolor=MidnightBlue,
    urlcolor=black,
    linktoc=page
}

\newtagform{blue}[\textcolor{MidnightBlue}]{\color{black}(}{)}
\mathtoolsset{showonlyrefs}

\newtheorem{thm}{Theorem}[section]
\newtheorem{cor}[thm]{Corollary}
\newtheorem{lem}[thm]{Lemma}
\newtheorem{prop}[thm]{Proposition}

\theoremstyle{definition}

\newtheorem{rem}[thm]{Remark}

\numberwithin{equation}{section}

\newtheoremstyle{myrunin}
  {3pt}   
  {3pt}   
  {}      
  {}      
  {}      
  {}      
  { }     
  {\underline{#3.}} 

\theoremstyle{myrunin}

\newcounter{pstep}
\newenvironment{pstep}[1]
{%
  \refstepcounter{pstep}%
  \noindent
  \hspace*{\parindent}%
  \textbf{#1.}
  \ignorespaces
}
{%
  \par
}

\DeclareMathOperator{\supp}{supp}

\newcommand{\R}{\mathbb{R}}
\newcommand{\Z}{\mathbb{Z}}

\newcommand{\N}{\mathbb{N}}
\newcommand{\C}{\mathbb{C}}

\newcommand{\eps}{\varepsilon}

\newcommand{\m}{\mathfrak{m}}

\newcommand{\br}[1]{\langle#1\rangle}
\newcommand{\norm}[1]{\left\lVert#1\right\rVert}
\newcommand{\abs}[1]{\left\lvert#1\right\rvert}
\newcommand{\zz}{\mathcal{Z}}

\newcommand{\vertiii}[1]{{\left\vert\kern-0.25ex\left\vert\kern-0.25ex\left\vert #1 
		\right\vert\kern-0.25ex\right\vert\kern-0.25ex\right\vert}}

\makeatletter

\renewcommand{\tocsection}[3]{%
  \indentlabel{\@ifnotempty{#2}{\bfseries\ignorespaces#1 #2\quad}}\bfseries#3}
\renewcommand{\tocsubsection}[3]{%
  \indentlabel{\@ifnotempty{#2}{\ignorespaces#1 #2\quad}}#3}
\renewcommand{\tocsubsubsection}[3]{%
  \indentlabel{\@ifnotempty{#2}{\ignorespaces#1 #2\quad}}#3}

\newcommand\@dotsep{4.5}
\def\@tocline#1#2#3#4#5#6#7{\relax
  \ifnum #1>\c@tocdepth 
  \else
    \par \addpenalty\@secpenalty\addvspace{#2}%
    \begingroup \hyphenpenalty\@M
    \@ifempty{#4}{%
      \@tempdima\csname r@tocindent\number#1\endcsname\relax
    }{%
      \@tempdima#4\relax
    }%
    \parindent\z@ \leftskip#3\relax \advance\leftskip\@tempdima\relax
    \rightskip\@pnumwidth plus1em \parfillskip-\@pnumwidth
    #5\leavevmode\hskip-\@tempdima{#6}\nobreak
    \leaders\hbox{$\m@th\mkern \@dotsep mu\hbox{.}\mkern \@dotsep mu$}\hfill
    \nobreak
    \hbox to\@pnumwidth{\@tocpagenum{\ifnum#1=1\bfseries\fi#7}}\par
    \nobreak
    \endgroup
  \fi}
\AtBeginDocument{%
\expandafter\renewcommand\csname r@tocindent0\endcsname{0pt}
}
\def\l@subsection{\@tocline{2}{0pt}{2.5pc}{5pc}{}}
\def\l@subsubsection{\@tocline{3}{0pt}{4pc}{5pc}{}}
\makeatother

\usepackage{cancel}

\addtocontents{toc}{\protect\setcounter{tocdepth}{2}}

\newcommand{\nocontentsline}[3]{}
\newcommand{\tocless}[2]{\bgroup\let\addcontentsline=\nocontentsline#1{#2}\egroup}

\author[L. Ertzbischoff]{Lucas Ertzbischoff}
\address{CEREMADE, CNRS, Université Paris-Dauphine, PSL Research University, 75016 Paris, France}
\email{ertzbischoff@ceremade.dauphine.fr}

\author[C. Jurja]{Catalina Jurja}
\address{Institute of Mathematics, University of Zurich, 8057 Zurich, Switzerland}
\email{catalina.jurja@math.uzh.ch}

\author[K. Widmayer]{Klaus Widmayer}
\address{Faculty of Mathematics, University of Vienna, 1090 Vienna, Austria \newline \hspace*{12pt} \& Institute of Mathematics, University of Zurich, 8057 Zurich, Switzerland}
\email{klaus.widmayer@univie.ac.at \& klaus.widmayer@math.uzh.ch}

\title[Global dynamics in the 3d pressureless Euler-Poisson system for ions]{Global dynamics in the 3d pressureless Euler-Poisson system for ions}

\let\origmaketitle\maketitle
\def\maketitle{
  \begingroup
  \origmaketitle
  \endgroup
}

\begin{document}
\usetagform{blue}

\begin{abstract} 
In this work we consider the pressureless Euler-Poisson system on $\R^3$ describing ion dynamics, which is a common model in the study of cold plasma. We prove global regularity and scattering for small, irrotational velocity fields around a constant density profile. At the heart of the result is a dispersive mechanism leading to amplitude decay at the full rate in $3$D. This fast decay feature and a favourable structure of the nonlinearity allow to control quadratic  interactions globally.
\end{abstract}

\maketitle

\tableofcontents

 \section{Introduction}

\subsection{Motivation and Main Result}\label{sec:Motivation-Main-Result}
We consider the following pressureless Euler-Poisson system set on  $\R^3$:
\begin{equation}\label{eq:EP}
\left\{
      \begin{aligned}
      \partial_t \varrho + \mathrm{div}(\varrho v) &=0, \\[1mm]
\partial_t v +(v \cdot \nabla )v &=-\nabla \Phi,  \\[1mm]
e^{\Phi}-\Delta \Phi&= \varrho,
\end{aligned}
    \right. 
\end{equation}
where $\varrho:\R_+\times\R^3\to \R$ is the plasma density, $v:\R_+\times \R^3\to \R^3$ is the plasma velocity, and $\Phi:\R_+\times\R^3 \to \R$ is the electric potential. The differential operators without subscript refer to the ones related to the space variable $x \in \R^3$. In this system, where all physical constants have been normalized, the two first equations describe the balance of mass and momentum, while the third one is an elliptic equation giving the value of the potential $\Phi$ at each time $t{\geq}0$.
In this work, we are interested in studying global dynamics of \eqref{eq:EP} arising as initially small perturbations of a certain simple steady state. 
The Euler-Poisson system is a standard model for describing the dynamics of a plasma at the fluid level:  here, we have considered the model for ions in a fixed background of electrons (assuming the latter have already reached their thermodynamic equilibrium, as they are significantly lighter and faster than ions), the density of which satisfies the standard Maxwell-Boltzmann law, proportional to $e^\Phi$ (see \cite{BardosGolseNguyenSentis, GrenierGuoPausaderSuzuki} for some derivations on the kinetic and fluid side).

  A key feature of \eqref{eq:EP} is that there is no pressure gradient in the momentum equation, a manifestation of what is typically coined the cold plasma ansatz \cite[\textcolor{MidnightBlue}{Ch.\ 4}]{chen2015introduction} -- see also the discussion in (6) below for its connection with the kinetic description of plasmas.

\medskip
In this article, we are concerned with the asymptotic stability of the following constant steady state of equation \eqref{eq:EP}: 
\begin{align}\label{eqn:steady-state}
    (\varrho, v, \Phi)=(1, 0, 0).
\end{align}
Setting $\varrho=1+\rho$ in \eqref{eq:EP}, the system on the perturbation reads
\begin{equation}\label{eq:EP-perturbed-full}
\left\{
      \begin{aligned}
      \partial_t \rho + \mathrm{div}(v)+\mathrm{div}(\rho v) &=0, \\[1mm]
\partial_t v +(v \cdot \nabla )v &=-\nabla \Phi,  \\[1mm]
e^{\Phi}-\Delta \Phi&= 1+\rho.
\end{aligned}
    \right.
\end{equation} 
The main result of the paper yields the global stability of the steady state \eqref{eqn:steady-state} for irrotational velocity fields through a dispersive mechanism that allows to control the nonlinear interactions. 

\begin{thm}\label{main-thm}
 There exist $N_0\in\N$, $\eps_0>0$ and a norm $X$ such that
if $v_0$ is irrotational, i.e.\
$\nabla \times v_0=0$,   and if for some $0<\eps\leq \eps_0$ there holds
\begin{align}\label{eq:assump-datasmallORIGINAL}
    \norm{\rho_0}_{H^{N_0-1}\cap \dot{H}^{-1}}+\norm{v_0}_{H^{N_0} {\cap \dot H^{-1}}}
    \leq \eps,\qquad 
\norm{\abs{\nabla}^{-1}\rho_0 \pm 
{i\br{\nabla}\abs{\nabla}^{-2}\mathrm{div}(v_0)}}_{X}\leq \eps,
\end{align}
 then the system \eqref{eq:EP-perturbed-full} with initial data $(\rho_0,v_0)$ has a unique global solution $(\rho,v)\in  \mathscr{C}\left(\R_+; H^{N_0-1}\times H^{N_0}\right) $ with $\nabla\times v=0$. Moreover, for all $t>0$ the solution satisfies:
 \begin{align}
      \label{csqTHM-bound}\norm{\rho(t)}_{H^{N_0-1}\cap \dot{H}^{-1}}+\norm{v(t)}_{H^{N_0} {\cap \dot H^{-1}}}&\lesssim \eps, \\
       \label{csqTHM-decay}\norm{\rho(t)}_{L^\infty}+\norm{v(t)}_{L^\infty} &\lesssim (1+ t )^{-(1+\beta)} \eps,
 \end{align}
 for some $\beta \in ]0, \frac{1}{5}[$, and scatters linearly in $L^2$.
\end{thm}

Independently and concurrently with the present work, a global existence result similar to Theorem \ref{main-thm} has been proved in \cite{SongiEP}. This and the present work appear to be the first to establish examples of global dynamics for the $3$D pressureless Euler-Poisson system \eqref{eq:EP} for ions. We refer to Section \ref{sec:Main-Result-Dispersive} for a more precise statement of Theorem \ref{main-thm} in the dispersive framework, as well as its proof. We now make several comments concerning the main result and more generally the system \eqref{eq:EP-perturbed-full} and related models. 

\begin{pstep}{(1) Irrotationality and dispersion}
Irrotational flows are a natural class of solutions to \eqref{eq:EP-perturbed-full}, since this condition is propagated over time.\footnote{Indeed, taking a curl in the momentum balance in \eqref{eq:EP-perturbed-full} yields
\begin{align*}
        \partial_t (\nabla\times v) +(v \cdot \nabla )(\nabla\times v) &=((\nabla\times v) \cdot \nabla )v \implies    \frac{\mathrm{d}}{\mathrm{d}t} \norm{\nabla\times v }_{L^2}^2 \lesssim \norm{\nabla v }_{L^\infty} \norm{\nabla \times v }_{L^2}^2.
    \end{align*}}
Irrotationality reduces the complexity of the dynamics: solutions can be described by two scalar unknowns (classical choices are the density and a velocity potential). At the linearized level, these exhibit purely oscillatory behavior with a dispersion relation given by $\pm i\Lambda(\xi)=\pm i\frac{\abs{\xi}}{\br{\xi}}$. The associated semigroup leads to amplitude decay in $L^\infty$ at the full, non-degenerate rate $t^{-\frac{3}{2}}$ (see Proposition \ref{prop:lin-decay}). This dispersive effect is at the heart of the global existence of solutions in Theorem \ref{main-thm}. The main challenge we address is to establish such a decay estimate for the solution to the nonlinear problem -- see Section \ref{sec:outline-of-proof} below for more details on this.
\end{pstep}

\begin{pstep}{(2) 
Functional framework} Aside from small energy for the initial data, Theorem \ref{main-thm} requires smallness in the $X$-norm, as defined in \eqref{def:X-norm}. Roughly speaking, up to regularity assumptions, it satisfies $\norm{f}_X\sim \norm{|x|^{1+\beta}f}_{L^2}$ for a small $\beta>0$. In order to enable Fourier and oscillatory integral techniques, it is implemented via a suitably weighted $L^2$-norm across frequency and space localized pieces of the solution. This strikes a balance that allows it to stay bounded for the nonlinear solution (essentially via the use of normal forms, see Section \ref{sec:outline-of-proof} for more details), while also guaranteeing decay of the nonlinear solution at the time integrable rate $t^{-(1+\beta)}$, see e.g.\ \eqref{csqTHM-decay} and Proposition \ref{prop:lin-decay}. As a consequence, by standard energy estimates, initially bounded $L^2$ Sobolev norms of the solution stay bounded. These arguments are combined in a continuity argument for appropriately chosen dispersive unknowns, see Proposition \ref{prop:bootstrap}.

Let us also remark that neutrality (i.e.\ $\int\rho_0(x)dx=0$) is not required in our result. The numerology of the parameters involved in our norm ($\beta<\frac12$ and $\gamma>-\frac12$) suggests that maybe a faster decay rate can only be propagated for neutral initial data, but this remains to be understood.
\end{pstep}

    \begin{pstep}{(3) Related works for pressureless systems}  As pointed out above, {a central} feature of \eqref{eq:EP} and \eqref{eq:EP-perturbed-full} is the absence of a pressure gradient in the momentum equation, which contrasts with the usual compressible Euler equations. While the existence of global solutions is still largely open in the pressure and pressureless cases (see also paragraph (4) below), there is scarcely any literature in the pressureless case.
    To that respect, it is even expected that, at least in the one-dimensional case, pressureless Euler systems are roughly speaking relevant to describe the dynamics of so-called ``sticky particles'', where mass concentration phenomena (including the formation of measure-valued solutions) naturally arise -- see for instance \cite{breniergrenier1998sticky,nguyen2008pressureless,brenierandco2013sticky}.

    In the realm of classical solutions, 
    a main line of research was actually devoted to 
     the pressureless Euler--Poisson system for electron dynamics in a fixed  background of ions.
     In one space dimension, this model exhibits a sharp critical-threshold phenomenon: depending on pointwise conditions on the initial data, smooth solutions either exist globally in time or break down in finite time. This behaviour was first identified in \cite{EngelbergLiuTadmor}, and further refined in subsequent works (extensions to non-constant background states \cite{choi2025critical}, or  higher dimension under radial symmetry assumptions  \cite{WeiTadmorBae-critical, carrillo2023existence}).
    For the ionic system (with the Maxwell-Boltzmann relation for the electrons), loss of regularity and singularity formation  with precise blow-up profiles has been  intensively studied \cite{choi2025critical,BaeChoiKwon-singularity, BaeKimKwon-singularity} recently in the one-dimensional case, showing in particular mass concentration. We also refer to \cite{haragus2002linear, bae2025emergence} about the stability of travelling-wave solutions, in connection with the emergence of singularities.  In higher dimensions, and in particular in three space dimension, the situation seems to be much less understood in the ionic case, especially for the full Maxwell-Boltzmann law. 
    Let us finally also mention the works \cite{LannesLinaresSaut, Pu} on the Euler-Poisson system for ions (with or without pressure) in the long-wave and small-amplitude regime: they provide a derivation of the Zakharov-Kuznetsov equation in dimension three if a constant magnetic field is added in the momentum equation, as well as a a derivation of the 2d KP-II equation.
    \end{pstep}

     \begin{pstep}{(4) Related models in the presence of pressure}
In contrast to our setting, the compressible Euler system \textit{with pressure} has been more extensively studied in the literature (in what follows all the discussed models have a pressure gradient).
 It is well-known (see \cite{Sideris,christodoulou2007formation}) that for the purely neutral compressible Euler equations, even smooth irrotational small initial data can lead to the formation of shock waves over large times. For this system, irrotational perturbations of a constant background only satisfy a quasilinear wave equation without null-structure, for which a global result is not expected. It turns out that coupling the dynamics of such a compressible fluid with a self-consistent electric field can have a stabilizing effect, at least for small, irrotational perturbations of a constant density state. This remarkable fact was first understood for the Euler-Poisson system for electrons in work of Guo \cite{Guo-electrons3d}, where dispersion of Klein-Gordon type leads to global smooth solutions in dimension three for neutral perturbations
(a condition that was subsequently removed in \cite{GermainMasmoudiPausader-elec}). The case of ions was then studied by Guo and Pausader in \cite{Guo-Pausader-ions3d}, where a more subtle (almost wave-like) but still non-degenerate dispersive effect yields global solutions for suitable small, irrotational perturbations. These results were then extended to dimension $2$ in \cite{Ionescu-Pausader-electrons2d, pressureEP-2d-elec1,pressureEP-2d-elec2} for the electron case under a neutrality condition, and recently in \cite{ions2d} for the ion case. For the $1$D case, we refer to \cite{Guo1d}, where electrons dynamics are investigated. The analysis of these kind of stabilizing effects has culminated in the construction of global solutions to the full Euler-Maxwell system, where the electromagnetic coupling is taken into account \cite{GermainMasmoudi-EulerMaxwell, GermainMasmoudiPausader-elec,IP14-KG-3D, GuoIonescuPausader-certainplasma, GuoIonescuPausader-EM3d, DIP17-Euler-Maxwell-2D}. There has also been a growing interest in the existence and stability of other physically relevant equilibria for the Euler-Poisson system, that are not constant, such as periodic travelling
waves or solitary waves, see \cite{CordierDegondMarkoSchmeiser, BaeKwon-existence, BaeKwon-LINSTAB, RoussetSun, NobleRodriguesSun,GuelmameHmidiHouamedRousset} and references therein.
     \end{pstep}

\begin{pstep}{(5) Remark on the model for electrons}
 In the pressureless Euler-Poisson equation describing the dynamics of electrons (in a fixed background of ions), the elliptic equation in \eqref{eq:EP} is replaced by $-\Delta \Phi=\varrho-1$. Contrary to the case of ions considered in this work, the dispersion relation in this case is a constant. In particular, this oscillatory behavior does not feature spatially spreading waves, and no amplitude decay should be expected at the linearized level. This preliminary analysis indicates that the global dynamics of electrons may be substantially harder to understand, in sharp contrast to the case with pressure \cite{Guo-electrons3d}.
\end{pstep}

    \begin{pstep}{(6) Link with Vlasov-Poisson equation}
Interestingly, the pressureless nature of  \eqref{eq:EP} allows to connect it to the classical Vlasov-Poisson system
\begin{equation}\label{eq:VP}
\left\{
      \begin{aligned}
\partial_t v +v \cdot \nabla_x f -\nabla_x \Phi \cdot \nabla_v f&=0,  \\[1mm]
e^{\Phi}-\Delta_x \Phi&= \varrho-1,
\end{aligned}
    \right.  
\end{equation}
a kinetic equation (i.e.\ transport in the phase space) on the distribution function $f: \R_+\times \R^3\times \R^3 \to \R_+$ for the ions, that is also widely used in plasma physics. We refer for instance to \cite{Bouchut, DanielMikaela1D, MeganMikaelaR3} as well as to \cite{HKR,BedrossianMasmoudiMouhot, HKNR-screen, IRW, Wei} for the often commonly used screened Coulomb case with a linearized exponential in the elliptic equation. The formal link between \eqref{eq:VP} and \eqref{eq:EP} is the following: if $(\rho, v)$ is a smooth solution to  \eqref{eq:EP}, then $f(t,x,v)=\rho(t,x)\delta_{v=v(t,x)}$ is a solution in the sense of distributions to \eqref{eq:VP}. Solutions to the pressureless Euler-Poisson system therefore yield monokinetic solutions (hence singular, because rough in the velocity variable) to the Vlasov-Poisson equation. Note that Dirac masses in velocity formally correspond to Gaussian profiles in velocity in the zero temperature regime, justifying the cold plasma denomination.
\par This link\footnote{In the case where the pressure gradient is considered, one should start from the collisional Vlasov-Poisson-Boltzmann equation and perform a suitable hydrodynamic limit to derive the compressible Euler-Poisson system with pressure, hence including some thermodynamics in the system (see e.g.\ \cite{Guo-derivEP, deriveEP-ions} and references therein).} between general pressureless Euler systems and Vlasov equations is well-known in the fluids community (see for instance the pioneering works of Zakharov \cite{Zakharov}, Brenier \cite{Brenier89, Br97} and Grenier \cite{Grenier96}) and has recently regained some attention from the mathematical point of view (see e.g.\ \cite{Baradat, gagnebin2025relativistic}). In \cite{BEHK} a general superposition of monokinetic solutions is considered in the so-called multiphasic framework, and allows for a rigorous derivation of \eqref{eq:EP} from \eqref{eq:VP} for well-prepared data. Let us also mention that, at least for the unscreened Coulomb interaction case, the pressureless Euler-Poisson system can be derived by a mean-field limit from Newton's second law \cite{Serfaty}, or by a semiclassical/mean-field limit from the $N$-body Schrödinger/von Neumann equation \cite{GolsePaul}.

Aside from providing an example of global dynamics for \eqref{eq:EP}, Theorem \ref{main-thm} can thus be interpreted as a stability result in the monokinetic regime of the Vlasov--Poisson equation \eqref{eq:VP}, since the steady state \eqref{eqn:steady-state} corresponds to the monokinetic profile $\delta_{v=0}$, which arises as a limiting case of a Maxwellian in velocity (at zero temperature). The latter is one of the main prototypes of homogeneous profiles (including more general decreasing solutions in velocity) for which nonlinear asymptotic stability of \eqref{eq:VP} with screened interactions is known on the whole space, in the context of Landau damping \cite{BedrossianMasmoudiMouhot, HKNR-screen, Screened-Trinh, Screened-improve, Screened-improv2}. In contrast, for the case of electrons the stability of homogeneous equilibria on $\R^3$ is still largely ill understood \cite{BMM-wholespaceElec, HKR-wholespaceelec, IPWW24-VP, NWZ-NonlinwholespaceElec-new}. We highlight that the case of electrons thus seems to pose more difficulties also in the pressureless fluid model, see point (5) above.
\end{pstep}

\subsection{Outline of the Proof}\label{sec:outline-of-proof}
The proof of Theorem \ref{main-thm} follows the fruitful approach in obtaining long-time dynamics in quasilinear dispersive equations developed and extended in e.g.\ \cite{GP-DefocusingNLS-GNT, Guo-Pausader-ions3d, GMS12-GWW3D, PS13-STR-Waves, Ionescu-Pausader-electrons2d, GermainMasmoudi-EulerMaxwell, IP14-KG-3D, IP15-WW, DIPP17, PW18, GPW23-EC, JW24} and references therein. It is conducted in Section \ref{sec:Main-Result-Dispersive} and is based on a standard bootstrap argument. 
\smallskip

{The main difficulty of the paper is treating the quadratic interactions of the nonlinearity. Thus in what follows and throughout Sections \ref{sec:set-up}--\ref{sec:X-norm} we treat the perturbed system with the linearized Maxwell-Boltzmann law, that is the system with the screened elliptic equation on the electric potential:
\begin{equation}\label{eq:EP-perturbed}
\left\{
      \begin{aligned}
      \partial_t \rho + \mathrm{div}(v)+\mathrm{div}(\rho v) &=0, \\[1mm]
\partial_t v +(v \cdot \nabla )v &=-\nabla \Phi,  \\[1mm]
\Phi-\Delta \Phi&= \rho.
\end{aligned}
    \right.
\end{equation} 
The additional arguments to treat the  nonlinear Poisson equation \eqref{eq:EP-perturbed-full}, mainly new cubic contributions, are presented in  Appendix \ref{appendix}. We remark that the cubic remainder is not handled via a dispersive analysis as presented below for the quadratic terms, but via suitable energy estimates.}
\smallskip

 \begin{pstep}{Dispersive formulation} As mentioned in the first comment following the theorem, the main observation is that the system \eqref{eq:EP-perturbed} exhibits oscillations. 
  Thus, the first step is to find favourable oscillating unknowns. After diagonalizing the linearized system in Lemma \ref{lem:linearization}, one equivalently (granted the irrotationality condition) rewrites the perturbed system \eqref{eq:EP-perturbed} (see Proposition \ref{prop-Duhamel}) as
\begin{align}
    \eqref{eq:EP-perturbed}\quad\Longleftrightarrow \quad\partial_tZ_\pm + NL_\pm(Z_+,Z_-)=\pm i\Lambda Z_\pm, \qquad Z_\pm=\abs{\nabla}^{-1}\rho\pm i \Lambda^{-1}\abs{\nabla}^{-1}\mathrm{div}v,\quad \Lambda=\frac{\abs{\nabla}}{\br{\nabla}}.
\end{align}
Here $NL_\pm$ is the quasilinear, quadratic nonlinearity in \eqref{eq:EP-perturbed} written in terms of $Z_\pm$. We note here that 
 the pressureless nature of the system imposes a 
 {hierarchy} of derivatives between $\rho,\, v$ and $Z_\pm$.
The choice of baseline regularity for $Z_\pm$ is governed by the lowest regularity for which energy estimates for the nonlinear system are still available. This choice is then reflected in the structure of the nonlinearity, as discussed below. 
\end{pstep}
\begin{pstep}{Decay}
The central feature of this system is its dispersive nature through the operator $\Lambda$ and notably the amplitude decay it generates. As a matter of fact, $\Lambda$ (or its symbol in the frequency space) is non-degenerate in the sense of the stationary phase lemma:
\begin{align}
    \det\mathrm{Hess}\,\Lambda(\xi)=-\frac{3}{|\xi|(1+|\xi|^2)^{11/2}}\neq 0 \quad \forall\xi\in \R^3 {\setminus \lbrace 0 \rbrace}.
\end{align}
A direct stationary phase lemma argument yields $L^\infty$ decay at the rate $t^{-\frac{3}{2}}$ for the associated linear semigroup, at the cost of $\dot{B}_{1,1}^\frac{1}{2}\cap\dot{B}_{1,1}^6$ regularity, see Remark \ref{remark:lin-decay}. In our setting however, we only propagate an integrable decay rate for the nonlinear solution, and not the full one, see Proposition \ref{prop:lin-decay}. This suffices for the energy to remain bounded as shown in Corollary \ref{cor:blow-up-criterion-decay}. Moreover, we note that in certain space and frequency regions a refined decay estimate \eqref{eqn:linear-decay-j<m} is available which is crucial in propagating nonlinear control as discussed below.
\end{pstep} 

\begin{pstep}{Nonlinear analysis set-up and the $X$-norm} The proof of Theorem \ref{main-thm} follows via a bootstrap argument (see Proposition \ref{prop:bootstrap}) involving the propagation of the $X$-norm, see the definition below. To that end, we define the profiles of the solutions $\zz_\pm=e^{\mp it\Lambda}Z_\pm$, and, as a consequence of the Duhamel formulation (see Proposition \ref{prop-Duhamel}), we reduce the problem to bounding quadratic terms of the form
\begin{align}
    \mathcal{F}\left(\mathcal{B}_\m(f,g)\right)(t,\xi)=\int_0^t\int_{\R^3} e^{is\Phi(\xi,\eta)}\m(\xi,\eta)\widehat{f}(\xi-\eta)\widehat{g}(\eta)d\eta ds
\end{align}
in the $X$-norm.
Here $\m$ is a multiplier encoding the structure of the nonlinearity and $f,g\in \{\zz_\pm\}$ and the integral oscillates with a phase $\Phi(\xi,\eta)=\Phi^{\mu\nu}(\xi,\eta)= \Lambda(\xi)+\mu\Lambda(\xi-\eta)+\nu\Lambda(\eta)$, $\mu,\nu\in\{+,-\}$.   Controlling these terms relies on an intricate analysis of nonlinear interactions. In addition to standard Littlewood-Paley localizations of the frequency sizes, we simultaneously localize the size of the space variable on dyadic scales, decomposing $f=\sum_{j+k\geq 0}Q_{jk}f$, where (informally speaking) $Q_{jk}f$ is the projection of $f(x)$ (resp.\ $\widehat{f}(\xi)$) to scales $|x|\sim 2^{j}$ and $|\xi|\sim 2^{k}$, and $j+k\geq 0$ formalizes an uncertainty principle -- see Section \ref{sec:localizations} for detailed definition and a thorough discussion. The main norm incorporates these localizations and is defined in \eqref{def:X-norm} as follows (for fixed $\beta,\; \gamma$)
\begin{align}
    {\norm{f}_X=\sup_{k+j\geq 0}2^{8k^+}2^{(1+\beta)(k+j)}2^{\gamma k^-}\norm{Q_{jk}f}_{L^2}, \quad \gamma+\beta=-\frac{1}{6}-\frac{2}{3}\beta}, \ \ \beta \in \Big] 0, \frac{1}{5} \Big[.
\end{align}
The choice of this norm  {serves two main purposes}. On the one hand it incorporates enough regularity $2^{8k^+}$ and spatial weight $2^{(1+\beta)j}$ to obtain the decay rate $t^{-(1+\beta)}$ ($\beta<\frac{1}{5}$) for the semigroup of the problem, see Proposition \ref{prop:lin-decay}. This decay suffices for the bootstrap argument as the energy remains bounded for all times, see Proposition \ref{prop:Sobolev-estimate}, Corollary \ref{cor:blow-up-criterion-decay}. Moreover, we note that this choice of norm and in particular, the additional $2^{\gamma k^-}$ allows for faster decay rates $t^{-\theta}$, $1+\beta<\theta<\frac{3}{2}$ in certain regions of the space-frequency space, see \eqref{eqn:linear-decay-j<m}. On the other hand, the nonlinearity can be controlled in this norm, also due to this refined decay estimate. We detail the analysis of interactions in the following paragraphs.

\par  Since the $X$-norm involves spatial and frequency localizations, the bilinear terms above behave differently depending on the localized region. Following a standard approach in treating such terms, we may significantly reduce the space and frequency sizes using a balance of energy and set-size estimates, see  Step 1 in proof of Propositions \ref{prop:finite-speed-of propagation}, \ref{prop:X-norm-j<m}. In terms of the spatial localizations, there are two main regions where the interactions stay bounded for different reasons.
 \end{pstep}
\begin{pstep}{Finite speed of propagation} We formalize the ``finite speed of propagation" feature of our dispersive problem in Proposition \ref{prop:finite-speed-of propagation}. Waves {essentially} travel at finite speed and so can only affect the solution in certain spatial regions. In particular, bilinear terms where the spatial localizations {of the output satisfies} $|x|\gtrsim |t|$ 
{are acceptably small}. 
{In practice, in Proposition \ref{prop:finite-speed-of propagation} this is established by} taking advantage of oscillations of the phase in the {output} frequency $\xi$ via an iterative integration by parts argument, see e.g.\ \eqref{fin-speed-ibp}.
\end{pstep}
\begin{pstep}{Structure of the nonlinearity and normal forms} For the remaining bilinear terms (where $|x|\lesssim |t|$), the key step is to take advantage of oscillations of the phase via a normal form \cite{Shatah1985NormalFA} (see Proposition \ref{prop:X-norm-j<m}): Integrating by parts in time in the integral above, we obtain one boundary term and two cubic {order} terms:
\begin{equation}
\mathcal{F}\left(\mathcal{B}_\m(f,g)\right)(t,\xi)= \int_{\R^3} e^{it\Phi(\xi,\eta)}\m\Phi^{-1}\widehat{f}(\xi-\eta)\widehat{g}(\eta)d\eta + \mathcal{F}\left(\mathcal{B}_{\m\Phi^{-1}}(\partial_tf,g)\right)(t,\xi)+  \mathcal{F}\left(\mathcal{B}_{\m\Phi^{-1}}(f,\partial_tg)\right)(t,\xi).
\end{equation}
Whenever $\m\Phi^{-1}$ satisfies suitable bounds, this is a powerful argument since in practice the time derivatives of $f,g$ decay comparatively rapidly in $L^2$. More precisely, in Lemma \ref{lem:time-derivative-L2} we show that $\norm{\partial_t\mathcal{Z}_\pm}_{L^2}\lesssim t^{-\frac{3}{2}}$, which improves over trivial rate $t^{-(1+\beta)}$. At the core of these arguments thus are the bounds on $\m\Phi^{-1}$, which we quantify in terms of the sizes of the involved frequencies in Lemmas \ref{LM:bound-from-below-PHASE}, \ref{lem:multiplier-bounds}. As it turns out, while the multiplier does not exhibit a full null-structure in the classical sense of e.g.\ \cite{K86-Null-Condition, PS13-STR-Waves}, it does have favourable cancellations that can offset some of the losses\footnote{in fact, $\Phi$ can vanish on a relatively large set of ``time resonances'' given by
\begin{align}\label{time-resonance-set}
   \mathcal{T}^{\mu \nu} \vcentcolon=\lbrace (\xi, \eta) \in \R^3 \times \R^3 \mid \Phi^{\mu\nu}(\xi,\eta)=0 \rbrace = \left\{ \begin{array}{ll}
       \lbrace (0,0) \rbrace  & \text{if}  \ \ (\mu, \nu)=(+,+),  \\
         \lbrace \xi=0\rbrace \cup \lbrace \xi=\eta \rbrace  & \text{if} \ \ (\mu, \nu)=(+,-), \\
         \lbrace \xi=0\rbrace \cup \lbrace \eta =0 \rbrace  & \text{if} \ \  (\mu, \nu)=(-,+), \\
         \lbrace \eta=0\rbrace \cup \lbrace \xi =\eta \rbrace  & \text{if}  \ \ (\mu, \nu)=(-,-).
    \end{array}
    \right\}.
\end{align}
}  of $\Phi^{-1}$, see also Remark \ref{rem:null-structure}. A remaining difficulty is a loss of regularity $|\m\Phi^{-1}|\lesssim |\xi|^{-1}$
when the output frequency $\xi$ is the smallest and satisfies $|\xi|\ll 1$. To overcome this, one takes advantage of the refined decay bound \eqref{eqn:linear-decay-j<m}, as well as the frequency weights embedded in the $X$-norm. We highlight that normal forms are a key tool also in the Euler-Poisson system with pressure, both for electrons \cite{Guo-electrons3d} and ions \cite{Guo-Pausader-ions3d}. The latter case shares many parallels with the present setting, in particular in terms of the form of the dispersion relation at low frequencies.
\end{pstep}

\subsection{Plan of the Article} In \textbf{Section \ref{sec:set-up}} we detail the dispersive formulation and the localization projections. \textbf{Section \ref{sec:Main-Result-Dispersive}} contains a rigorous statement and proof of Theorem \ref{main-thm} in the dispersive framework, including the main bootstrap argument Proposition \ref{prop:bootstrap}. \textbf{Section \ref{sec:linear-decay}} is devoted to the linear decay estimates for the semigroup of the problem, see Proposition \ref{prop:lin-decay}. The energy estimates are performed in \textbf{Section \ref{sec:energy-estimates}}. The main dispersive analysis is conducted in \textbf{Section \ref{sec:X-norm}}, containing the main Propositions \ref{prop:finite-speed-of propagation}, \ref{prop:X-norm-j<m} as well as the decay estimates for the time-derivative of profiles in $L^2$, see Lemma \ref{lem:time-derivative-L2}. \textbf{Section \ref{sec:auxiliary-results}} contains the main phase and multiplier bounds and some useful tools. Finally, \textbf{Appendix} \ref{appendix} presents the arguments to treat the full exponential in the Poisson equation.
\section{Set-up and Main Result}\label{sec:set-up}

\subsection{Dispersive Formulation} \label{sec:dispersive-formulation}
In this section we want to establish the dispersive structure and the ``good unknowns" described in Section \ref{sec:outline-of-proof}. 
One of the main assumptions in Theorem \ref{main-thm} is the irrotationality of the velocity field. This has been discussed in the comments following the theorem. Here we make use of this by rewriting the system \eqref{eq:EP-perturbed} on $(v,\rho)$ in terms of $\rho$ and the divergence part of $v$. If $\nabla\times v=0$, then there exists formally a stream function $\psi(t,x)$ such that $v=\nabla \psi$. The perturbed system \eqref{eq:EP-perturbed} can now be rewritten as
\begin{equation}\label{eq:EP-perturbed-vorticity}
\left\{
      \begin{aligned}
      \partial_t \rho + \Delta \psi+\mathrm{div}(\rho \nabla \psi) &=0, \\[1mm]
\partial_t  \psi +\frac{1}{2} \vert \nabla \psi \vert^2 &=- \Phi,  \\[1mm]
\Phi-\Delta \Phi&= \rho.
\end{aligned}
    \right.
\end{equation}
In the above we have used the general formula $(U \cdot \nabla )U=\frac{1}{2}\nabla \vert U \vert^2-U \times (\nabla\times U)$ that holds for any regular vector field $U(x)$, yielding that $(v \cdot \nabla )v=\frac{1}{2}\nabla \vert \nabla \psi \vert^2$. 
\par We establish the dispersive structure of \eqref{eq:EP-perturbed-vorticity} (and thus of \eqref{eq:EP-perturbed}) by studying its linearization. 
\begin{lem}\label{lem:linearization}
The linearization of \eqref{eq:EP-perturbed-vorticity} (and thus also  \eqref{eq:EP-perturbed}) is equivalent to 
\begin{align}
\partial_tZ_{\pm}=\pm i\Lambda(\abs{\nabla}) Z_\pm,
\end{align}
where the oscillating variables are given by
\begin{equation}\label{eq:def-Unknown}
\begin{split} Z_\pm&\coloneqq\abs{\nabla}^{-1}\rho\pm i \Lambda^{-1}\abs{\nabla}^{-1}\mathrm{div}(v)\\
 &=\abs{\nabla}^{-1}\rho\mp i \br{\nabla}\psi,
 \end{split}
 \end{equation}
and the dispersive operator $\Lambda$ has the symbol given by:
\begin{align}\label{def:dispersionrel}
    \Lambda(\vert \xi \vert )\coloneqq \frac{\vert \xi \vert}{\sqrt{1+ \vert \xi \vert^2}}=\frac{\abs{\xi}}{\langle \xi\rangle}.
\end{align}
\end{lem}
\begin{proof} The linearization of the perturbed system \eqref{eq:EP-perturbed-vorticity} reads:
\begin{equation}\label{eq:EP-linearized}
\left\{
      \begin{aligned}
      &\partial_t \rho + \Delta\psi=0, \\[1mm]
&\partial_t \psi =- \Phi,  \\[1mm]
&\Phi-\Delta \Phi= \rho,
\end{aligned}
    \right. \ \ \ \ 
\end{equation}
In particular, in the frequency space we obtain
\begin{equation}\label{eq:EP-linearized-Fourier}
\left\{
      \begin{aligned}
      &\partial_t \widehat{\rho}(\xi) = \abs{\xi}^2\widehat{\psi}(\xi), \\[1mm]
&\partial_t \widehat{\psi}(\xi) =-(1+\abs{\xi}^2)^{-1}\widehat{\rho}(\xi). \\[1mm]
\end{aligned}
    \right. \ \ \ \
\end{equation}
By multiplying the first equation with $\abs{\xi}^{-1}$ and the second one with $i\br{\xi}$, we can see that $\widehat{Z_\pm}$ in fact diagonalizes this system with dispersion relation given by $\Lambda(\xi)$, which proves the lemma.\end{proof}

\begin{rem}\label{rem:dispersive-unknowns}
\begin{enumerate}
\item One readily observes that {at high frequency} the unknowns $Z_\pm$ carry one more derivative on the velocity $v$ (two on the stream function $\psi)$ compared to the density $\rho$. As a matter of fact, in the absence of pressure term, we cannot treat the equation as an hyperbolic system in $(\rho, v)$, hence enforcing a shift of regularity between the two. 
\item Since $v$ is irrotational, we can express the original unknowns $(v,\psi,\rho)$ in terms of $ Z_\pm$ as follows 
\begin{align}\label{eqn:rho,psi-in-Z}
\rho=\abs{\nabla}\left(\frac{Z_++Z_-}{2}\right), \quad v=i\nabla\langle \nabla \rangle^{-1}\left(\frac{Z_+-Z_-}{2}\right), \quad \Delta \psi=i\br{\nabla}^{-1}\Delta\left(\frac{Z_+-Z_-}{2}\right).
\end{align}
\item We note that $Z_\pm$ are not unique in diagonalizing the linearized system of \eqref{eq:EP-perturbed-vorticity}, but this choice of unknowns allows for favourable $H^s$ estimates, see Section \ref{sec:energy-estimates}. In particular, good commutator estimates, which are crucial in obtaining standard bounds on the energy, are available.
\end{enumerate}
\end{rem}

Since the unknowns $Z_\pm$ are the ones oscillating on the linear level, it is natural to rewrite the nonlinear system using them. This allows us to define in \eqref{def:bilinear-expressions} the main bilinear objects on which the dispersive analysis will be conducted in later sections.
\begin{prop}\label{prop-Duhamel}
    Let the dispersive unknowns $Z_\pm$ and the dispersion relation $\Lambda$ be defined as in \eqref{eq:def-Unknown}-\eqref{def:dispersionrel}.  Let the profiles of the unknowns be given by:
    \begin{align}
        \widehat{\mathcal{Z}_\pm}(\xi)\vcentcolon= e^{\mp it\Lambda(\xi)}\widehat{Z_\pm}.
    \end{align}
Then the following statements hold true:
    \begin{enumerate}
        \item  The system \eqref{eq:EP-perturbed-vorticity} is equivalent to
       \begin{equation}\label{eq:Syst-dispersive}
   \begin{split}
       \partial_tZ_\pm&+\frac{i}{4}\abs{\nabla}^{-1}\mathrm{div}\Big(\abs{\nabla}(Z_++Z_-)\nabla\langle\nabla\rangle^{-1}(Z_+-Z_-)\Big)\\
       &\hspace{7cm}{\pm} \frac{i}{8}\langle \nabla\rangle\abs{\nabla\langle\nabla\rangle^{-1}(Z_+-Z_-)}^2=\pm i\Lambda Z_\pm.
   \end{split} 
\end{equation}
\item The profiles $\zz_\pm$ satisfy 
\begin{align}
    &\zz_{\pm}(t)=\zz_\pm(0)+\sum_{\mu\in\{+,-\}}\mathcal{B}_{\m_\pm^{\mu\mu}}(\zz_\mu,\zz_\mu)(t)+\mathcal{B}_{\m_\pm^{+-}}(\zz_+,\zz_-), \label{eqn:Duhamel-profiles}
\end{align}
where the quadratic nonlinearity is given for $\mu,\nu\in\{+,-\}$ with $\mu\geq\nu$, by
\begin{equation}\label{def:bilinear-expressions}
\begin{split}
    \mathcal{B}_{\m_\pm^{\mu\nu}}(\zz_\mu,\zz_\nu)(t)&\vcentcolon =\int_0^t \mathcal{Q}_{\m_\pm^{\mu\nu}}(\zz_\mu,\zz_\nu)(s)ds,\\
    \mathcal{F}(\mathcal{Q}_{\m_\pm^{\mu\nu}}(\zz_\mu,\zz_\nu)(s,\xi)&\vcentcolon =\int_{\R^3}e^{is\Phi_\pm^{\mu\nu}(\xi,\eta)}\m_\pm^{\mu\nu}(\xi,\eta)\widehat{\zz_\mu}(s,\xi-\eta)\widehat{\zz_\nu}(s,\eta)d\eta.
\end{split}
\end{equation}
The phases read
\begin{align}\label{def:phases}
    \Phi_\pm^{\mu\nu}(\xi,\eta)\vcentcolon=\mp\Lambda(\xi)+\mu\Lambda(\xi-\eta)+\nu\Lambda(\eta),
\end{align}
and the multipliers are given by
\begin{align}\label{def:multipliers}
\begin{split}
        \m_\pm^{\mu\mu}(\xi,\eta)&\coloneqq {\mu\frac{i}{8}}\left(\frac{\xi\cdot\eta}{\abs{\xi}\langle\eta\rangle}\abs{\xi-\eta}+\frac{\xi\cdot(\xi-\eta)}{\abs{\xi}\langle\xi-\eta\rangle}\abs{\eta}\right){\mp} \frac{i}{8}\langle\xi\rangle\frac{(\xi-\eta)\cdot\eta}{\langle{\xi-\eta}\rangle\langle{\eta}\rangle},\\
    \m_\pm^{+-}(\xi,\eta)&\coloneqq-{\frac{i}{4}}\left(\frac{\xi\cdot\eta}{\abs{\xi}\langle\eta\rangle}\abs{\xi-\eta}-\frac{\xi\cdot(\xi-\eta)}{\abs{\xi}\langle\xi-\eta\rangle}\abs{\eta}\right)\pm \frac{i}{4}\langle\xi\rangle\frac{(\xi-\eta)\cdot\eta}{\langle{\xi-\eta}\rangle\langle{\eta}\rangle}.
    \end{split}
\end{align}
\end{enumerate}
\end{prop}
\begin{proof} The system \eqref{eq:Syst-dispersive} follows by a direct computation from \eqref{eq:EP-perturbed-vorticity}. Recall 
\begin{align}
    &Z_\pm=\abs{\nabla}^{-1}\rho\pm i\Lambda^{-1}\abs{\nabla}^{-1}\Delta \psi=\abs{\nabla}^{-1}\rho\mp i\br{\nabla}\psi, \quad \rho=\frac{\abs{\nabla}}{2}(Z_++Z_-),\\
    &\psi=\frac{i}{2}\br{\nabla}^{-1}(Z_+-Z_-), \quad \Delta\psi=-\frac{i}{2}\Lambda\abs{\nabla}(Z_+-Z_-).
\end{align}
Now we compute
\begin{align}
    \partial_t\abs{\nabla}^{-1}\rho&=-\abs{\nabla}^{-1}\mathrm{div}(\rho\nabla\psi)-\abs{\nabla}^{-1}\Delta\psi\\
    &=-\frac{i}{4}\abs{\nabla}^{-1}\mathrm{div}\Big(\abs{\nabla}(Z_++Z_-)\nabla\br{\nabla}^{-1}(Z_+-Z_-)\Big)+\frac{i}{2}\abs{\nabla}\br{\nabla}^{-1}(Z_+-Z_-).
\end{align}
Similarly, we have
\begin{align}
    \partial_t i\br{\nabla}\psi&=-\frac{i}{2}\br{\nabla}\abs{\nabla\psi}^2-i\br{\nabla}^{-1}\rho
    \\&
    =-\frac{i}{8}\br{\nabla}\abs{i\nabla\br{\nabla}^{-1}(Z_+-Z_-) }^2-\frac{i}{2}\br{\nabla}^{-2}\br{\nabla}\abs{\nabla}(Z_++Z_-)\\
    &=\frac{i}{8}\br{\nabla}\abs{\nabla\br{\nabla}^{-1}(Z_+-Z_-) }^2-\frac{i}{2}\Lambda(Z_++Z_-).
\end{align}
We prove the second statement.
    Using the definition of the profiles $\zz_\pm=e^{\mp it\Lambda}Z_\pm$ and the equations \eqref{eq:Syst-dispersive} on the dispersive unknowns we may write in the frequency space
    \begin{equation}\label{eqn:profiles-eqn-rough}
    \begin{split}
        \partial_t\widehat{\zz_\pm}+&\frac{i}{4}i\xi\cdot\abs{\xi}^{-1}\int_{\R^3}e^{\mp it\Lambda(\xi)}\abs{\xi-\eta}\left(e^{it\Lambda(\xi-\eta)}\widehat{\zz_+}(\xi-\eta)+e^{-it\Lambda(\xi-\eta)}\widehat{\zz_-}(\xi-\eta)\right)\\
        &\qquad\qquad\qquad \cdot i\eta\langle\eta\rangle^{-1}\left(e^{it\Lambda(\eta)}\widehat{\zz_+}(\eta)-e^{-it\Lambda(\eta)}\widehat{\zz_-}(\eta)\right)d\eta\\
        &\mp\frac{i}{8}\langle\xi\rangle\int_{\R^3}e^{\mp it\Lambda(\xi)}i(\xi-\eta)\langle\xi-\eta\rangle^{-1}\left(e^{it\Lambda(\xi-\eta)}\widehat{\zz_+}(\xi-\eta)-e^{{-}it\Lambda(\xi-\eta)}\widehat{\zz_-}(\xi-\eta)\right)\\
        &\qquad\qquad \qquad \cdot i\eta\langle\eta\rangle^{-1}\left(e^{it\Lambda(\eta)}\widehat{\zz_+}(\eta)-e^{{-}it\Lambda(\eta)}\widehat{\zz_-}(\eta)\right)d\eta=0.
        \end{split}
    \end{equation}
    In particular, we can directly read the multipliers $\m_\pm^{\mu\mu}$ for $\mu\in\{+,-\}$. To see the multipliers in case of the mixed convolutions, it suffices to do a change of variables $\xi-\eta \leftrightarrow\eta$ in the integrals containing $\widehat{\zz_-}(\xi-\eta)\widehat{\zz_+}(\eta)$ and observe that the phase stays the same due to the change of sign in the semigroup. Indeed, by a change of variable in the second integral below and the fact that $\Phi^{+-}_\pm(\xi,\eta)=\Phi^{-+}_\pm(\xi,\xi-\eta)$ we obtain for the first integral in \eqref{eqn:profiles-eqn-rough}:
    \begin{align}
    &-\int_{\R^3}e^{\mp it\Lambda(\xi)+it\Lambda(\xi-\eta)-it\Lambda(\eta)} \abs{\xi-\eta}\eta\br{\eta}^{-1}\widehat{\zz_+}(\xi-\eta)\widehat{\zz_-}(\eta)d\eta\\
       &\qquad \qquad\qquad + \int_{\R^3}e^{\mp it\Lambda(\xi)-it\Lambda(\xi-\eta)+it\Lambda(\eta)} \abs{\xi-\eta}\eta\br{\eta}^{-1}\widehat{\zz_-}(\xi-\eta)\widehat{\zz_+}(\eta)d\eta\\
       &=-\int_{\R^3}e^{it\Phi_\pm^{+-}(\xi,\eta)}\eta\frac{\abs{\xi-\eta}}{\br{\eta}}\widehat{\zz_+}(\xi-\eta)\widehat{\zz_-}(\eta)d\eta+ \int_{\R^3}e^{it\Phi_\pm^{-+}(\xi,\xi-\eta)} (\xi-\eta)\frac{\abs{\eta}}{\br{\xi-\eta}}\widehat{\zz_-}(\eta)\widehat{\zz_+}(\xi-\eta)d\eta\\
       &=\int_{\R^3}e^{it\Phi_\pm^{+-}(\xi,\eta)}\left(-\eta\frac{\abs{\xi-\eta}}{\br{\eta}}+(\xi-\eta)\frac{\abs{\eta}}{\br{\xi-\eta}}\right)\widehat{\zz_+}(\xi-\eta)\widehat{\zz_-}(\eta)d\eta.
    \end{align}The remaining integrals in \eqref{eqn:profiles-eqn-rough} containing mixed terms follow analogously and this finishes the proof.
\end{proof}

\begin{rem}\label{rem:null-structure}
In view of the time resonances \eqref{time-resonance-set}, we see that the multipliers in \eqref{def:multipliers} vanish at first order when $\xi=\eta$ or $\eta=0$. This is employed in obtaining a faster decay of the time derivative {of the dispersive profiles in $L^2$} in Lemma \ref{lem:time-derivative-L2} and in the normal form transformation in Proposition \ref{prop:X-norm-j<m}.
\end{rem}

\subsection{Localizations}\label{sec:localizations}
 Let $\psi: \R \to [0,1]$ be an even, smooth function with $\supp\psi\subset [-8/5,8/5]$ and $\psi\vert_{[-5/4,5/4]}\equiv1$.
With a slight abuse of notation we also let $\psi$ be the corresponding radial function on $\R^3$.
For $k\in\Z$ we define $\varphi_k(x) \vcentcolon= \psi(2^{-k}|x|) - \psi(2^{-k+1}|x|)$, so that the family $(\varphi_k)_{k\in\Z}$
forms a partition of unity:
\begin{equation*}
 \sum_{k\in\Z}\varphi_k(\xi)=1, \quad \xi \neq 0.
\end{equation*}
We also let
\begin{align*}
\varphi_{I}(x) \vcentcolon= \sum_{k \in I \cap \Z}\varphi_k, \quad \text{for any} \quad I \subset \R, \qquad
\varphi_{\leq a}(x) \vcentcolon= \varphi_{(-\infty,a]}(x), \qquad \varphi_{> a}(x) \vcentcolon= \varphi_{(a,\infty]}(x),
\end{align*}
with similar definitions for $\varphi_{< a},\varphi_{\geq a}$.
Using these, we may define frequency projections $P_k$ through
\begin{equation*}
 P_k g\vcentcolon=\mathcal{F}^{-1}\left(\varphi_k(\xi)\widehat{g}(\xi)\right)
\end{equation*}
and similarly $P_{I}g\vcentcolon=\mathcal{F}^{-1}\left(\varphi_{I}(\xi)\widehat{g}(\xi)\right)$,
$P_{\leq k}g\vcentcolon=\mathcal{F}^{-1}\left(\varphi_{\leq k}(\xi) \widehat{g}(\xi)\right)$, $k\in\Z$ etc.
To simultaneously localize in space, for $(k,j)\in\mathcal{J}\vcentcolon=\{(k,j)\in \Z \times \Z: \; k+j\geq 0,\;j\geq 0\}$ define
\begin{equation}
\varphi_j^{(k)}(x)\vcentcolon=
\begin{cases}
 \varphi_j(x), \; & j \geq -k+1, \text{ or }j\geq 1,
 \\ \varphi_{\leq 0}(x), \; & j = 0, \quad (k \geq 0)
 \\ \varphi_{\leq -k}(x), \; & j = -k, \quad (k \leq 0).
\end{cases}
\end{equation}
Notice that for any $k\in\Z$ we have $\sum_{j \geq -\min\{0,k\}} \varphi_j^{(k)}(x) = 1$. Then we have the following decomposition
\begin{equation}\label{Q_jk-def}
f = \sum_{(k,j)\in\mathcal{J}} Q_{jk} f, \qquad Q_{jk} f \vcentcolon= P_{[k-2,k+2]}\varphi_j^{(k)} P_k f.
\end{equation}
In later sections, we will use the localization projections simultaneously for the variables $\xi, \, \xi-\eta$ and $\eta$, and thus introduce the  short-hand notation and fix the convention for $k,k_1,k_2\in \Z$
\begin{align}
&\chi(\xi,\eta)=\varphi_{k}(\xi)\varphi_{k_1}(\xi-\eta)\varphi_{k_2}(\eta) \label{eqn: chi localizations},\\
    &\abs{\xi}\sim 2^k,\quad \abs{\xi-\eta}\sim 2^{k_1}, \quad \abs{\eta}\sim 2^{k_2}. \label{def:frequencies-k}
\end{align}

\subsection{\texorpdfstring{$X$}{X}-norm and Main Result}\label{sec:Main-Result-Dispersive} The main norm in our result is defined using the atomic decomposition introduced above. For $k\in\Z$, let
   $k^+\vcentcolon=\max\{0,k\} ,\,  k^-\vcentcolon=\min\{0,k\}$, then we define the $X$-norm by
\begin{align}\label{def:X-norm}
    &\norm{f}_X\vcentcolon=\sup_{(k,j)\in \mathcal{J}}2^{8k^+}2^{(1+\beta)(k+j)}2^{\gamma k^-}\norm{Q_{jk}f}_{L^2}.
\end{align}
 This norm guarantees the decay of profiles in $L^\infty$, see Proposition \ref{prop:lin-decay}. Throughout the paper, we fix parameters $\beta$ and $\gamma$ satisfying (see also Remark \ref{rem:parameters})
\begin{align}\label{eqn:gamma-beta-relation}
  \beta \in \left]0, \frac{1}{5}\right[,\qquad  \gamma=-\frac{1}{6}-\frac{5}{3}\beta, \qquad \gamma+\beta=-\frac{1}{6}-\frac{2}{3}\beta, \qquad \gamma\in \left]-\frac{1}{2},0\right[.
\end{align}

Coming back to the original formulation of the system, Theorem \ref{main-thm} will essentially be a consequence of the following theorem on the dispersive unknowns from Proposition \ref{prop-Duhamel}.
\begin{thm}\label{main-thm-dispersive}
 Let $Z_\pm$ solve \eqref{eq:Syst-dispersive} with initial data $Z_{\pm,0}$. Then there exist $N_0>0$, $\eps_0>0$ such that for any $\eps \in ]0, \eps_0[$, if $Z_{\pm,0}$ satisfies
   \begin{align}\label{assumption-smallnessdata}
\norm{Z_{\pm,0}}_{H^{N_0}}+\norm{Z_{\pm,0}}_{X} < \eps,
   \end{align}
   then there exists a unique global in time solution $Z_\pm \in \mathscr{C}\left(\R; H^{N_0}\right)$ to the system \eqref{eq:Syst-dispersive} with initial data $Z_{\pm,0}$. Furthermore, there holds 
   \begin{align}
       \label{csqTHM-boundZ}\norm{Z_\pm(t)}_{H^{N_0}} + \norm{ e^{\mp it \Lambda}Z_\pm(t)}_{X} &\lesssim \eps_0, \ \ t>0, \\
       \label{csqTHM-decayZ}\norm{\vert \nabla \vert Z_\pm(t)}_{L^\infty}+\norm{ \nabla  \langle \nabla \rangle^{-1}Z_\pm(t)}_{L^\infty} &\lesssim (1+t )^{-(1+\beta)}\eps_0, \ \ t>0.
   \end{align}
  
   Moreover, we have linear scattering: there exists $\zz_\pm^{\infty} \in \R$ satisfying  $\norm{\zz_\pm^{\infty}}_{L^2 \cap X}<\infty$ such that 
   \begin{align*}
       \norm{ e^{\mp it \Lambda}Z_\pm(t)-\zz_\pm^{\infty}}_{L^2 \cap X} \underset{t \rightarrow +\infty}{\longrightarrow} 0.
   \end{align*}
\end{thm}

The proof of Theorem \ref{main-thm-dispersive} will be achieved thanks to a standard bootstrap argument, performed on the Duhamel formulation from Proposition \ref{prop-Duhamel}.

\begin{prop}\label{prop:bootstrap}
    Let $Z_\pm \in \mathscr{C}\left([0,T],H^{N_0} \right)$ be solutions to \eqref{eq:Syst-dispersive} for some time $T>0$ with initial data satisfying \eqref{assumption-smallnessdata} and $\mathcal{Z}_\pm\vcentcolon=e^{\mp it\Lambda}Z_\pm$ the corresponding profiles. Assume that for $t\in [0,T]$ there holds:
    \begin{align}
\label{eq:bootstrap-assumptionX}\norm{ \zz_+(t)}_X +\norm{ \zz_-(t)}_{X} &\leq 100\eps, 
\end{align}
then for all $t \in [0,T]$ the following statements hold.
\begin{enumerate}
    \item As a direct consequence, the following decay estimate holds
    \begin{align}
          \label{eqn:bootstrap-decayLinfty} \norm{\vert \nabla \vert Z_\pm(t)}_{L^\infty}+\norm{ \nabla  \langle \nabla \rangle^{-1}Z_\pm(t)}_{L^\infty}&\lesssim \br{t}^{-(1+\beta)}\eps,
    \end{align}
    and moreover the energy remains bounded
    \begin{align}
          \label{eqn:boostrap-energy-bounded}
\norm{{Z}_\pm(t)}_{H^{N_0}}=\left\Vert\mathcal{Z}_\pm (t)\right\Vert_{H^{N_0}}&\lesssim \eps.
    \end{align}
    \item The improved bound holds:
    \begin{align}
\label{eqn:bootstrap-improved-bound-X}
   \left\Vert \zz_+(t) \right\Vert_{X} +\left\Vert \zz_-(t) \right\Vert_{X} &\leq 10\eps.
\end{align}
\end{enumerate}

\end{prop}
\begin{proof}
    The proof follows as a consequence of the arguments presented throughout the paper, and that we invoke here. 
    The decay estimate \eqref{eqn:bootstrap-decayLinfty} and the uniform bounds \eqref{eqn:boostrap-energy-bounded} on the energy follow directly from Corollary \ref{cor:blow-up-criterion-decay}.

    The main difficulty of this work is showing the improved estimate \eqref{eqn:bootstrap-improved-bound-X}. To that end, we consider the Duhamel formulation for the profiles $\zz_\pm$ \eqref{eqn:Duhamel-profiles} and obtain
    \begin{equation}\label{bootstrap-z-in-proof}
    \begin{split}
         \norm{\mathcal{Z}_+(t)}_{X}+ \norm{\mathcal{Z}_-(t)}_{X}&\leq\norm{\mathcal{Z}_+(0)}_X+\sum_{\mu\in\{+,-\}}\norm{\mathcal{B}_{\m_+^{\mu\mu}}(\zz_\mu,\zz_\mu)(t)}_X+\norm{\mathcal{B}_{\m_+^{+-}}(\zz_+,\zz_-)(t)}_X\\
        &\quad+\norm{\mathcal{Z}_-(0)}_X+\sum_{\mu\in\{+,-\}}\norm{\mathcal{B}_{\m_-^{\mu\mu}}(\zz_\mu,\zz_\mu)(t)}_X+\norm{\mathcal{B}_{\m_-^{+-}}(\zz_+,\zz_-)(t)}_X.
    \end{split}
    \end{equation}
    Therefore, since the initial data satisfies \eqref{assumption-smallnessdata}, it suffices to prove that for $\m\in \{\m_\pm^{\mu\nu} \; \vert\; \mu,\nu\in \{+,-\}\}$ and the corresponding profiles the bilinear terms satisfy:
    \begin{align}\label{eqn:bootstrap-bilin<3eps}
       \norm{\mathcal{B}_{\m}(\zz_\mu,\zz_\nu)(t)}_X\leq 3\eps.
    \end{align}
    To that end, we localize the time variable  $t\in[0,T]$ as follows: we decompose the indicator function  $\mathds{1}_{[0,t]}$ in functions $\tau_0,...\tau_{L+1}:\R\to [0,1]$ with $\abs{L-\log_2(2+t)}\leq 2$ such that
    \begin{align*}
      &  \supp\tau_0\subset [0,2],\hspace{0.5cm} \supp \tau_m \subset [2^{m-1},2^{m+1}], \; m\in \{1,..., L\}, \hspace{0.5cm} \supp \tau_{L+1}\subset [t-2,t],\\
     &   \sum_{m=0}^{L+1}\tau_m(s)=\mathds{1}_{[0,t]}, \hspace{0.5cm}\tau_m(s)\in \mathscr{C}^1(\R), \hspace{0.5cm}\int_0^t\abs{\tau_m'(s)}ds\lesssim 1, \; m\in \{1,..., L\}.
    \end{align*}
    Then for a bilinear expression as in \eqref{def:bilinear-expressions} there holds
    \begin{equation}\label{eqn: bootstrap B_m time decomposition}
    \begin{split}
        \mathcal{B}_{\m}(f,g)(t)&=\int_0^t\mathcal{Q}_{\m}(f,g)(s)ds=\sum_{m\geq 1}\int_0^t\tau_m(s)\mathcal{Q}_{\m}(f,g)(s)ds=\sum_{m\geq 1}\mathcal{B}_{\m}^m(f,g)(t),\\
\mathcal{B}_{\m}^m(f,g)(t)&\vcentcolon=\int_0^t\tau_m(s)\mathcal{Q}_{\m}(f,g)(s)ds.
    \end{split}
    \end{equation}
    In the main dispersive analysis Proposition \ref{prop:X-norm-bounds} we prove
    \begin{align}
         \norm{\mathcal{B}_{\m}(\zz_\mu,\zz_\nu)(t)}_X\leq \sum_{m\geq 1}\norm{\mathcal{B}_{\m}^m(\zz_\mu,\zz_\nu)}_{X}\lesssim \sum_{m\geq 1}2^{-\frac{\delta}{2}m}\eps^2\leq C_1\eps^2,
    \end{align}
    where $\delta>0$ as in Remark \ref{rem:parameters}. Finally, we obtain \eqref{eqn:bootstrap-bilin<3eps} and hence \eqref{eqn:bootstrap-improved-bound-X} by taking $\eps\leq 3C_1^{-1}$. This finishes the proof of the bootstrap argument.
\end{proof}
\begin{cor}[Scattering]\label{coro-scattering}
    Under the same assumption as in Proposition \ref{prop:bootstrap}, the following holds: there exists $\zz_\pm^{\infty} \in X \cap L^2$ such that 
   \begin{align*}
       \norm{ \zz_\pm(t)-\zz_\pm^{\infty}}_{L^2 \cap X} \underset{t \rightarrow +\infty}{\longrightarrow} 0.
   \end{align*}
\end{cor}
\begin{proof}
The argument is standard but we briefly detail it for the sake of completeness. First, let us show that $\zz_\pm$ is Cauchy (in the time variable) in the space $L^2$. This comes from the fact that the time derivative $\partial_t \zz_\pm$ decays fast enough: indeed thanks to Lemma \ref{lem:time-derivative-L2} we know that $\norm{\partial_t \zz_\pm(t)}_{L^2} \lesssim \langle t \rangle^{-3/2}$ for $t$ large enough, which directly implies that  $\zz_\pm(t)$ converges to some $\zz_\pm^{\infty}$ in $L^2$.

Similarly, by the Duhamel formula \eqref{prop-Duhamel} and Proposition \ref{prop:X-norm-bounds} $\zz_\pm(t)$ is Cauchy in the space $X$ and so converges to $\zz_\pm^\infty\in X$.
\end{proof}
\begin{rem}
The statement of Corollary \ref{coro-scattering} can actually be strengthened by obtaining scattering in Sobolev spaces $H^s$ for $0<s<9$. This follows from the estimate \eqref{eq:estimSob-to-X} in Lemma \ref{Lem:compX-weightSob}.
\end{rem}
Let us now show how Proposition \ref{prop:bootstrap} entails Theorem \ref{main-thm-dispersive}.
\begin{proof}[Proof of Theorem \ref{main-thm-dispersive}]
    Let $T_0$ be the maximal time of existence of the unique solution $Z_\pm$ to \eqref{eq:Syst-dispersive} with value in $H^{N_0}$. Such a time is ensured to exist by a standard local well-posedness procedure that we omit here (see the result from \cite{LannesLinaresSaut} on the Euler-Poisson system \eqref{eq:EP-perturbed} and Remark \ref{rem-LWP-disp}) and we furthermore have  $ \mathcal{Z}_\pm \in \mathscr{C}\left([0,1]; H^{N_0} \cap X \right)$ thanks to the initial smallness assumption \eqref{assumption-smallnessdata} on the data.

Let $T^\star \in (0, T_0)$ be the maximal time such that $\mathcal{Z}_\pm \in \mathscr{C}\left([0,1]; H^{N_0} \cap X \right)$  satisfies the bootstrap assumption \eqref{eq:bootstrap-assumptionX}.  By the former continuity in time, we know that $T^\star>0$. Thanks to Proposition \ref{prop:bootstrap}, the bound \eqref{eqn:boostrap-energy-bounded} entails that the $H^{N_0}$-norm of $\mathcal{Z}_\pm$ stays bounded on $[0,T^\star]$ and the control 
\eqref{eqn:bootstrap-improved-bound-X} ensures that we have improved the bootstrap assumption on $[0,T^\star]$. By continuity, this contradicts the maximality of $T^\star$, hence $T^\star=T_0$. Since $\norm{Z_\pm}_{H^{N_0}}=\norm{\mathcal{Z}_\pm}_{H^{N_0}}$ does not blow up on $[0,T_0]$, it ensures that $T_0=+\infty$ and therefore the solution is global in time. 

The bounds and the decay estimates are then a consequence of \eqref{eqn:bootstrap-improved-bound-X}--\eqref{eqn:boostrap-energy-bounded}--\eqref{eqn:bootstrap-decayLinfty}. Finally, the linear scattering in inferred directly from Corollary \ref{coro-scattering}, which ends the proof.
\end{proof}
We now turn to the proof of the main Theorem \ref{main-thm}, as a consequence of Theorem \ref{main-thm-dispersive} on the dispersive unknowns. The subtlety lies in switching between $(v,\rho)$ and the dispersive formulation above. In terms of regularity, the link between the original unknowns $(\rho, v, \psi)$ and the dispersive unknowns $(Z_+, Z_-)$ is the following.  
\begin{lem}\label{LEM-link-unkwowns}
    Let $s>0$. The following identities hold:
    \begin{align*}
\norm{Z_+}_{H^{s+1}}^2+\norm{Z_{-}}_{H^{s+1}}^2&=\frac{1}{2}\norm{{Z_++Z_-}}_{H^{s+1}}^2+\frac{1}{2}\norm{{Z_+-Z_-}}_{H^{s+1}}^2, \\
\norm{\rho}_{H^s\cap\dot H^{-1} }^2 + \norm{v}_{H^{s+1}{\cap \dot H^{-1}}}^2 &\lesssim \norm{Z_+}_{H^{s+1}}^2+ \norm{Z_-}_{H^{s+1}}^2, \\
\norm{Z_+}_{H^{s+1}}^2+ \norm{Z_-}_{H^{s+1}}^2 &\lesssim \norm{\rho}_{H^s\cap\dot H^{-1} }^2 + \norm{v}_{H^{s+1}{\cap \dot H^{-1}}}^2 
    \end{align*}
\end{lem}
\begin{proof}
    The proof of the first identity is straightforward. For the second, it comes the first and second equality in \eqref{eqn:rho,psi-in-Z} that yields the following estimates 
    \begin{align*}
        \norm{\rho}_{H^s\cap\dot H^{-1} }^2 + \norm{v}_{H^{s+1}{\cap \dot H^{-1}}}^2 &=\int  \left(\vert \xi \vert^{-2}+\langle \xi \rangle^{2s}\right)\vert \widehat{\rho}(\xi) \vert^2 \, \mathrm{d}\xi + \int  ({\vert \xi \vert^{-2}}+\langle \xi \rangle^{2(s+1)})\vert \widehat{v}(\xi) \vert^2 \, \mathrm{d}\xi \\
        &\lesssim \int  \frac{\langle \xi \rangle^{2s+2}}{\vert \xi \vert^2} \vert \widehat{\rho}(\xi) \vert^2\mathrm{d}\xi + \int  (\vert {\xi \vert^{-2}}+\langle \xi \rangle^{2(s+1)}) \frac{\vert \xi \vert^2}{\langle \xi \rangle^2}\left(\vert \widehat{Z_+}(\xi) \vert^2+\vert \widehat{Z_-}(\xi) \vert^2 \right)\mathrm{d}\xi 
 \\
 &\lesssim \int  \langle \xi \rangle^{2(s+1)} \left(\vert \widehat{Z_+}(\xi) \vert^2+\vert \widehat{Z_-}(\xi) \vert^2 \right)\mathrm{d}\xi \\
 &=\norm{Z_+}_{H^{s+1}}^2+ \norm{Z_-}_{H^{s+1}}^2.
    \end{align*}
    For the third estimate, by the same bound as before for the density, and because $(Z_+-Z_-)/2 =-i \langle \nabla \rangle \psi$, we get
    \begin{align*}
        \frac{1}{4}\norm{{Z_++Z_-}}_{H^{s+1}}^2+\frac{1}{4}\norm{{Z_+-Z_-}}_{H^{s+1}}^2 \lesssim \norm{\rho}_{H^s\cap\dot H^{-1} }^2 + \norm{\psi}_{H^{s+2}}^2.
    \end{align*}
Since $\nabla \psi=v$, we also have $\Delta \psi=\mathrm{div}(v)$ hence the bound
\begin{align*}
    \norm{\psi}_{H^{s+2}}^2=\left( \int_{\vert \xi \vert \leq 1}+\int_{\vert \xi \vert \geq 1} \right)  \langle \xi \rangle^{2(s+2)}\vert \widehat{\psi}(\xi) \vert^2 \, \mathrm{d}\xi &\leq \int_{\vert \xi \vert \leq 1}  \vert \widehat{\psi}(\xi) \vert^2 \, \mathrm{d}\xi
    +\int_{\vert \xi \vert \geq 1}  \langle \xi \rangle^{2(s+2)} \frac{\vert \xi \vert^2}{\vert \xi \vert^4}\vert \widehat{v}(\xi) \vert^2 \, \mathrm{d}\xi \\
    & \lesssim  \int \vert \widehat{\psi}(\xi) \vert^2 \, \mathrm{d}\xi
    +\int  \langle \xi \rangle^{2(s+1)}\vert \widehat{v}(\xi) \vert^2 \, \mathrm{d}\xi \\
    &=\norm{v}_{H^{s+1}{\cap \dot H^{-1}}}^2.
\end{align*}
This yields the claimed estimate, and concludes the proof.
\end{proof}
We can now proceed with the proof of Theorem \ref{main-thm}.
\begin{proof}[Proof of Theorem \ref{main-thm}]
In what follows, we consider $s=N_0-1$, where $N_0$ is the regularity index given by Theorem \ref{main-thm-dispersive}. Since $(\rho_0, v_0) \in H^{N_0-1}\cap \dot{H}^{-1} \times  H^{N_0}$, with $v_0=\nabla \psi_0$ where $\psi_0 \in H^{N_0+1}$, we deduce from Lemma \ref{LEM-link-unkwowns} that $Z_\pm \in H^{N_0}$. Furthermore, as a consequence of the smallness  assumption \eqref{eq:assump-datasmallORIGINAL} on $(\rho_0, v_0, \psi_0)$, the smallness assumption \eqref{assumption-smallnessdata} on $Z_\pm$ is satisfied at $t=0$.

We can therefore apply Theorem \ref{main-thm-dispersive}, which entails that the solution $Z_\pm$ is global in time, satisfy \eqref{csqTHM-boundZ}--\eqref{csqTHM-decayZ}, and furthermore scatters linearly in the sense of Corollary \ref{coro-scattering}. From this, again appealing to Lemma \ref{LEM-link-unkwowns}, we obtain that $(\rho(t), v(t))$ is global in time, and that the desired bound and estimates \eqref{csqTHM-bound}--\eqref{csqTHM-decay} hold.

It remains to prove the claim on the linear scattering on $(\rho(t), v(t))$ from Theorem \ref{main-thm}, which will be a consequence of Corollary \ref{coro-scattering}. We want to prove that there exists $(\rho_0^\infty, v_0^\infty) \in \dot H^{-1} \times L^2$ such that if $(\rho_{\mathrm{lin}}, v_{\mathrm{lin}})$ is the solution to the linearized system
 \begin{equation*}
\left\{
      \begin{aligned}
      \partial_t \rho_{\mathrm{lin}} + \mathrm{div}(v_{\mathrm{lin}}) &=0,\\
\partial_t v_{\mathrm{lin}} &=-\nabla (\mathrm{I}-\Delta)^{-1} \rho_{\mathrm{lin}},\\
(\rho_{\mathrm{lin}}(0), v_{\mathrm{lin}}(0))&=(\rho_0^\infty, v_0^\infty)
\end{aligned}
    \right.
\end{equation*}
then one has
\begin{align}\label{eq:CLAIMscattering}
\norm{\rho(t)-\rho_{\mathrm{lin}}(t)}_{\dot H^{-1}}+ \norm{v(t)-v_{\mathrm{lin}}(t)}_{L^2} \underset{t \rightarrow + \infty}{\longrightarrow} 0.
\end{align}
To do so, let us recall from definition \eqref{eqn:rho,psi-in-Z} that
\begin{align*}
\rho=\abs{\nabla}\left(\frac{Z_++Z_-}{2}\right), \quad v=i\nabla\langle \nabla \rangle^{-1}\left(\frac{Z_+-Z_-}{2}\right),\end{align*}
and that in view of Corollary \eqref{coro-scattering}, 
there exists $\zz_\pm^{\infty} \in X \cap L^2$ such that 
   \begin{align*}
       \norm{ \zz_\pm(t)-\zz_\pm^{\infty}}_{L^2} \underset{t \rightarrow +\infty}{\longrightarrow} 0.
   \end{align*}
This directly implies the convergence
\begin{align*}
    \Vert Z_\pm -e^{ \pm it \Lambda} \zz^\infty_\pm \Vert_{L^2} \rightarrow 0.
\end{align*}
The corresponding linear dynamics are then defined as
\begin{align*}
    \rho_{\mathrm{lin}}(t)\vcentcolon=\abs{\nabla}\left(\frac{e^{ it \Lambda} \zz^\infty_+ + e^{ - it \Lambda} \zz^\infty_-}{2}\right), \quad 
    v_{\mathrm{in}}(t)\vcentcolon=i\nabla\langle \nabla \rangle^{-1}\left(\frac{e^{ it \Lambda} \zz^\infty_+ - e^{ - it \Lambda} \zz^\infty_-}{2}\right).
\end{align*}
By writing the difference, we therefore obtain the claimed convergence \eqref{eq:CLAIMscattering}, and this concludes the proof.
\end{proof}

\begin{rem}\label{rem-LWP-disp}
    Note that at the beginning of the bootstrap Proposition \ref{prop:bootstrap} (needed to prove that Theorem \ref{main-thm-dispersive}) we have implicitly used  that we have a suitable local well-posedness theory for the dispersive unknowns $Z_\pm$ in $H^{N_0}$. This is a consequence of Lemma \ref{LEM-link-unkwowns} and of the standard local well-posedness theory on the original unknowns $(\rho, v) \in \mathscr{C}([0,T]; H^{N_0-1}) \times \mathscr{C}([0,T]; H^{N_0})$ (see e.g.\ \cite{LannesLinaresSaut}).  The propagation of the $\dot H^{-1}$ norm of $\rho$ in \eqref{eq:EP-perturbed} is indeed standard, since it solves a continuity equation. In order to estimate the $L^2$ norm of $\psi$, we can use on the equation that is solves, namely
\begin{align*}
    \partial_t  \psi +\frac{1}{2} v\cdot \nabla \psi  =- \Phi, \quad (\Phi-\Delta)\Phi&= \rho,
\end{align*}
and perform  a direct $L^2$ estimate to obtain
\begin{align*}
    \frac{\mathrm{d}}{\mathrm{d}t} \norm{\psi}_{L^2} \lesssim \norm{\nabla v}_{L^\infty} \norm{\psi}_{L^2}+\norm{\Phi}_{L^2} \lesssim \norm{v}_{H^{N_0}} \norm{\psi}_{L^2}+\norm{\rho}_{H^{N_0-1}},
\end{align*}
 thanks to  Sobolev embedding. Hence, by Grönwall lemma,  we obtain that $\Vert \psi \Vert_{L^\infty(0,T; L^2)}< \infty$ provided that $\psi_0 \in L^2$. 
\end{rem}

\section{Linear Decay}\label{sec:linear-decay}

We prove the available linear decay estimates arising from the dispersive nature of the system. Proposition \ref{prop:lin-decay} gathers the main estimates that will be used in the proof of Proposition \ref{prop:bootstrap} (Proposition \ref{prop:X-norm-bounds} respectively). In particular it shows that the $X$-norm allows us to propagate $|t|^{-(1+\beta)}$ decay, while actually the semigroup enjoys the full $|t|^{-\frac32}$ decay rate according to stationary phase lemma (see Section \ref{sec:outline-of-proof}). 
\begin{prop}\label{prop:lin-decay}
    Let {$f:\R^3\to\R$.}
    \begin{enumerate}
      \item On the dyadic frequency scale $k\in \Z$ there holds for all $t \in \R$: 
      \begin{align}\label{eqn:decay-Pk-profiles}
    \norm{e^{it\Lambda}P_kf}_{L^\infty}&\lesssim |t|^{-(1+\beta)}2^{-(\frac{9}{2}+2\beta)k^+}\norm{f}_X.
       \end{align}
         \item For the space-and-frequency localized pieces \eqref{Q_jk-def} with $(j,k) \in \mathcal{J}$, there holds for all $t \in \R$:
       \begin{align}\label{eqn:linear-decay-j<m}
         \norm{e^{it\Lambda}{Q_{jk}}f}_{L^\infty} &\lesssim |t|^{-\frac{3}{2}}2^{-2k^+}\min\left\{2^{-(1+\beta)k^+}2^{(\frac{1}{2}-\beta)j}2^{(-\frac{1}{3}+\frac{2}{3}\beta)k^-}, 2^{\frac{3}{2}j}2^{(\frac{2}{3}+\frac{5}{3}\beta)k^-}\right\}\norm{f}_{X}.
       \end{align}
    \end{enumerate}   
   
\end{prop}

\begin{proof} We prove the Proposition via a stationary phase argument closely following the approach in \cite[\textcolor{MidnightBlue}{Section 2}]{GPW08}.
    Recalling the notation from Section \ref{sec:localizations}, we obtain using the Young inequality
    \begin{equation}\label{eqn:decay-young}
    \begin{split}    
         \norm{e^{it \Lambda(\vert \nabla \vert)} P_kf}_{L^\infty}
         &=\norm{\mathcal{F}^{-1}\left(e^{i\Lambda(\xi)t}\varphi(2^{-k}\abs{\xi})\hat{f}(\xi)\right)}_{L^\infty}\\
         &=\norm{\mathcal{F}^{-1}\left(e^{i\Lambda(\xi)t}\varphi(2^{-k}\abs{\xi})\right) \ast \mathcal{F}^{-1}\left(\widetilde{\varphi}(2^{-k}\abs{\xi})\hat{f}(\xi)\right)}_{L^\infty}\\
         &\leq  \norm{e^{it\Lambda(\abs{\nabla})}\varphi(2^{-k}\cdot)}_{L^\infty}\norm{\widetilde{P_k}f}_{L^1},
          \end{split}
    \end{equation}
    where $\widetilde{P_k}$ is the projector associated to a bump function $\widetilde{\varphi}$ with $\widetilde{\varphi}\equiv1$ on $\supp(\varphi)$. Thus, we need to control the first term on the right-hand side above. Observe that with a change of variables $\xi \mapsto 2^k\eta$
    \begin{align}\label{eqn:decay-rescaling}
        \int_{\R^3}e^{ix\cdot\xi+i\Lambda(\xi)t}\varphi(2^{-k}\abs{\xi})d\xi= 2^{3k}\int_{\R^3}e^{i2^kx\cdot\eta+i\Lambda_k(\eta)t}\varphi(\abs{\eta})d\eta=2^{3k}(e^{it\Lambda_k(\abs{\nabla})}\mathcal{F}^{-1}(\varphi))(2^kx),
    \end{align}
    where 
    \begin{align}
        \Lambda_k(r)= \frac{r}{\sqrt{2^{-2k}+r^2}}.
    \end{align}is the rescaled radial dispersion relation. 
    We now use the representation of the inverse Fourier transform for radial functions (both $\xi \mapsto \varphi(\xi)$ and $\xi \mapsto \Lambda_k(\vert \xi \vert)$ are) using Bessel functions \cite[\textcolor{MidnightBlue}{Ch. 8.1, eq. (4)}]{Stein-Shakarchi-FuncAna}
  which in $3$D with $\abs{\xi}=r$ yields:
    \begin{align*}
(e^{it\Lambda_k(\abs{\nabla})t}\mathcal{F}^{-1}(\varphi))(x)=
        \frac{1}{2\pi^2}\int_0^\infty e^{it\Lambda_k(r)}\sin(\abs{x}r)\abs{x}^{-1}r \varphi(r) \, \mathrm{d}r=\vcentcolon I(x).
    \end{align*}
    In view of \eqref{eqn:decay-rescaling}, we will show
    \begin{align}\label{lin-decay-main-claim}
        \norm{I}_{L^{\infty}}\lesssim t^{-\frac{3}{2}} 2^{3k^+}2^{-\frac{5}{2}k^-}.
    \end{align}
    This yields
    \begin{align}\label{eqn:decay-rescaling-proof}
\norm{e^{it\Lambda(\abs{\nabla})}\varphi(2^{-k}\cdot)}_{L^\infty} \lesssim t^{-\frac{3}{2}}2^{6k^+}2^{\frac{k^-}{2}}.
    \end{align}
    \par \textbf{Proof of \eqref{lin-decay-main-claim}:} We split the analysis in two cases depending on the size of $\abs{x}$.
    \par\textbf{Case A: $\abs{x}\leq 1$.} In this case, we can integrate by parts once (twice respectively) in $\partial_r$ using
    \begin{align*}
        e^{it\Lambda_k(r)}=\frac{1}{it\Lambda_k'(r)}\partial_r\left( e^{it\Lambda_k(r)}\right).
    \end{align*}
    Before proceeding with the integration by parts we record a few useful bounds used throughout the proof: we have 
    \begin{align}
      &  \frac{1}{\Lambda_k'(r)}=2^{2k} (2^{-2k}+r^2)^{\frac{3}{2}}, \quad \partial_r\left( \frac{1}{\Lambda_k'(r)}\right)=3 \cdot 2^{2k} r (2^{-2k}+r^2)^{\frac{1}{2}}, \quad \partial_r^2\left( \frac{1}{\Lambda_k'(r)}\right)=3 \cdot 2^{2k} \frac{2r^2+2^{-2k}}{(r^2+2^{-2k})^{\frac{1}{2}}},
    \end{align}
    and therefore on $\mathrm{supp}(\varphi)$ where $r \sim 1$, there holds
    \begin{align}  \label{eqn:lin-decay-bounds-ibp}
\abs{ \frac{1}{\Lambda_k'(r)}}\sim 2^{2k^+}2^{-k^-}, 
      \;
      \abs{ \partial_r\left( \frac{1}{\Lambda_k'(r)}\right)}\lesssim  2^{2k^+}2^{k^-},\;
     \; \abs{\partial_r^2\left( \frac{1}{\Lambda_k'(r)}\right)}\lesssim 2^{2k^+}
    \end{align}
    We integrate by parts once in $I$ and obtain:
    \begin{equation}\label{eqn:lin-decay-Case-A-ibp-1}
       \begin{split}
           I=I^{(1)}&=\frac{1}{it}\int_0^\infty e^{it\Lambda_k(r)}\partial_r\left(\frac{1}{\Lambda_k'(r)}\sin(\abs{x}r)\abs{x}^{-1}r \varphi(r)\right) \, \mathrm{d}r\\
        &=-\frac{1}{it}\int_0^\infty e^{it\Lambda_k(r)}  \Bigg(\partial_r\left( \frac{1}{\Lambda_k'(r)}\right) \frac{\sin(\abs{x}r)}{\abs{x}r}r^2 \varphi(r)\\
        &\qquad \qquad +\frac{1}{\Lambda_k'(r)}\cos(\abs{x}r)r\varphi(r)+\frac{1}{\Lambda_k'(r)}\frac{\sin(\abs{x}r)}{\abs{x}r}(r\varphi(r)+r^2\varphi'(r))\Bigg) \, \mathrm{d}r.
       \end{split} 
    \end{equation}
    Since $\varphi$ is a smooth bump function so are $r^2\varphi'(r)$, $r\varphi(r)$ with a slightly modified compact support and using the bounds \eqref{eqn:lin-decay-bounds-ibp} we obtain:
\begin{align}\label{eqn:lin-decay-CaseA-1-ibp}
    \abs{I^{(1)}(x)}\lesssim t^{-1} (2^{2k^+}2^{k^-}+2^{2k^+}2^{-k^-})\lesssim t^{-1} 2^{2k^+}2^{-k^-}.
\end{align}
Integrating by parts once more in  \eqref{eqn:lin-decay-Case-A-ibp-1} yields for $\Tilde{\varphi}(r)=\varphi(r)+r\varphi'(r)$
\begin{equation}\label{eqn:lin-decay-CaseA-ibp2-formula}
    \begin{split}
        I=I^{(2)}&=-\frac{1}{t^2}\int_0^\infty e^{it\Lambda_k(r)}\partial_r\Bigg(\frac{1}{\Lambda_k'(r)}\Bigg[\partial_r\left( \frac{1}{\Lambda_k'(r)}\right) \frac{\sin(\abs{x}r)}{\abs{x}r}r^2 \varphi(r)+\frac{1}{\Lambda_k'(r)}(\cos(\abs{x}r)r\varphi(r)\\
        &\qquad \qquad \qquad \qquad \qquad\qquad \qquad+\frac{1}{\Lambda_k'(r)}\frac{\sin(\abs{x}r)}{\abs{x}}\Tilde{\varphi}(r))\Bigg] \Bigg) \, \mathrm{d}r\\
       &=-\frac{1}{t^2}\int_0^\infty e^{it\Lambda_k(r)}
\Bigg[
\Bigg(
\left(\partial_r\!\left(\frac{1}{\Lambda_k'(r)}\right)\right)^2
+\frac{1}{\Lambda_k'(r)}\partial_r^2\!\left(\frac{1}{\Lambda_k'(r)}\right)
\Bigg)
\frac{\sin(\abs{x}r)}{\abs{x}r}r^2\varphi(r) \\
&\hspace{3.2cm}
+3\,\frac{1}{\Lambda_k'(r)}\partial_r\!\left(\frac{1}{\Lambda_k'(r)}\right)
\Bigg(
\cos(\abs{x}r)\,r\varphi(r)
+\frac{\sin(\abs{x}r)}{\abs{x}}\widetilde{\varphi}(r)
\Bigg) \\
&\hspace{3.2cm}
+\left(\frac{1}{\Lambda_k'(r)}\right)^2
\Bigg(
-\abs{x}\sin(\abs{x}r)\,r\varphi(r)
+2\cos(\abs{x}r)\,\widetilde{\varphi}(r)
+\frac{\sin(\abs{x}r)}{\abs{x}}\widetilde{\varphi}'(r)
\Bigg)
\Bigg]\, \, \mathrm{d}r.
    \end{split}
\end{equation}
Thus, using $\abs{x}\leq 1$ and \eqref{eqn:lin-decay-bounds-ibp}, we obtain
\begin{align}
    \abs{I^{(2)}(x)} &\lesssim t^{-2}(2^{4k^+}2^{2k^-}+ 2^{4k^+}2^{-k^-}+2^{4k^+}+ 2^{4k^+}2^{-2k^-} )
\end{align}
and therefore
\begin{align}\label{eqn:lin-decay-CaseA-2-ibp}    
 \abs{I^{(2)}(x)} \lesssim t^{-2} 2^{4k^+}2^{-2k^-}.
\end{align}
In particular, interpolating between \eqref{eqn:lin-decay-CaseA-1-ibp} and \eqref{eqn:lin-decay-CaseA-2-ibp} we obtain
\begin{align*}
    \abs{I(x)}\lesssim t^{-\frac{3}{2}}2^{3k^+}2^{-\frac{3}{2}k^-},
\end{align*}
which implies the desired claim.
    \par\textbf{Case B: $\abs{x}\geq 1$.}
    In this case, we use $\sin(\vert x \vert r)=\frac{1}{2i}(e^{i\vert x \vert r}-e^{-i\vert x \vert r})$
    on order to split the integral $I$ in two parts as follows:
    \begin{align*}
        I(x)&=\int_0^\infty e^{it\Lambda_k(r)}\sin(\abs{x}r)\abs{x}^{-1}r \varphi(r) \, \mathrm{d}r \\
        &= \int_0^\infty e^{it\Lambda_k(r)+ir\abs{x}} h(x,r) \, \mathrm{d}r-\int_0^\infty e^{it\Lambda_k(r)-ir\abs{x}} h(x,r) \, \mathrm{d}r\\
        &= \int_0^\infty e^{i\phi_+(r)} h(x,r) \, \mathrm{d}r-\int_0^\infty e^{i\phi_-(r)} h(x,r) \, \mathrm{d}r\\
        &=\vcentcolon I_+(x) +I_-(x),
    \end{align*}
    with the respective phases defined by $$\phi_\pm(r)\vcentcolon =t\Lambda_k(r)\pm r\abs{x},$$
    and with the amplitude 
    $$h(x,r)\vcentcolon=\frac{1}{2i}\abs{x}^{-1}r \varphi(r).$$
    This will allow to treat $I_+$ via integration by parts as in \textbf{Case A}, while the remaining case for $I_-$ will follow via the standard Van der Corput lemma. 
    
   First, we deal with $I_+$: by observing that $\phi_+'(r)=t \Lambda_k'(r)+\vert x \vert \geq t \Lambda_k'(r)$, there holds
   \begin{align*}
     \abs{\frac{1}{\phi_+'(r)}} \leq t^{-1} \abs{\frac{1}{\Lambda'_k(r)}}, \ \   \abs{\partial_r\left(\frac{1}{\phi_+'}\right)} \leq t^{-1}  \abs{\frac{\Lambda_k''(r)}{\Lambda_k'(r)^2}}, \ \  \abs{\partial_r^2\left(\frac{1}{\phi_+'}\right)} \leq t^{-1} \abs{\frac{\Lambda_k'''(r)}{\Lambda_k'(r)^2}}+2t^{-1} \abs{\frac{(\Lambda_k''(r))^2}{\Lambda_k'(r)^3}},
   \end{align*}
   so that thanks to \eqref{eqn:lin-decay-bounds-ibp} and to 
\begin{align*}
    \vert \Lambda''_k(r)\vert + \vert \Lambda'''_k(r) \vert \lesssim 2^{-2k}r (2^{-2k}+r^2)^{-\frac{5}{2}}+ 2^{-2k}\vert 4r^2-2^{-2k} \vert (2^{-2k}+r^2)^{-\frac{7}{2}} \lesssim 2^{-2k^+} 2^{3k^-},
\end{align*}
when $r \sim 1$, 
   there holds on the support of $h$:
    \begin{align}\label{eqn:lin-decay-bounds-phi+}
        \abs{\frac{1}{\phi_+'(r)}}\leq t^{-1}2^{2k^+}2^{-k^-}, 
         \quad \abs{\partial_r\left(\frac{1}{\phi_+'}\right)}\lesssim t^{-1}2^{2k^+}2^{k^-},
         \quad \abs{\partial_r^2\left(\frac{1}{\phi_+'}\right)}\lesssim t^{-1}2^{2k^+}2^{k^-}.
    \end{align}
    Integrating by parts once as in \textbf{Case A} yields with \eqref{eqn:lin-decay-bounds-phi+} that
    \begin{align}
        I_+^{(1)}(x)=-\int_0^\infty e^{i\phi_+(r)}\left(\partial_r\left(\frac{1}{\phi_+'(r)}\right) h(r,x)+\frac{1}{\phi_+'(r)}\partial_r h(r,x)\right) \, \mathrm{d}r,
    \end{align}
    and then, using that $\vert x \vert^{-1} \leq 1$, 
    \begin{align}\label{eqn:lin-decay-J+-1ibp}
        \abs{I_+^{(1)}(x)}\lesssim t^{-1} 2^{2k^+}2^{-k^-}.
    \end{align}
    Similarly, a second integration by parts following \eqref{eqn:lin-decay-CaseA-ibp2-formula} and using \eqref{eqn:lin-decay-bounds-phi+} yields 
    \begin{align}\label{eqn:lin-decay-J+-2ibp}
        \abs{I_+^{(2)}}\lesssim t^{-2} 2^{4k^+}2^{-2k^-}.
    \end{align}
    Therefore interpolation yields the claim on the term $I_+$.

    Finally, we handle $I_-$: in that case, we note that the phase $\phi_-(r)$ has a stationary point for $t\Lambda'_k(r)=\abs{x}$. Thus, we split the analysis in the region away and close to the stationary point.
    \par\textbf{Case B.1: $\abs{x}\geq ct\Lambda_k'(r)$ or $\abs{x}\leq c^{-1}t\Lambda_k'(r)$ for some $c>1$.} In this case, the bounds \eqref{eqn:lin-decay-bounds-phi+} hold similarly for $\phi_-(r)$ as we are away from the critical point. As a consequence, \eqref{eqn:lin-decay-J+-1ibp}, \eqref{eqn:lin-decay-J+-2ibp} and therefore the claim hold analogously for $I_-$.
    \par\textbf{Case B.2: $c^{-1}t\Lambda_k'(r)\leq \abs{x}\leq ct\Lambda_k'(r)$.} Observe that on the support of the localization, the second derivative of the phase $\phi_-$ is nondegenerate with \begin{align*}
        \abs{\phi_-''(r)}=\abs{t\Lambda_k''(r)}=t2^{-2k}r (2^{-2k}+r^2)^{-\frac{5}{2}} \gtrsim  t 2^{-2k^+}2^{3k^-}.
    \end{align*}
     Hence, we may appeal to Van der Corput Lemma \cite[\textcolor{MidnightBlue}{Chapter 8}]{Stein-Shakarchi-FuncAna}: thanks to the lower bound above combined with $\abs{x}^{-1} \lesssim t^{-1} (\Lambda_k'(r))^{-1} \leq t^{-1} 2^{2k^+}2^{-k^-}$
    (inferred from \eqref{eqn:lin-decay-bounds-ibp}), we obtain:
    \begin{align*}
        \abs{I_-(x)}&\leq c\abs{\phi_-''(r)}^{-\frac{1}{2}}\int_0^\infty \abs{e^{i\phi_-(r)} \varphi(r)r\abs{x}^{-1}} \, \mathrm{d}r\lesssim (t 2^{-2k^+}2^{3k^-})^{-\frac{1}{2}} \cdot (t^{-1} 2^{2k^+}2^{-k^-}) =
        t^{-3/2}2^{3k^+}2^{-\frac{5}{2}k^-}.
    \end{align*}
    Combining \textbf{Case A} and \textbf{B} together with \eqref{eqn:decay-rescaling} finishes the proof of \eqref{lin-decay-main-claim}.
\par \textbf{Proof of \eqref{eqn:decay-Pk-profiles} and \eqref{eqn:linear-decay-j<m}.}
  From the decay estimate \eqref{eqn:decay-rescaling-proof} and  Bernstein and H\"older inequalities, for the space-frequency localized pieces \eqref{Q_jk-def} we have:
    \begin{align}
       \norm{e^{it\Lambda}Q_{jk}f}_{L^\infty}&\lesssim  \min\left\{\norm{\widehat{Q_{jk}f}}_{L^1}, |t|^{-\frac32} 2^{6k^+}2^{\frac{k^-}{2}}\norm{Q_{jk}f}_{L^1}\right\}\\ 
       &\label{eqn:preludeQ_jk-decay-j-proof}\lesssim \min\left\{
       2^{\frac{3}{2} k}\norm{Q_{jk}f}_{L^2}, |t|^{-\frac{3}{2}}2^{6k^+}2^{\frac{k^-}{2}}2^{\frac{3}{2} j}{\norm{\varphi_j^{(k)} P_k f}_{L^2}} 
       \right\}. 
        \end{align}
  We want to use the $X$-norm in order to bound the right-hand side above. However, note that the second term does not exactly correspond to the $Q_{jk}f$ and thus needs to be treated more carefully. To that end, we further decompose the function $f$ using \eqref{Q_jk-def} 
    \begin{equation}
        \varphi_j^{(k)} P_k f\vcentcolon=\sum_{(k',j')\in\mathcal{J}} \varphi_j^{(k)} P_k Q_{j'k'} f\vcentcolon=\sum_{\abs{k-k'}\leq 4}\sum_{j':(k',j')\in\mathcal{J}} T^{k,k'}_{j,j'}(P_{k'}f),
    \end{equation}
    where
    \begin{equation}
       T^{k,k'}_{j,j'}g(x)=\int K^{k,k'}_{j,j'}(x,y)g(y) dy,\qquad K^{k,k'}_{j,j'}(x,y)= \varphi_j^{(k)}(x)\mathcal{F}^{-1}[\varphi_k\varphi_{k'}](x-y)\varphi_{j'}^{(k')}(y).
    \end{equation}
    For $\abs{j-j'}>4$, there holds
    \begin{equation}
        \int \abs{K^{k,k'}_{j,j'}(x,y)} dx+\int \abs{K^{k,k'}_{j,j'}(x,y)} dy\lesssim \int_{\abs{z}\gtrsim \max\{2^j,2^{j'}\}}\abs{\mathcal{F}^{-1}[\varphi_k\varphi_{k'}](z)}dz\lesssim \max\{2^j,2^{j'}\}^{-10}.
    \end{equation}
    The last inequality holds since by the non-stationary phase, the convolution kernel associated to $\varphi_k\varphi_{k'}$ decays like $2^{3k}(1+2^k\vert x \vert)^{-10}$.
    Hence from Schur's lemma, it follows that
    \begin{equation}
    \begin{split}
       \norm{\varphi_j^{(k)} P_k f}_{L^2}&\lesssim \sum_{\abs{k-k'}\leq 4}\sum_{\abs{j-j'}\leq 4}\norm{Q_{j'k'}f}_{L^2}+2^{-2j}\sum_{ (k',j')\in\mathcal{J}: \abs{k-k'}\leq 4, \abs{j-j'}>4}2^{-8\abs{j-j'}}\norm{P_{k'}f}_{L^2}\\
       &\lesssim \sum_{\abs{k-k'}\leq 4}\sum_{\abs{j-j'}\leq 4}\norm{Q_{j'k'}f}_{L^2}+2^{-2j}\sum_{\abs{k-k'}\leq 4}\norm{P_{k'}f}_{L^2}.
    \end{split}
    \end{equation}
  Note that for the first  contribution, we can write by definition of the $X$-norm that 
    \begin{align}
   \norm{Q_{j'k'}f}_{L^2}  \lesssim  2^{-8k'^+} 2^{-\gamma k'^-} 2^{-(1+\beta)(k'^+ + j')}\norm{f}_X,
    \end{align}
    while for the second one we have the bound 
    \begin{align*}
        2^{-2j}\norm{P_{k'}f}_{L^2}\lesssim  2^{-2j} \sum_{\abs{k''-k'}\leq 4}\sum_{j'': (k'',j'') \in \mathcal{J}}\norm{P_{k'}Q_{j''k''}f}_{L^2} \lesssim 2^{-2j} 2^{-8k'^+} 2^{-\gamma k'^-} 2^{-(1+\beta)k'^+}\norm{f}_X.
    \end{align*}
    All in all, since $1+\beta \leq 2$ and $\abs{k-k'}\leq 4$, we obtain
    \begin{align}\label{eqn:tech-localization}
        \norm{\varphi_j^{(k)} P_k f}_{L^2} \lesssim 2^{-8k^+}2^{-(1+\beta)(j+k)-\gamma k^-}\norm{f}_X.
    \end{align}
    Coming back to \eqref{eqn:preludeQ_jk-decay-j-proof}, there holds
    \begin{align}\label{eqn:Q_jk-decay-j-proof}
  \norm{e^{it\Lambda}Q_{jk}f}_{L^\infty}&\lesssim  \min\left\{
       2^{\frac{3}{2} k}, |t|^{-\frac{3}{2}}2^{6k^+}2^{\frac{k^-}{2}}2^{\frac{3}{2} j}
       \right\}2^{-8k^+}2^{-(1+\beta)(j+k)-\gamma k^-}\norm{f}_X \\
       &\lesssim \min\left\{2^{-(\frac{15}{2}+\beta)k^+}2^{(\frac{2}{3}+\frac{2}{3}\beta)k^-}2^{-(1+\beta)j}, |t|^{-\frac{3}{2}}2^{-(3+\beta)k^+}2^{(\frac{1}{2}-\beta)j}2^{(-\frac{1}{3}+\frac{2}{3}\beta)k^-}\right\}\norm{f}_X.
       \end{align}
    
       Now we minimize the term on the right-hand side. A direct computation shows that the first term is smaller for $2^{j}\geq\abs{t}2^{\frac{2}{3}k^-}2^{-3k^+}$. Thus we can sum in $j\in \Z_+$ by splitting the sum at this threshold and obtain the decay estimate \eqref{eqn:decay-Pk-profiles}:
       \begin{equation}\label{eqn:computation-lin-decay-X-norm}
       \begin{split}\norm{e^{it\Lambda}P_kf}_{L^\infty}
    &\lesssim \sum_{j+k \geq 0, 2^{j}\geq |t|2^{\frac{2}{3}k^-}2^{-3k^+}} 2^{-(\frac{15}{2}+\beta)k^+}2^{(\frac{2}{3}+\frac{2}{3}\beta)k^-}2^{-(1+\beta)j}\norm{f}_X\\
    &\qquad\qquad+\sum_{j+k \geq 0, 2^{j}\leq |t|2^{\frac{2}{3}k^-}2^{-3k^+}}|t|^{-\frac{3}{2}}2^{-(3+\beta)k^+}2^{(\frac{1}{2}-\beta)j}2^{(-\frac{1}{3}+\frac{2}{3}\beta)k^-}\norm{f}_X\\
    & \lesssim |t|^{-(1+\beta)}2^{-(\frac{9}{2}+2\beta)k^+}\norm{f}_X.
    \end{split}
    \end{equation}
       The estimate \eqref{eqn:linear-decay-j<m} follows from \eqref{eqn:Q_jk-decay-j-proof} and \eqref{eqn:gamma-beta-relation}:
    \begin{align}
\norm{e^{it\Lambda}Q_{jk}f}_{L^\infty}&\lesssim |t|^{-\frac{3}{2}}2^{6k^+}2^{\frac{k^-}{2}}2^{\frac{3}{2} j}2^{-8k^+}2^{-(1+\beta)(j+k)-\gamma k^-}\norm{f}_X\\
        &=\vert t \vert^{-\frac{3}{2}} 2^{-2k^+} 2^{\frac{3}{2}j} 2^{-(1+\beta)(j+k)} 2^{(\frac{2}{3}+\frac{5}{3})k^-}{\norm{f}_X}.
        \end{align}
Using on the one hand the fact that $k+j \geq 0$, or on the other hand that $k=k^+ + k^-$, we obtain
        \begin{align}
\norm{e^{it\Lambda}Q_{jk}f}_{L^\infty}\lesssim |t|^{-\frac{3}{2}}\min\left\{2^{-(3+\beta)k^+}2^{(\frac{1}{2}-\beta)j}2^{(-\frac{1}{3}+\frac{2}{3}\beta)k^-},2^{-2k^+}2^{(\frac{2}{3}+\frac{5}{3}\beta)k^-}2^{\frac{3}{2}j}\right\}\norm{f}_X,
    \end{align}
    which concludes the proof of \eqref{eqn:linear-decay-j<m}.

\end{proof}
\begin{rem}\label{remark:lin-decay}
For the semigroup of the problem there holds that
\begin{align}\Vert e^{it \Lambda} f \Vert_{L^\infty} &\lesssim {t^{-\frac{3}{2}}}\Vert f \Vert_{\dot{B}^{\frac12}_{1,1}\cap \dot{B}^{6}_{1,1}}. \end{align} This is a direct consequence of \eqref{eqn:decay-young} and \eqref{eqn:decay-rescaling-proof}.
\end{rem}

\section{Energy Estimates in \texorpdfstring{$H^s$}{Hs}}\label{sec:energy-estimates}
In this Section we prove standard energy estimates for the unknowns $Z_\pm$ solution to \eqref{eq:Syst-dispersive}.

\begin{prop}\label{prop:Sobolev-estimate}
For $s\geq 0$, there exist constants $c_s,C_s>0$ such that for all solutions $Z_\pm\in C_tH^s_x$ to \eqref{eq:Syst-dispersive} there holds that
\begin{align*}
\sum_{\mu\in \lbrace +,-\rbrace }\norm{Z_\mu(t)}_{H^s}^2\leq C_s \exp\left(c_s\int_0^t \mathrm{A}(\tau) \, \mathrm{d}\tau\right)\sum_{\mu\in \lbrace +,-\rbrace }\norm{Z_\mu(0)}_{H^s}^2, 
\end{align*}
where
\begin{align}\label{eqn:blow-up-crit}
    \mathrm{A}(\tau):=\sum_{\mu\in \lbrace +,-\rbrace } \norm{\nabla \br{\nabla}^{-1}Z_\mu(\tau)}_{W^{1,\infty}}+\norm{\abs{\nabla}Z_\mu(\tau)}_{W^{1,\infty}}.
\end{align}
\end{prop}

\begin{proof}
In the variables $(\rho,\psi)$, the system \eqref{eq:Syst-dispersive} reads (see also \eqref{eq:EP-perturbed-vorticity})
\begin{equation}
  \partial_t \rho + \Delta \psi+\mathrm{div}(\rho \nabla \psi) =0,\qquad   \partial_t  \psi +\frac{1}{2} \vert \nabla \psi \vert^2 =- \br{\nabla}^{-2}\rho.
\end{equation}
Standard energy estimates show that the energies
\begin{equation}
      E_0^2(t):=\frac12 \norm{\br{\nabla}\psi}_{L^2}^2+\frac12 \lVert{\abs{\nabla}^{-1}\rho}\rVert_{L^2}^2,\qquad E_s^2(t):=\frac12 \norm{\nabla\psi}_{H^s}^2+\frac12 \norm{\rho}_{H^{s-1}}^2,\quad s>0,
  \end{equation}
satisfy
\begin{equation}
  \frac{d}{dt}E_s^2(t)\lesssim \left(\norm{\rho(t)}_{W^{1,\infty}}+\norm{\nabla\psi(t)}_{W^{1,\infty}}\right)E_s^2(t),\qquad s\geq 0.
\end{equation}
The claim then follows by observing with Lemma \ref{LEM-link-unkwowns} and \eqref{eqn:rho,psi-in-Z} that
\begin{equation}
  E_0^2(t)+E_s^2(t)\lesssim \sum_{\mu\in \lbrace +,-\rbrace }\norm{Z_\mu(t)}_{H^s}^2 \lesssim  E_0^2(t)+E_s^2(t), \qquad \norm{\rho(t)}_{W^{1,\infty}}+\norm{\nabla\psi(t)}_{W^{1,\infty}}\lesssim\mathrm{A(t)}.
\end{equation}
(All implicit constants may depend on $s\geq 0$.)
\end{proof}

As a consequence, in the bootstrap framework of Proposition \ref{prop:bootstrap}, we obtain uniform energy estimates. 
\begin{cor}\label{cor:blow-up-criterion-decay}
    Let the blow-up criterion $\mathrm{A}(t)$ be defined as in \eqref{eqn:blow-up-crit}. Then under the bootstrap assumption \eqref{eq:bootstrap-assumptionX}, there holds for all $t \geq 0$
    \begin{align}
        \mathrm{A}(t) \lesssim \br{t}^{-(1+\beta)}\eps.
    \end{align}
    As a consequence, under the bootstrap assumption \eqref{eq:bootstrap-assumptionX}, the energy remains uniformly bounded for all $t\geq 0$:
    \begin{align}
        \sum_{\mu\in \{+,-\}}\norm{Z_\mu(t)}_{H^s}\lesssim \eps.
    \end{align}
\end{cor}
\begin{proof}
   Let $Z\in\{Z_\pm\}$. We compute using \eqref{eqn:decay-Pk-profiles} and the bootstrap assumption \eqref{eq:bootstrap-assumptionX}:
   \begin{align}
       \mathrm{A}(t)&=\sum_{\mu}\norm{\nabla \br{\nabla}^{-1}Z_\mu(t)}_{W^{1,\infty}}+\norm{\abs{\nabla}Z_\mu(t)}_{W^{1,\infty}}\\
       &\lesssim \sum_{\mu}\sum_{k\in \Z}(2^{k}(1+2^{2k})^{-\frac12}+2^k)(1+2^{k})\norm{e^{\mu it\Lambda}\zz_\mu(t)}_{L^{\infty}}\\
       &\lesssim \sum_{\mu}\sum_{k\leq 0} 2^{k}t^{-(1+\beta)}\norm{\zz_\mu(t)}_X+\sum_{\mu}\sum_{k>0}2^{-(\frac{5}{2}+\beta)k}t^{-(1+\beta)}\norm{\zz_\mu(t)}_X
       \lesssim t^{-(1+\beta)}\eps.
   \end{align}
    For $0 \leq t \leq 1$, we bound $\mathrm{A}(t)$ by the $X$-norm using the Bernstein inequality and the fact that the semigroup is unitary
     \begin{align*}
        \mathrm{A}(t) &\lesssim  \sum_{k\in\Z} (2^{k}(1+2^{2k})^{-\frac12}+2^k)(1+2^{k})2^{\frac{3}{2}k}\norm{P_k \zz(t)}_{L^2} \\
        & \lesssim \sum_{k \in \Z}  2^{k^+} 2^{\frac{5}{2}k} 2^{-8k^+}2^{-\gamma k^-} \sum_{j+k \geq 0} 2^{-(1+\beta)(k+j)}\norm{\zz(t)}_X \lesssim \eps.
    \end{align*}

\end{proof}

\section{Weighted Estimates}\label{sec:X-norm}
This section contains the main dispersive analysis guaranteeing Theorems \ref{main-thm}, \ref{main-thm-dispersive}. Recall the definition of the $X$-norm \eqref{def:X-norm}
\begin{align}
    \norm{f}_{X}=\sup_{(k,j)\in J}2^{8k^+}2^{(1+\beta)(k+j)}2^{\gamma k^-}\norm{Q_{jk}f}_{L^2},
\end{align}
and the definition of bilinear terms \eqref{def:bilinear-expressions} with phases and multipliers as in \eqref{def:phases}, \eqref{def:multipliers}:
\begin{align}
    \mathcal{F}(\mathcal{B}_\m(f,g))(\xi,t)=\int_0^t\int_{\R^3} e^{is\Phi(\xi,\eta)}\m(\xi,\eta)\widehat{f}(\xi-\eta)\widehat{g}(\eta)d\eta ds.
\end{align}
Moreover, recall that in the proof of the main bootstrap argument Proposition \ref{prop:bootstrap}, we localize each of these bilinear expressions on dyadic in time intervals, so we in fact deal with expressions $\mathcal{B}_\m^m(f,g)$ as defined in \eqref{eqn: bootstrap B_m time decomposition}. Throughout this section, we drop the superscript $m$ for ease of notation.
\par In this section we use the notation \eqref{eqn: chi localizations} regarding the spatial and frequency localizations we employ and moreover: 
\begin{equation}\label{def:localizations-max-min}
\begin{split}
    &k_{\max}\vcentcolon=\max\{k,k_1,k_2\},\quad k_{\min}\vcentcolon=\min\{k,k_1,k_2\},\quad  j_{\max}\vcentcolon=\max\{j_1,j_2\},\quad j_{\min}\vcentcolon=\min\{j_1,j_2\}.
\end{split}
\end{equation}
Moreover, as a consequence of the triangle inequality we have the following possible settings for the sizes of the frequencies. This will be used throughout the following Sections, in particular Lemma \ref{LM:bound-from-below-PHASE}.
\begin{rem}\label{rem:notation-rel-size-frequencies} For $k,k_1,k_2\in \Z$ as in \eqref{def:frequencies-k}, one of the following scenarios is true:
\begin{enumerate}
    \item  There holds  $k_1\leq k-4$, and as a consequence $\abs{k-k_2}\leq 2$.  In this case we use the notation $k_1\ll k\sim k_2$.
    \item  There holds  $k_2\leq k-4$, and as a consequence $\abs{k-k_1}\leq 2$.  In this case we use the notation $k_2\ll k\sim k_1$.
    \item There holds $\abs{k_1-k_2}<4$. Then either $k\leq \min\{k_1,k_2\}-4$, in which case we use the notation $k\ll k_1\sim k_2$, or $\abs{k-\min\{k_1,k_2\}}<4$ and we use the notation $k\sim k_1\sim k_2$.
\end{enumerate}
\end{rem}
The main result of this section is the following:
\begin{prop}\label{prop:X-norm-bounds} In the setting of Proposition \ref{prop:bootstrap} and under the bootstrap assumption \eqref{eq:bootstrap-assumptionX}, consider bilinear terms \eqref{def:bilinear-expressions} with phases and multipliers as in \eqref{def:phases},\eqref{def:multipliers}. Then for $m\in \N$, $t\in ]2^m,2^{m+1}]\cap[0,T]$ and $F_i \in \{\mathcal{Z}_\pm\}$, there holds:
    \begin{align*}
        \sup_{(k,j)\in \mathcal{J}}2^{8k^+}2^{(1+\beta)(k+j)}2^{\gamma {k^-}}\norm{Q_{jk}\mathcal{B}_{\m}(F_1,F_2)}_{L^2} \lesssim 2^{-\frac{\delta}{2} m}\eps^2.
    \end{align*}
\end{prop}
\begin{rem}\label{rem:parameters}
    We recall the relation between the fixed parameters involved in the $X-$norm \eqref{def:X-norm}:
    \begin{align}
    \gamma=-\frac{1}{6}-\frac{5}{3}\beta, \qquad \gamma+\beta=-\frac{1}{6}-\frac{2}{3}\beta.
\end{align}
There are several more parameters involved in the proof of Proposition \ref{prop:X-norm-bounds} throughout this section, in particular, $\delta>0$ appearing in the statement above. As a result of the proof, they are related as follows:
\begin{align}
    6\delta<\beta<\frac{1}{5}-6\delta,\qquad 0<\delta\leq \frac{9}{40}\beta, \qquad \delta_0\leq 10^{-1}\delta, \qquad N_0\geq 3\delta_0^{-1}+13.
\end{align}
In particular, one can choose $\beta=\frac{1}{5}-10\delta$, such that $\delta=0.01$, $\delta_0=0.001$, $N_0=3020$ and from \eqref{eqn:gamma-beta-relation} $\gamma=-\frac{1}{2}+\frac{50}{3}\delta$.
\end{rem}
\begin{proof}[Proof of Proposition \ref{prop:X-norm-bounds}] The proof is a direct consequence of the two following Propositions \ref{prop:finite-speed-of propagation}, \ref{prop:X-norm-j<m}.  
\end{proof}
We first formalize the fact that waves travel at finite speed and so affect the solution only in the light cone of size $\sim |t|$.
In our oscillatory integral setting, we make this concrete by switching to the frequency space, where there should hold $\abs{x}\lesssim t|\nabla_\xi\Phi|$. Hence, in terms of our parameters, the nonlinear interactions decay fast if $2^j\gtrsim 2^m\abs{\nabla_\xi\Phi}$. For the phases \eqref{def:phases} there holds
\begin{align}\label{eqn:bound-nabla_xiPhi}
    \abs{\nabla_\xi\Phi}=\abs{{\xi}\abs{\xi}^{-1}\langle\xi\rangle^{-3}\pm {(\xi-\eta)}\abs{\xi-\eta}^{-1}\langle\xi-\eta\rangle^{-3} }\leq 3\cdot 2^{-3k_{\min}^+}\leq 3,
\end{align}
which motivates the restriction on $j$ in the proposition below.
\begin{prop}[Finite Speed of Propagation]\label{prop:finite-speed-of propagation}
  Consider the setting of Proposition \ref{prop:bootstrap} and let the bootstrap assumption \eqref{eq:bootstrap-assumptionX} hold true. Moreover, let $\m\in \{\m^{\mu\nu}_\pm, \; \mu,\nu\in \{+,-\}\}$, $t\in [2^{m},2^{m+1}[\cap[0,T]$, $A\geq 3$, $\delta$ as in Remark \eqref{rem:parameters} and $F_i\in \{\mathcal{Z}_\pm\}$ be the profiles of the problem. Then the following estimate holds:
    \begin{align}
       \sup_{\substack{(k,j)\in \mathcal{J}\\ j\geq m+A}} 2^{8k^+}2^{(1+\beta)(j+k)}2^{\gamma k^-}  \|Q_{jk}\mathcal{B}_{\m}(F_1,F_2)\|_{L^2}\lesssim 2^{-\frac{\delta}{2} m}\eps^2.
    \end{align}
\end{prop} 
 
\begin{proof}[Proof of Proposition \ref{prop:finite-speed-of propagation}] The proof follows a standard approach in treating such nonlinearities via dispersive methods. To begin with, one localizes the inputs $F_1,F_2$ in frequency and the available energy estimates \eqref{eqn:boostrap-energy-bounded} together with set-size estimates Lemma \ref{Lem-bilinear-estimates} \eqref{set-size-esitmate} allow to restrict the frequency localization parameters on logarithmic-in-$j$ intervals. Moreover, one additionally localizes the inputs $F_1,F_2$ according to their spatial localizations \eqref{Q_jk-def} and restricts these parameters in terms of the output  parameter $j$. This parameter reduction is done in Step 1 below. This puts us in a favourable setting to make use of the finite speed of propagation argument outlined above. To that end, in Step 2, one employs an iterated integration by parts argument detailed in Section \ref{sec:ibp} to conclude the claim for the remaining range of parameters.
\par \textbf{Step 1: Set size and energy estimates.} We localize the size of the frequencies of the profiles 
\begin{align}\label{eqn:freq-loc-profiles-X-norm-proof}
   F_i=\sum_{k_i\in\Z}P_{k_i}F_i, \quad i=1,2. 
\end{align} 
Applying Lemma \ref{Lem-bilinear-estimates} with an $L^2-L^2$ bound, together with the bound \eqref{main-multiplier-bound} on the multiplier $\norm{\m\chi}_{S^\infty}\lesssim 2^{k_1+k_2}$ and using the energy estimates \eqref{eqn:boostrap-energy-bounded}, we obtain:
\begin{align}
     2^{8k^+}2^{(1+\beta)(j+k)}&2^{\gamma k^-}  \norm{Q_{jk}\mathcal{B}_{\m}(F_1,F_2)}_{L^2}\\
     &\lesssim 2^{8k^+}2^{(1+\beta)(j+k)}2^{\gamma k^-}  2^{m}\sum_{k_1,k_2 \in \Z}2^{\frac{3}{2}k_{\min}}\norm{\m\chi}_{S^\infty}\norm{P_{k_1}F_1}_{L^2}\norm{P_{k_2}F_2}_{L^2}\\
     &\lesssim  \sum_{k_1,k_2 \in \Z}2^{8k^+}2^{(1+\beta)(j+k)}2^{\gamma k^-} 2^m 2^{k_1+k_2}2^{\frac{3}{2}k_{\min}}2^{-N_0k_1^+}2^{-N_0k_2^+}\norm{P_{k_1}F_1}_{H^{N_0}}\norm{P_{k_2}F_2}_{H^{N_0}}\\
     &\lesssim \sum_{k_1,k_2 \in \Z} 2^{(2+\beta+\delta)j}2^{-\delta m}2^{(\frac{7}{3}-\frac{2}{3}\beta)k_{\min}}2^{-N_0\min\{k_1,k_2\}^+}2^{-(N_0-12)k_{\max}^+}\eps^2.
\end{align}
The last inequality above follows since we have used $m\lesssim j$ and the coefficient of $k_{\min}$ comes from the case $k=k_{\min}$ and the coefficients in the $X-$norm.
Hence, if either $(\frac{7}{3}-\frac{2}{3}\beta)k_{\min}<-(2+\beta+2\delta)j$ or $k_{\max}>\delta_0 j$ where $\delta_0=3/(N_0-12)$ (cf.\ Remark \ref{rem:parameters}), then  we obtain  
\begin{align}
     2^{8k^+}2^{(1+\beta)(j+k)}2^{\gamma k^-}  \|Q_{jk}\mathcal{B}_{\m}(F_1,F_2)\|_{L^2}&\lesssim 2^{-\delta (j+m)}\eps^2 \lesssim 2^{-\delta m}\eps^2.
\end{align}
By Remark \ref{rem:parameters}, we may assume:
\begin{align}\label{eqn:parameter-restriction-finite-speed}
   -(1-\delta)j< k,k_1,k_2\leq \delta_0 j.
\end{align}
Moreover, we localize further in $j_1,j_2$ using \eqref{Q_jk-def} with \begin{align}
    P_{k_i}F_i=\sum_{j_i+k_i\geq 0}Q_{k_i,j_i}F_i, \quad f_i=Q_{j_i,k_i}F_i.
\end{align} 
For the following computation assume w.l.o.g.\ that $k_1\leq k_2$ so that $k_{\max}\sim k_2$ and $k\lesssim k_2$. In particular, observe that if $j_1,j_2\geq j_{\min}\geq(1-\delta)j$, by the $L^2-L^2$ bound in Lemma \ref{Lem-bilinear-estimates} we obtain: 
\begin{align}
     2^{8k^+}2^{(1+\beta)(j+k)}2^{\gamma k^-}  &\norm{Q_{jk}\mathcal{B}_{\m}(P_{k_1}F_1,P_{k_2}F_2)}_{L^2}\\
     &\lesssim  2^{8k^+}2^{(1+\beta)(j+k)}2^{\gamma k^-} 2^m\sum_{j_1,j_2}2^{\frac{3}{2}k_{\min}}\norm{\m\chi}_{S^\infty}\norm{f_1}_{L^2}\norm{f_2}_{L^2}\\
     &\lesssim 2^{8k^+}2^{(1+\beta)(j+k)}2^{\gamma k^-} 2^{m}2^{k_1+k_2}2^{\frac{3}{2}k_{\min}}\\
     &\qquad \qquad\qquad\sum_{j_1,j_2}2^{-(1+\beta)(j_1+k_1)}2^{-\gamma k_1^-}2^{-(1+\beta)(j_2+k_2)}2^{-\gamma k_2^-}2^{-8k_1^+-8k_2^+}\norm{f_1}_X\norm{f_2}_X\\
     &\lesssim 2^{(2+\beta+\frac{3}{2}\delta)j}2^{-\frac{3}{2}\delta m}2^{4k_{\max}^+}\sum_{j_1,j_2, j_{\min}\geq (1-\delta)j}2^{-(2+2\beta-\delta^2)j_{\min}}2^{-\delta^2j_{\max}}\eps^2\\
     &\lesssim 2^{(-\beta+\frac{3}{2}\delta +4\delta_0)j}2^{-\frac{3}{2}\delta m}\eps^2\\
     &\lesssim 2^{-\frac{3}{2}\delta(j+m)}\eps^2.
\end{align}
Note that in the last line we have also used  $k,k_1,k_2\lesssim\delta_0 j$ from \eqref{eqn:parameter-restriction-finite-speed} and Remark \ref{rem:parameters}. 
Hence, it remains to prove
\begin{align}\label{eqn:fin-speed-prop-summation-ks-js}
   2^{8k^+}2^{(1+\beta)(j+k)}2^{\gamma k^-}  \sum_{k_1,k_2\in \mathcal{K}}\sum_{j_1,j_2\in \mathcal{L}}\|Q_{jk}\mathcal{B}_{\m}(f_1,f_2)\|_{L^2}\lesssim 2^{-\frac{3}{2}\delta (j+m)}\eps^2,
\end{align}
where 
\begin{equation}\label{eqn:K-L-sets-def}
\begin{split}
      \mathcal{K}&\vcentcolon=\{k_1,k_2\in \Z\;\vert\;   -(1-\delta)j\leq k,k_1,k_2\leq \delta_0 j \},\\
    \mathcal{L}&\vcentcolon=\{j_1,j_2\in \Z^+\;\vert\; j_i+k_i\geq 0,\;  j_{\min}\leq (1-\delta)j \}.
\end{split}  
\end{equation}
Thus, each parameter in the summation in \eqref{eqn:fin-speed-prop-summation-ks-js} ranges over an interval of size $j\leq 2^{\kappa j}$ for all $\kappa> 0$ and so it suffices to prove
\begin{align}\label{eqn:fin-speed-prop-step2-claim}
   2^{8k^+}2^{(1+\beta)(j+k)}2^{\gamma k^-} \norm{Q_{jk}\mathcal{B}_{\m}(f_1,f_2)}_{L^2}\lesssim 2^{-2\delta (j+m)}\eps^2,
\end{align}
for fixed parameters $k,k_1,k_2\in \mathcal{K}$ and $j_1,j_2\in \mathcal{L}.$

\par \textbf{Step 2: Proof of \eqref{eqn:fin-speed-prop-step2-claim}.} For the remaining frequency and space localizations, we want to take advantage of oscillations in $\xi$ and integrate by parts in $\nabla_\xi$ using Lemma \ref{lemma-IBPgen}. 
\par \underline{1. Assume $j_{\min}=j_1$}. Recall from \eqref{Q_jk-def} and \eqref{eqn: bootstrap B_m time decomposition} that
\begin{align}
    Q_{jk}\mathcal{B}_{\m}(f_1,f_2)&=\varphi_j^{(k)}\int_0^t \tau_m(s)\int_{\R^6} e^{ix\cdot \xi+is\Phi}\varphi_k(\xi)\m(\xi,\eta)\widehat{f_1}(\xi-\eta)\widehat{f_2}(\eta)d\eta d\xi ds\\
    &=\varphi_j^{(k)}\int_0^t \tau_m(s)\int_{\eta}\left[\int_{\xi} e^{ix\cdot \xi+is\Phi}\varphi_k(\xi)\m(\xi,\eta)\widehat{f_1}(\xi-\eta) d\xi\right] \widehat{f_2}(\eta)d\eta ds
\end{align}
In view of \eqref{eqn:bound-nabla_xiPhi} and the assumption of the Lemma on the parameter $j\geq m+A$, there holds
\begin{align}
    \abs{\nabla_\xi(x\cdot\xi+s\Phi)}=\abs{x+s\nabla_\xi\Phi}\geq |x|-s|\nabla_\xi\Phi|\gtrsim 2^j-3\cdot 2^m2^{-3k_{\min}^+}\geq 2^{j-1},
\end{align}
since the constant satisfies by assumption $A\geq 3$.
We want to take advantage of the oscillations of the phase and integrate by parts in the expression in square brackets above using the  iterated bounds of Lemma \ref{lemma-IBPgen}. With the notation of Lemma \ref{lemma-IBPgen} and the bounds on higher order derivatives of the phase \eqref{lem:it-phase-bounds} we have 
\begin{align}
    &F(\xi)=2^{-j}(x\cdot \xi+s\Phi), \quad K= 2^j,\quad \abs{\nabla_\xi F}\gtrsim 1, \\
    &|D_\xi^\alpha F(\xi)|\lesssim 2^{-j}2^m|D_\xi^\alpha\Phi|\lesssim 2^{(1-\abs{\alpha})\min\{k,k_1\}}, \text{ with } \vert \alpha \vert \geq 2,\quad  \epsilon= 2^{\min\{k,k_1\}}.
\end{align}
Thus by Lemma \ref{lemma-IBPgen} and Lemma \ref{lem:ibp} with $g(\xi,\eta)=\m(\xi,\eta)\varphi_k(\xi)\widehat{f_1}(\xi-\eta)$ there holds
\begin{align}\label{fin-speed-ibp}
\begin{split}
    &\abs{\int e^{ix\cdot \xi+is\Phi}\varphi_k(\xi)\m(\xi,\eta)\widehat{f_1}(\xi-\eta) d\xi} \\
    &\qquad \lesssim \frac{1}{(2^{\min\{k,k_1\}+j})^M}\sum_{\abs{\alpha}\leq M}2^{\abs{\alpha}\min\{k,k_1\}}\norm{D^{\abs{\alpha}}_\xi g}_{L^1}\\ 
    &\qquad  \lesssim \frac{1}{(2^{\min\{k,k_1\}+j})^M}\sum_{\abs{\alpha}\leq M}2^{\abs{\alpha}\min\{k,k_1\}} 2^{k_{\max}^+}(2^{-\min\{k,k_1\}}+2^{j_1})^{\abs{\alpha}}\norm{\varphi_k}_{L^2}\norm{f_1}_{L^2} \\
    &\qquad \lesssim 2^{-jM}2^{k_{\max}^+}(2^{-k_{\min}}+2^{j_1})^M 2^{\frac32 k}\norm{{f_1}}_{L^2}.
    \end{split}
\end{align}
Hence, by the restriction \eqref{eqn:K-L-sets-def} the claim follows by choosing $M$ such that $M\delta >4$. :
\begin{align*}
    2^{8k^+}&2^{(1+\beta)(j+k)}2^{\gamma k^-}  \norm{Q_{jk}\mathcal{B}_{\m}(f_1,f_2)}_{L^2}\\
     &\lesssim  2^{10 k^+_{\max}}2^{(2+\beta+2\delta)j}2^{-2\delta m}\norm{\widehat{f_2}}_{L^1} 2^{-jM}2^{k_{\max}^+}(2^{-k_{\min}}+2^{j_1})^M 2^{\frac32 k}\norm{{f_1}}_{L^2}\\
     &\lesssim 2^{3j}2^{-2\delta m}2^{-jM}2^{(1-\delta)M j}\eps^2\\
     &\lesssim 2^{-2\delta(j+m)}\eps^2.
\end{align*}
\par\underline{2. $j_{\min}=j_2$.} This case follows similarly by a change of variables $\eta\leftrightarrow\xi-\eta$ and integrating by parts in
\begin{align}
  \int e^{ix\cdot \xi+is\Phi(\xi,\xi-\eta)}\varphi_k(\xi)\m(\xi,\xi-\eta)\widehat{f_2}(\xi-\eta) d\xi.
\end{align}
The corresponding conditions of Lemma \ref{lemma-IBPgen} are satisfied with $F,K$ as above,  $\epsilon=\min\{k,k_2\}$ and $g=\varphi_k(\xi)\m(\xi,\xi-\eta)\widehat{f_2}(\xi-\eta)$.
\end{proof}

Finally, it remains to prove the claim of Proposition \ref{prop:X-norm-bounds} for the remaining spatial localization parameters $j$. We do this in Proposition \ref{prop:X-norm-j<m} below. The main ingredients allowing for this result are favourable lower bounds on the phase $\Phi$, see \eqref{phase-bound-proof} and Lemma \ref{LM:bound-from-below-PHASE}, in conjunction with good cancellations in the multipliers, see Lemma \ref{lem:multiplier-bounds}. As a consequence, after reducing the interacting frequencies via the available energy estimates \eqref{eqn:boostrap-energy-bounded}, we may perform a normal form via \eqref{eqn:normal-form-decomposition} in \textbf{Step 2} below. The restriction on the rate of decay due to $\beta<\frac{1}{5}$ (c.f.\ Remark \ref{rem:parameters}) arises from treating the boundary term in \ref{eqn:normal-form-decomposition} without additional oscillatory arguments. (We recall however, that this allows to avoid a neutrality assumption on the perturbation, see Remark \textbf{(2)} following the main theorem in Section \ref{sec:Motivation-Main-Result}.) Finally, performing a normal form introduces cubic order terms involving the time derivative of the profiles. These contributions are generally easier to control since, as a consequence of the refined decay estimate \eqref{eqn:linear-decay-j<m} and the favourable structure of the multipliers, the time derivative of the profiles in $L^2$ decays at the full rate in $3$D. We prove this in the following Lemma. 

\begin{lem}[Decay in $L^2$ of $\partial_t\mathcal{Z}_\pm$]\label{lem:time-derivative-L2} Let $t\in [2^{m},2^{m+1}[\cap [0,T]$ and $F\in \{\zz_+,\zz_-\}$. Under the bootstrap assumption \eqref{eq:bootstrap-assumptionX}, the time derivative of a profile $F$ satisfies the improved decay estimate:
\begin{align}
    \norm{P_k\partial_tF(t)}_{L^2}\lesssim 2^{-\frac32m}\eps^2.
\end{align}   
\end{lem}
\begin{proof}  First of all we note that from \eqref{eqn:Duhamel-profiles},  we have 
\[\norm{P_k\partial_tF(t)}_{L^2}\lesssim\sum_{\mu,\nu}\sum_{\substack{k_1,k_2,j_1,j_2 \\ j_i+k_i\in \mathcal{J}}}\norm{P_k\mathcal{Q}_{\m_\pm^{\mu\nu}}(Q_{k_1,j_1}F_1,P_{k_2,j_2}F_2)}_{L^2},\] where $F_i\in \{\zz_{\pm}\}$ correspondingly. We recall the notation for the localization parameters \eqref{eqn: chi localizations}, \eqref{def:localizations-max-min} that will be used throughout the proof. We bound the terms on the right-hand side above with $\m\in\{m_\pm^{\mu\nu}\}$, observing using  \eqref{main-multiplier-bound}.
    To that end, we rely on an $L^\infty-L^2$ bound in Lemma \ref{Lem-bilinear-estimates}, the linear decay estimate \eqref{eqn:linear-decay-j<m} and the bootstrap assumption \eqref{eq:bootstrap-assumptionX}. Using the additional notation $\overline{k}_{\max}=\max\{k_1,k_2\}$ and $\underline{k}_{\min}=\min\{k_1,k_2\}$ we obtain
\begin{align}
        &\norm{P_k\mathcal{Q}_{\m}(Q_{k_1,j_1}F_1,Q_{k_2,j_2}F_2)}_{L^2}\\
        &  \lesssim \sum_{\substack{k_1,k_2,j_1,j_2 \\ j_i+k_i\in \mathcal{J}}} \norm{\m \chi}_{S^\infty}\min\left\{\norm{e^{it\Lambda}Q_{k_1,j_1}F_1}_{L^\infty}\norm{Q_{k_2,j_2}F_2}_{L^2},\norm{e^{it\Lambda}Q_{k_2,j_2}F_2}_{L^\infty}\norm{Q_{k_1,j_1}F_1}_{L^2}\right\} \\
        & \lesssim \sum_{\substack{k_1,k_2,j_1,j_2 \\ j_i+k_i\in \mathcal{J}}}2^{k_1+k_2} 2^{-8\underline{k}_{\min}^+}2^{-(1+\beta)j_{\max}-(1+\beta)\underline{k}_{\min}-\gamma \underline{k}_{\min}^-}2^{-\frac{3}{2}m}2^{(\frac{1}{2}-\beta)j_{\min}-(\frac{1}{3}-\frac{2}{3}\beta)\overline{k}_{\max}^{-}}2^{-(3+\beta)\overline{k}_{\max}^+}\norm{F_1}_X\norm{F_2}_X\\
        &\lesssim 2^{-\frac{3}{2}m}\sum_{\substack{k_1,k_2,j_1,j_2 \\ j_i+k_i\in \mathcal{J}}} 2^{(\frac{1}{6}+\frac{2}{3}\beta)\underline{k}_{\min}^-}2^{(\frac{2}{3}+\frac{2}{3}\beta)\overline{k}_{\max}^-}2^{-8\underline{k}_{\min}^+}2^{-2\overline{k}_{\max}^+} 2^{-\beta j_1} 2^{-(\frac{1}{2}+\beta)j_2}\eps^2,
    \end{align}
    where we have used \eqref{eqn:gamma-beta-relation}. The sum above is convergent, so this concludes the proof.
    \end{proof}
\begin{rem} We make the following remark on the summability of the decay estimate for the time derivative. In fact, for any $0<\kappa<1$ there holds:
    \begin{align}
        \norm{\partial_t F(t)}_{L^2}\lesssim 2^{-(\frac{3}{2}-\kappa)m}\eps^2.
    \end{align}
    This follows by slightly adapting the proof above. In the setting of Lemma \ref{lem:time-derivative-L2} we may
    interpolate between the estimate above and an $L^2-L^2$ set-size estimate using Lemma \ref{Lem-bilinear-estimates}. With $0<\kappa<1$ and since $k\lesssim \overline{k}_{\max}$:
    \begin{align}
         \norm{P_k\partial_tF(t)}_{L^2}&\lesssim \min\Big\{2^{-\frac{3}{2}m}\eps^2,\sum_{k_1,k_2\in \Z} \norm{\m\chi}_{S^\infty} 2^{\frac{3}{2}k}\norm{P_{k_1}F_1}_{L^2}\norm{P_{k_2}F_2}_{L^2}\Big\}\\
         &\lesssim  2^{-(\frac{3}{2}-\kappa)m}2^{\kappa k}2^{-2\kappa k^+} \Big(\sum_{k_1,k_2\in \Z}2^{k_1+k_2}2^{-(N_0-3)k_1^+}2^{-(N_0-3)k_2^+}\Big)^{\frac{2}{3}\kappa}\eps^2.
    \end{align}
    The claim follows then by summation.
\end{rem}
Finally, we complete the proof of Proposition \ref{prop:X-norm-bounds}.
\begin{prop}\label{prop:X-norm-j<m}
 Consider the setting of Proposition \ref{prop:bootstrap} and let the bootstrap assumption \eqref{eq:bootstrap-assumptionX} hold true. Moreover, let $\m\in \{\m^{\mu\nu}_\pm, \; \mu,\nu\in \{+,-\}\}$, $t\in [2^{m},2^{m+1}[\cap[0,T]$, $A\geq 3$, $\delta>0$ as in Remark \ref{rem:parameters} and $F_i\in \{\mathcal{Z}_\pm\}$ be the profiles of the problem. Then the following estimate holds:
     \begin{align}
       \sup_{\substack{k+j \in \mathcal{J} \\ j<m+A}} 2^{8k^+}&2^{(1+\beta)(j+k)}2^{\gamma k^-} \norm{Q_{jk}\mathcal{B}_{\m}(F_1,F_2)}_{L^2}\lesssim 2^{-{\frac{\delta}{2}}m}\eps^2.
    \end{align}
\end{prop}
\begin{proof}
   
Since $j\leq m+A$, it suffices to prove
\begin{align}
        2^{8k^+}&2^{(1+\beta)(m+k)}2^{\gamma k^-} \norm{Q_{jk}\mathcal{B}_{\m}(F_1,F_2)}_{L^2}\lesssim 2^{-\frac{\delta}{2}m}\eps^2.
    \end{align}
    We again follow the approach of Proposition \ref{prop:finite-speed-of propagation}. We begin by localizing the inputs $F_i$ in frequency on dyadic scales as in \eqref{eqn:freq-loc-profiles-X-norm-proof}. In Step 1, we use the available energy estimates \eqref{eqn:boostrap-energy-bounded} and set-size estimates to reduce the frequency localization parameters \eqref{eqn:j-upperbound-standard}. As outlined in the introduction Section \ref{sec:outline-of-proof}, for the remaining parameters we do a normal form transformation detailed in Step 2 below.

\par\textbf{Step 1: Set-size and energy estimates.}
With the notation \eqref{eqn:freq-loc-profiles-X-norm-proof}, we use the $L^2-L^2$ estimate in  Lemma \ref{Lem-bilinear-estimates}, and the identity \eqref{eqn:gamma-beta-relation} together with the energy estimates \eqref{eqn:boostrap-energy-bounded} to obtain:
\begin{align}
     2^{8k^+}&2^{(1+\beta)(m+k)}2^{\gamma k^-}\norm{Q_{jk}\mathcal{B}_\m(F_1,F_2)}_{L^2}\\
    &\lesssim 2^{8k^+}2^{(2+\beta)m}2^{(1+\beta)k+\gamma k^-}\sum_{k_1,k_2}2^{\frac32 k_{\min}}\norm{\m\chi}_{S^\infty}2^{-N_0k_1^+}\norm{P_{k_1}F_1}_{H^{N_0}}2^{-N_0k_2^+}\norm{P_{k_2}F_2}_{H^{N_0}}\\
    &\lesssim 2^{(9+\beta)k^+}2^{(2+\beta)m}\sum_{k_1,k_2}2^{(1+\beta+\gamma)k^-}2^{\frac{3}{2}k_{\min}^-}2^{\frac{3}{2}k_{\min}^+}2^{k_1^-+k_2^-} 2^{-(N_0-1)k_1^+}2^{-(N_0-1)k_2^+}\\
    &\lesssim 2^{(9+\beta)k^+}2^{(2+\beta)m}\sum_{k_1,k_2}2^{(\frac73-\frac23\beta) k_{\min}^-}2^{(\frac{1}{6}+\frac{2}{3}\beta)(k_1^-+k_2^-)}2^{\frac{3}{2}k_{\min}^+}2^{-(N_0-1)k_1^+}2^{-(N_0-1)k_2^+}\eps^2 .
\end{align}
The claim follows if   $\max\{k_1,k_2\}\geq \delta_0 m$ with $\delta_0=\frac{3}{N_0-13}$. Moreover, the claim follows as long as $(\frac73-\frac23\beta)k <-(2+\beta+\delta)m$ or if $\frac{5}{2}\min\{k_1,k_2\}\leq -(2+\beta+\delta)m$. Thus, with Remark \ref{rem:parameters} on the size of the parameters, we may assume
\begin{align}\label{eqn:k-restrictions-j<m}
    -(1-4\delta)m<k, k_1,k_2\lesssim \delta_0m.
\end{align}
Since the remaining summation is over $O(m^2)$ terms, it suffices to prove
\begin{align}
        2^{8k^+}2^{(1+\beta)k+\gamma k^-}2^{(1+\beta)m}&\norm{Q_{jk}\mathcal{B}_\m(P_{k_1}F_1,P_{k_2}F_2)}_{L^2}\lesssim 2^{-\delta m}\eps^2.
\end{align}
Moreover, we may additionally localize the inputs in their spatial localizations \begin{align}
    P_{k_i}F_i=\sum_{j_i+k_i\geq 0} Q_{j_i,k_i}F_i=\sum_{j_i+k_i\geq 0}f_i.
\end{align} 
We first note that if $j_{\max}\geq (1-\delta)m$, the claim follows using the $L^2-L^\infty$ bound in Lemma \ref{Lem-bilinear-estimates}. Indeed, employing \eqref{main-multiplier-bound}, the decay estimate \eqref{eqn:decay-Pk-profiles} on the input corresponding to $j_{\min}$ and the $X-$norm on the input corresponding to $j_{\max}$ to obtain:
\begin{align}
    2^{8k^+}&2^{(1+\beta)k+\gamma k^-}2^{(1+\beta)m}\norm{Q_{jk}\mathcal{B}_\m(P_{k_1}F_1,P_{k_2}F_2)}_{L^2}\\
    &\lesssim 2^{8k^+}2^{(1+\beta)k+\gamma k^-}2^{(1+\beta)m}2^m\sum_{j_i, j_i+k_i\geq 0} \norm{\m\chi}_{S^\infty}\min\left\{\norm{e^{it\Lambda}f_1}_{L^\infty}\norm{f_2}_{L^2}, \norm{f_1}_{L^2}\norm{e^{it\Lambda}f_2}_{L^\infty} \right\}\\
    &\lesssim 2^{8k^+}2^{(1+\beta)k+\gamma k^-}2^{(2+\beta)m}2^{k_1+k_2}2^{-(1+\beta)m}\\
    &\quad\cdot\sum_{j_i, j_i+k_i\geq 0}2^{-8\max\{k_1,k_2\}^+}2^{-(1+\beta)j_{\max}}2^{-(1+\beta)\max\{k_1,k_2\}-\gamma \max\{k_1,k_2\}^-} 2^{-(\frac{9}{2}+2\beta)\min\{k_1,k_2\}^+}\eps^2\\
    &\lesssim 2^{(1+5\delta_0)m} \sum_ {j_i, j_i+k_i\geq 0}2^{-(1+\beta-\delta^2)j_{\max}}2^{-\delta^2j_{\min}}\eps^2\\
    &\lesssim 2^{-\delta m}\eps^2.
\end{align}
Note the sizes of parameters fixed in Remark \ref{rem:parameters}.
Thus, under the restriction \eqref{eqn:k-restrictions-j<m}, it suffices to prove the following claim:
\begin{align}\label{eqn:j-upperbound-standard}
    2^{8k^+}2^{(1+\beta)k+\gamma k^-}2^{(1+\beta)m}&\norm{Q_{jk}\mathcal{B}_\m(f_1,f_2)}_{L^2} \lesssim 2^{-\frac{3}{2}\delta m}\eps^2, \quad 0\leq j_{\max}<(1-\delta)m.
\end{align}
The summation over $O(m^2)$ terms yields then the desired claim.
\par\textbf{Step 2: Proof of \eqref{eqn:j-upperbound-standard}.}
The main idea is to take advantage of oscillations in time of the phase $\Phi$ in the bilinear terms. Concretely, this amounts to establishing favourable lower bounds on the phase of the form $\abs{\Phi}\gtrsim \lambda(k,k_1,k_2)$, for some $\lambda>0$  adapted to the relevant geometric setting. These are conducted in Lemma \ref{LM:bound-from-below-PHASE}. Once an admissible lower bound is available, one performs a normal form (i.e.\ integrates by parts in time) in the bilinear expressions and obtains:
\begin{align}\label{eqn:normal-form-decomposition}
    \mathcal{B}_{\m}(f_1,f_2)=\mathcal{Q}_{\m\Phi^{-1}}(f_1,f_2) +\mathcal{B}_{\m\Phi^{-1}}(\partial_tf_1,f_2)+\mathcal{B}_{\m\Phi^{-1}}(f_1,\partial_tf_2).
\end{align}
Thus, it remains to prove the claim \eqref{eqn:j-upperbound-standard} for the three terms on the right-hand side above. Note that we have a boundary term $\mathcal{Q}_{\m\Phi^{-1}}(f_1,f_2)$, which has one less time parameter to control, and two cubic terms for which we can use Lemma \ref{lem:time-derivative-L2}. This procedure modifies the multiplier $\m\longrightarrow \m\Phi^{-1}$, bounds on which are obtained using Lemma \ref{lem:multiplier-bounds}. The main point however, is that we can control $\abs{\Phi}^{-1}$ in terms of the available parameters. Indeed, by Lemma \ref{LM:bound-from-below-PHASE} there holds
\begin{align}\label{phase-bound-proof}
    \abs{\Phi}\gtrsim 2^{k_{\min}^- + 2 k_{\max\vphantom{\min}}^-}.
\end{align}
Using this, we can conclude the claim in most of the cases by a combination of set-size (Lemma \ref{Lem-bilinear-estimates}) and linear decay estimates (Proposition \ref{prop:lin-decay}). 
\par We proceed with the cubic terms in \eqref{eqn:normal-form-decomposition}, focusing first on $\mathcal{B}_{\m\Phi^{-1}}(\partial_tf_1,f_2)$.
We may bound it using an $L^2-L^\infty$ estimate in Lemma \ref{Lem-bilinear-estimates}, Lemma \ref{lem:time-derivative-L2} on $\partial_tf_1$, the phase bound \eqref{phase-bound-proof} and the multiplier bound \eqref{main-multiplier-bound}. Altogether we obtain under the bootstrap assumption \eqref{eq:bootstrap-assumptionX}:
\begin{align}
    2^{8k^+}2^{(1+\beta)k+\gamma k^-}2^{(1+\beta)m}&\norm{Q_{jk}\mathcal{B}_{\m\Phi^{-1}}(\partial_tf_1,f_2)}_{L^2}\\
    &\lesssim  2^{8k^+}2^{(1+\beta)k+\gamma k^-}2^{(2+\beta)m}\norm{\m\Phi^{-1}\chi}_{S^\infty}\norm{\partial_tf_1}_{L^2}\norm{e^{it\Lambda}f_2}_{L^\infty}\\ 
    &\lesssim 2^{8k^+}2^{(1+\beta)k+\gamma k^-}2^{(2+\beta)m}2^{k_1+k_2}2^{-k_{\min}^--2k_{\max\vphantom{\min}}^-}2^{-\frac{3}{2}m}\eps^2 2^{-(1+\beta)m}2^{-(\frac{9}{2}+2\beta)k_2^+}\eps\\
    &\lesssim 2^{-\frac{m}{2}}2^{10\delta_0m}2^{-(\frac{1}{6}+\frac{2}{3}\beta)k_{\min}^-}\eps^3.
\end{align}
In the last inequality we have used \eqref{eqn:gamma-beta-relation}. Now by the restriction \eqref{eqn:k-restrictions-j<m}, $k_{\min}\gtrsim -(1-4\delta)m$, which implies $-(\frac{1}{6}+\frac{2}{3}\beta)k_{\min}^-\lesssim (\frac{1}{6}+\frac{2}{3}\beta+\delta)m$. Together with Remark \ref{rem:parameters}, in particular $\delta_0\ll\beta<\frac{1}{5}$, we obtain
\begin{align}
     2^{8k^+}&2^{(1+\beta)k+\gamma k^-}2^{(1+\beta)m}\norm{Q_{jk}\mathcal{B}_{\m\Phi^{-1}}(\partial_tf_1,f_2)}_{L^2}\lesssim 2^{(-\frac{1}{3}+\frac{2}{3}\beta+\delta+10\delta_0)m}\eps^3\lesssim 2^{(-\frac{1}{5}+2\delta) m}\eps^3,
\end{align}
which is more than enough to obtain the claim. The other cubic term $\mathcal{B}_{\m\Phi^{-1}}(f_1,\partial_tf_2)$ in \eqref{eqn:normal-form-decomposition} follows symmetrically using Lemma \ref{lem:time-derivative-L2} on $\partial_tf_2$ instead. 
\par It remains to estimate the boundary term in \eqref{eqn:normal-form-decomposition}. Note that the most involved case is when all frequencies are small (i.e.\ $<1$) since the contribution coming from $\abs{\Phi}^{-1}$ via \eqref{phase-bound-proof} might be very large. Thus we may assume w.l.o.g.\ $k,k_1,k_2<0$ such that $k=k^-,\; k_i=k_i^-$. Moreover, since the estimates are symmetric in $k_1,k_2$, we may assume w.l.o.g.\ that $k_1\leq k_2$. We proceed by splitting the frequency space according to the relative sizes of the input and output frequencies in two main cases using the notation of Remark \ref{rem:notation-rel-size-frequencies}.

\par\textbf{\underline{Case 1: $k=k_{\min}$, $k_1\sim k_2\sim k_{\max}$.}} From \eqref{phase-bound-proof} there holds $\abs{\Phi}\gtrsim 2^{k+2k_1}$.
An $L^2-L^2$ bound using Lemma \ref{Lem-bilinear-estimates} and the bootstrap assumption \eqref{eq:bootstrap-assumptionX} allows to restrict $j_{\min}$ in terms of $m$, while obtaining a gain in the smallest parameter $k$. Indeed
\begin{align}
    2^{(1+\beta)m}2^{(1+\beta+\gamma)k}\norm{Q_{jk}\mathcal{Q}_{\m\Phi^{-1}}(f_1,f_2)}_{L^2}&\lesssim 2^{(1+\beta)m}2^{(1+\beta+\gamma)k} 2^{\frac{3}{2}k}\norm{\m\Phi^{-1}\chi}_{S^\infty}\norm{f_1}_{L^2}\norm{f_2}_{L^2}\\
    &\lesssim 2^{(1+\beta)m}2^{(\beta+\gamma)k}2^{\frac{3}{2}k}2^{-(1+\beta)j_{\min}-(1+\beta+2\gamma)k_1}\eps^2 \\
    & =2^{(1+\beta)m}2^{\frac{4-2\beta}{3}k}2^{-(1+\beta)j_{\min}-(\frac{2-7\beta}{3})k_1}\eps^2,
\end{align}
where we have used the identity \eqref{eqn:gamma-beta-relation}. 
The claim then follows if $-j_{\min}\leq -\left(1+\frac{\delta}{1+\beta}\right)m-\frac{{4}-{2}\beta}{3(1+\beta)}k+\frac{{2}-{7}\beta}{3(1+\beta)}k_1$. Otherwise there holds:
\begin{align}
    j_{\min}\leq \left(1+\frac{\delta}{1+\beta}\right)m+\frac{{4}-{2}\beta}{3(1+\beta)}k-\frac{{2}-{7}\beta}{3(1+\beta)}k_1.
\end{align}
In this configuration, we use an $L^\infty-L^2$ bound in Lemma \ref{Lem-bilinear-estimates} and the linear decay \eqref{eqn:linear-decay-j<m} :  
\begin{align}
    2^{(1+\beta)m}2^{(1+\beta+\gamma)k}\norm{Q_{jk}\mathcal{Q}_{\m\Phi^{-1}}(f_1,f_2)}_{L^2} & \lesssim 2^{(1+\beta)m}2^{(\beta+\gamma)k} 2^{-\frac{3}{2}m}2^{(\frac12-\beta)j_{\min}-(\frac12+\beta+\gamma)k_1}  \\
    & \qquad \qquad \cdot   2^{-\gamma k_1} 2^{-(1+\beta)(j_{\min}+k_1)} 2^{-2\delta(j_{\min}+k_1)} 2^{-(1+\beta-2\delta)(j_{\min}+k_1)}\eps^2\\
    &\lesssim 2^{(1+\beta)m}2^{(\beta+\gamma)k}2^{-\frac{3}{2}m}2^{(\frac12-\beta)j_{\min}-(\frac12+\beta+\gamma)k_1}2^{-2\delta j_{\min}-(\gamma+2\delta)k_1}\eps^2\\
    &\lesssim 2^{\big(-2\delta +\frac{\frac12-\beta-2\delta}{1+\beta}\delta \big)m}2^{\frac{1-5\beta-\frac{8}{3}(2-\beta)\delta}{2(1+\beta)}k}2^{\frac{-1+8\beta-\frac{4}{3}(1+10\beta)\delta}{2(1+\beta)}k_1}\eps^2.
\end{align}
By Remark \ref{rem:parameters}, since $\beta\leq \frac15-6\delta$, the coefficient of $k$ is positive, and because $k\leq k_1\leq 0$ we obtain the claim \eqref{eqn:j-upperbound-standard}:
\begin{align}
    2^{(1+\beta)m}2^{(1+\beta+\gamma)k}\norm{Q_{jk}\mathcal{Q}_{\m\Phi^{-1}}(f_1,f_2)}_{L^2} & \lesssim 2^{-\frac{3}{2}\delta m}2^{\frac{3\beta-\frac{4}{3}(5+8\beta)\delta}{2(1+\beta)}k}\eps^2\lesssim 2^{-\frac{3}{2}\delta m}\eps^2.
\end{align}

\par\textbf{\underline{Case 2: $k_1=k_{\min}$, $k\sim k_2\sim k_{\max}$.}} The bound on the phase  from \eqref{phase-bound-proof} reads:\begin{align}
    \abs{\Phi}\gtrsim 2^{k_1+2k}.
\end{align}
By an $L^\infty-L^2$ bound in Lemma \ref{Lem-bilinear-estimates} using \eqref{eqn:decay-Pk-profiles} on $f_1$ and the bootstrap assumption \eqref{eq:bootstrap-assumptionX} we obtain:
\begin{align}
    2^{(1+\beta)m}2^{(1+\beta+\gamma)k}\norm{Q_{jk}\mathcal{Q}_{\m\Phi^{-1}}(f_1,f_2)}_{L^2}&\lesssim 2^{(1+\beta)m}2^{(1+\beta+\gamma)k}\norm{\m\Phi^{-1}\chi}_{S^\infty}\norm{e^{it\Lambda}f_1}_{L^\infty} \norm{f_2}_{L^2} \\
    &\lesssim 2^{(1+\beta)m}2^{(1+\beta+\gamma)k}  2^{k_1+k}2^{-k_1-2k}2^{-(1+\beta)m}\eps 2^{-j_2-k_2-\gamma k_2}\eps\\
    &\lesssim 2^{-j_2-(1-\beta)k}\eps^2,
\end{align}
so the claim follows if $-j_2-(1-\beta)k\leq -2\delta m$. Otherwise $j_2<2\delta m -(1-\beta)k$ and we may use an appropriate balance of set-size estimates and improved $L^\infty-L^2$ decay (using $L^\infty$ \eqref{eqn:linear-decay-j<m} on $f_2$) to obtain:
\begin{align}
   2^{(1+\beta)m}2^{(1+\beta+\gamma)k}&\norm{Q_{jk}\mathcal{Q}_{\m\Phi^{-1}}(f_1,f_2)}_{L^2}\\&\lesssim 
   2^{(1+\beta)m}2^{(1+\beta+\gamma)k}\norm{\m\Phi^{-1}\chi}_{S^\infty}\norm{f_1}_{L^2}\min\left\{\norm{e^{it\Lambda}f_2}_{L^\infty}, 2^{\frac{3}{2}k}\norm{f_2}_{L^2} \right\} 
   \\&\lesssim 2^{(1+\beta)m}2^{(\beta+\gamma)k} 2^{-\gamma k_1}  \min\left\{2^{-\frac{3}{2}m}2^{(\frac{1}{2}-\beta)j_2}2^{-(\frac{1}{2}+\beta+\gamma)k_2}, 2^{\frac{3}{2}k}2^{-\gamma k}\right\}\eps^2\\
   &\lesssim 2^{(1+\beta)m}2^{\beta k}  \cdot\min\left\{2^{-(\frac{3}{2}-2(\frac{1}{2}-\beta)\delta)m} 2^{(-(\frac{1}{2}-\beta)(1-\beta)-(\frac{1}{3}-\frac{2}{3}\beta))k}, 2^{(\frac{5}{3}+\frac{2}{3}\beta)k} \right\}\eps^2\\
   &\lesssim 2^{(1+\beta)m}2^{\beta k} \cdot\min\left\{2^{-(\frac{3}{2}-{\delta})m} 2^{(-\frac{5}{6} + \frac{13}{6} \beta- \beta^2)k},2^{(\frac{5}{3}+\frac{2}{3}\beta)k}\right\}\eps^2
\end{align}
In particular, the desired decay in $m$ and a positive factor of $k$ is achieved by weighting the first term in the minimum above with $\theta$ and the second one with $1-\theta$ with
\begin{align}
   \theta=\frac{1+\beta+{2}\delta}{\frac{3}{2}-{\delta}}\in\,]0,1[,\qquad  1-\theta=\frac{\frac{3}{2}-{\delta}-1-\beta-{2}\delta}{\frac{3}{2}-{\delta}}=\frac{\frac{1}{2}-\beta-{3\delta}}{\frac{3}{2}-{\delta}},
\end{align} 
as long as $\delta$ as in  Remark \ref{rem:parameters}. We obtain that the boundary term verifies the claim \eqref{eqn:j-upperbound-standard}:
\begin{align}  2^{(1+\beta)m}2^{(1+\beta+\gamma)k}\norm{Q_{jk}\mathcal{Q}_{\m\Phi^{-1}}(f_1,f_2)}_{L^2}\lesssim 2^{-2\delta m}\eps^2.
\end{align}

\par The case $k_2=k_{\min}$ is completely symmetric to the one above with the roles of $k_1\leftrightarrow k_2$ interchanged and $k_2\sim k$. In particular, by Lemma \ref{LM:bound-from-below-PHASE}, $\abs{\Phi}\gtrsim 2^{k_2+2k}$ and the combined dispersive decay \eqref{eqn:linear-decay-j<m} and set-size estimates yield the bound. This finishes the proof of the proposition.
\end{proof}

\section{Phase and Multiplier Bounds}
\label{sec:auxiliary-results}

\subsection{Phase Bounds}
In this section we prove the main bounds on the phases that allow to perform the normal form in Proposition \ref{prop:X-norm-j<m}.
\begin{lem}[Lower bounds on the phase]\label{LM:bound-from-below-PHASE}
Let $\Phi\in\{\Phi_{\mu\nu}^\pm=\pm\Lambda(\xi)+\mu\Lambda(\xi-\eta)+\nu\Lambda(\eta)\}$ be one of the phases of the problems defined in \eqref{def:phases}. For any $(k,k_1,k_2) \in \Z^3$, then on the support of $\chi$ (defined in \eqref{eqn: chi localizations}) there holds that
\begin{align}
    \abs{\Phi}\gtrsim 2^{k_{\min}^- +2k_{\max\vphantom{\min}}^-}.
\end{align}
\end{lem}

\begin{proof} We may consider without loss of generality $\Phi_{\mu \nu} \vcentcolon=\Phi^+_{\mu\nu}$ due to the following symmetry:
\begin{align}
    \Phi^-_{++}(\xi,\eta)=-\Phi^+_{--}(\xi,\eta), \quad \Phi^{-}_{--}(\xi,\eta)=-\Phi^+_{++}(\xi,\eta),\quad \Phi^-_{+-}(\xi,\eta)=-\Phi^+_{+-}(\xi,\xi-\eta).
\end{align}
Furthermore, since $\Phi_{--}(\xi,\eta)=-\Phi_{+-}(\eta,\xi)$, it is enough to consider only $\Phi_{++}$ and $\Phi_{+-}$. In the following we will use the fact that the dispersion relation is radial by defining: \begin{align}
    \Lambda(z)= q(\vert z \vert),\quad  \forall z\in \R^3,\qquad q:\R^+\to \R^+, r\mapsto \frac{r}{\br{r}}.
\end{align} We note the following useful computations
\begin{align}
    q'(r)=\frac{1}{\br{r}^3}, \qquad q''(r)=-\frac{3r}{\br{r}^5}.
\end{align}
Moreover, observe that $q$ is strictly increasing, while $q'$ is strictly decreasing on $\R_+$.
\par For the phase $\Phi_{++}$ the computation is direct with a larger lower bound: \begin{align}
    \vert \Phi_{++} \vert &=q(|\xi|)+q(|\xi-\eta|)+q(|\eta|)\geq \max\{ q(|\xi|),q(|\xi-\eta|),q(|\eta|)\}\geq 2^{k_{\max}-k_{\max}^+}=2^{k_{\max}^-}\geq 2^{k_{\min}^-}.
\end{align}
We therefore consider the phase $\Phi_{+-}$ in what follows. We separate cases according to the relative sizes of the frequency localization parameters using the notation in Remark \ref{rem:notation-rel-size-frequencies}.

\medskip

  \par \underline{I. $2^{k_1} \sim 2^k \gg 2^{k_2}$.} Here the computation is straightforward with a larger bound since $q(|\xi-\eta|)>q(|\eta|)$:
  \begin{align}
      \abs{\Phi_{+-}(\xi, \eta)} \geq q(|\xi|)+q(|\xi-\eta|)-q(|\eta|)\geq q(|\xi|)\sim 2^{k^-}\geq 2^{k_{\min}^-}.
  \end{align}

\medskip

\par\underline{II. $2^{k_1}\sim 2^{k_2}\gg 2^{k}$} Since $\vert \xi-\eta \vert \geq \vert \eta \vert -\vert \xi \vert \geq 0$ and since $q$ is increasing there holds
$$\Phi_{+-}(\xi, \eta) \geq q(\vert \xi \vert)+q(\vert \eta \vert -\vert \xi \vert)-q(\vert \eta \vert).$$
Let $a=\vert \xi \vert \sim 2^k$ and $b=\vert \eta \vert \sim 2^{k_2}$. Then, using Remark \ref{rem:notation-rel-size-frequencies} there holds  
$$\Phi_{+-}(\xi, \eta) \geq F_b(a), \qquad F_b(a)\vcentcolon=q(a)+q(b-a)-q(b), \ \ \text{for} \ \ 0<a \leq 2^{-4}b.$$
Since $a<b-a$ and $q'$ is strictly decreasing there holds $F'_b(a)=q'(a)-q'(b-a)>0$. We consider two cases:

\underline{II.1. Let $b \geq 1$.} By mean-value theorem, there exists $\theta \in\, ]b-a, b[\subset \,](1-2^{-4})b, b[$ such that $$F_b(a)=q(a)-aq'(\theta).$$
Since $q'$ is decreasing, there holds in fact
\begin{align*}
    aq'(\theta) \leq a q'((1-2^{-4})b)=\frac{a}{(1+(1-2^{-4})^2b^2)^{3/2}}.
\end{align*}
As a consequence, we have
\begin{align*}
    \frac{aq'(\theta)}{q(a)} \leq \frac{(1+a^2)^{1/2}}{(1+(1-2^{-4})^2b^2)^{3/2}} \leq \frac{(1+2^{-8}b^2)^{1/2}}{(1+(1-2^{-4})^2b^2)^{3/2}}.
\end{align*}
The right-hand side above is a continuous and non-increasing on $[1, \infty[$ function of $b$, reaching the value $0.39$ at $b=1$, hence there exists a universal constant $0<m<1$ such that
$$aq'(\theta) \leq m q(a).$$
Thus we obtain:
$$ \vert \Phi_{+-}(\xi, \eta) \vert \geq F_b(a) \geq (1-m)q(a)=(1-m) \frac{a}{\langle a \rangle} \sim 2^{k-k^+}=2^{k^-}\geq 2^{k^-_{\min}+2k_{\max\vphantom{\min}}^-}.$$ 

\underline{II.2. Let $b < 1$.} We write
$$ F_b(a)=\int_0^a F_b'(s) \, \mathrm{d}s=\int_0^a (q'(s)-q'(b-s)) \, \mathrm{d}s,$$
and we want to bound the integrand from below. For $s \in [0,a]$ (satisfying $s \leq b-s$ since $s \leq 2^{-1}b$), we further write
$$ q'(s)-q'(b-s)=-\int_s^{b-s} q''(u) \, \mathrm{d}u=\int_s^{b-s} \frac{3u}{(1+u^2)^{5/2}} \, \mathrm{d}u. $$
We observe that $s \leq a\leq 2^{-4}b \leq b/4$ and $b/2 \leq b-s$ therefore 
$$ q'(s)-q'(b-s)=\int_s^{b-s} \frac{3u}{(1+u^2)^{5/2}} \, \mathrm{d}u\geq \int_{b/4}^{b/2} \frac{3u}{(1+u^2)^{5/2}}\, \mathrm{d}u \gtrsim \int_{b/4}^{b/2} u\, \mathrm{d}u,$$
since the denominator is bounded as $ (1+u^2)^{5/2} \lesssim (1+b^2/4)^{5/2} \lesssim 1$ for $u \in [b/4, b/2]$ with $b < 1$. All in all, we have obtained
$$q'(s)-q'(b-s) \geq \int_{b/4}^{b/2} u\, \mathrm{d}u \gtrsim b^2,$$
hence the final bound
$$F_b(a) = \int_0^a (q'(s)-q'(b-s)) \, \mathrm{d}s \gtrsim ab^2.$$
Thus, we have proved
$$ \Phi_{+-}(\xi, \eta) \geq 2^k 2^{2k_1}=2^{k^- +2k_1^-}.$$
\medskip
  \par\underline{III. $2^{k_2} \sim 2^k \gg 2^{k_1}$.} 
   We proceed in the same way as in the previous case. Since $\vert \xi \vert \geq  \vert \eta \vert-\vert \xi-\eta \vert \geq 0$ and $q$ is increasing, we have
   $$\Phi_{+-}(\xi, \eta) \geq q(\vert \xi-\eta  \vert)+q(\vert \eta \vert-\vert \xi-\eta \vert)-q(\vert \eta \vert).$$
Now setting $a=\vert \xi-\eta  \vert \sim 2^{k_1}$ and $b=\vert \eta \vert \sim 2^{k}$,  one needs  to bound from below the function
$$F_b(a)\vcentcolon=q(a)+q(b-a)-q(b), \ \ \text{for} \ \ 0<a \leq 2^{-4}b.$$
The bound is obtained similarly to case II above.

   \medskip
   \par\underline{IV. $2^{k}\sim 2^{k_1}\sim 2^{k_2}$}.
Since $q$ is increasing, concave and vanishing at zero, and $\abs{\eta}\leq \abs{\xi}+\abs{\xi-\eta}$ there holds:
\begin{align*}
    \Phi_{+-}(\xi, \eta)= \Lambda(\xi)+\Lambda(\xi-\eta)-\Lambda(\eta)
    &\geq  q( \vert \xi\vert)+q(\vert \xi-\eta \vert)-q(\vert \xi \vert+\vert \xi-\eta \vert) \geq 0.
\end{align*}

Letting $a\sim 2^k$ and $b\sim 2^{k_1}$ (with $\abs{k-k_1}\leq 4$) and $c=a+b$, we have
$$ a=t c, \ \ b=(1-t)c, \ \ t=\frac{a}{a+b} \in ]0,1[.$$
Thus, we have
\begin{align*}
   \Phi_{+-}(\xi,\eta)\geq q_t(c), \qquad q_t(c)\coloneqq q(tc)+q((1-t)c)-q(c), \ \ c>0.
\end{align*}
We again consider two cases:
\par \underline{IV.1. $c \leq 1$.} By Taylor's formula we have that
$$ c-\frac{1}{2}c^3 \leq q(c) \leq c-\frac{1}{2}c^3+\frac{3}{8}c^5,$$
so we obtain
\begin{align*}
    q_t(c) \geq -\frac{1}{2}(t^3+(1-t)^3-1)c^3-\frac{3}{8}c^5=\frac{3}{2}t(1-t)c^3-\frac{3}{8}c^5.
\end{align*}
Note that in this configuration, $t(1-t)$ is uniformly bounded from below. Indeed, since $a$ and $b$ are comparable, we have $2^{-4} \leq \frac{a}{b} \leq 2^4$ therefore $t=\frac{a}{a+b}=\frac{a/b}{1+a/b}$ is bounded as 
\begin{align*}
    0<\frac{1}{1+2^4}\leq t \leq \frac{2^4}{1+2^4}<1, \qquad t(1-t)\geq (1+2^4)^{-2}.
\end{align*}
Thus, we obtain the claim  $\vert \Phi_{+-}(\xi, \eta) \vert \geq q_t(c) \gtrsim  2^{3 k}$,
by taking $c \leq c_0<1$. This leaves the last case.
\par \underline{IV.2. Let $c \geq c_0$.} We observe that since $q'$ is decreasing, we have
\begin{align*}
    q_t'(c)=t q'(tc)+(1-t) q'((1-t)c)-q'(c)  \geq tq'(c)+(1-t) q'(c)-q'(c)=0.
\end{align*}
So $q_t$ is non-decreasing and in particular $q_t(c) \geq p_t(c_0)$ if $c \geq c_0$. By continuity, since $t\in [(1+2^{4})^{-1},2^4(1+2^{4})^{-1}]\subset]0,1[$, there holds $q_t(c) \geq q_{t^*}(c_0)>0$ for some $t^* \in ]0,1[$, which is a uniform bound from below. All in all, we have proven that 
\begin{align*}
    \vert \Phi_{+-}(\xi, \eta) \vert \gtrsim \min\{1,2^{3k}\}=2^{3k^-}.
\end{align*}
This finishes the proof.
\end{proof}

\subsection{Bilinear Estimates and Multiplier Bounds}
Recall that for $t>0$, we consider localized bilinear operators with $\m, \Phi$ defined in \eqref{def:multipliers},\eqref{def:phases}:
\begin{align}\label{bilin-op-appendix}
    \mathcal{F}(P_k\mathcal{Q}_{\m}(P_{k_1}f_1,P_{k_2}f_2))(\xi)=\int_{\R^3} e^{it\Phi(\xi,\eta)}\m(\xi,\eta)\chi(\xi,\eta)\widehat{f_1}(\xi-\eta)\widehat{f_2}(\eta)d\eta,
\end{align}
 In Lemma \ref{Lem-bilinear-estimates} below we prove H\"older-type estimates on the bilinear objects above. For this, we define the following norm for $m\in \mathscr{C}_c(\R^3\times \R^3,\C)$:
\begin{align*}
    \norm{m}_{\mathcal{S}^\infty} \coloneqq\norm{\mathcal{F}^{-1}(m) }_{L^1_{\xi, \eta}}, \qquad S^\infty\hookrightarrow L^\infty.
\end{align*}

\begin{lem}[Bilinear estimates]\label{Lem-bilinear-estimates}
  For a bilinear operator as in \eqref{bilin-op-appendix} there holds:
  \begin{enumerate}
      \item $ \norm{ P_{k}\mathcal{Q}_\m(P_{k_1}f_1,P_{k_1}f_2) }_{L^2}\lesssim  \norm{\m \chi}_{\mathcal{S}^\infty} \min \left\{\norm{P_{k_1}e^{it\Lambda}f_1}_{L^\infty}\norm{P_{k_2}f_2}_{L^2},\norm{P_{k_1}f_1}_{L^2}\norm{P_{k_2}e^{it\Lambda}f_2}_{L^\infty}\right\}$;
      \item \label{set-size-esitmate} $ \norm{ P_{k}\mathcal{Q}_\m(P_{k_1}f_1,P_{k_1}f_2) }_{L^2}\lesssim  \norm{\m \chi}_{\mathcal{S}^\infty}2^{\frac{3}{2}k_{\min} }\norm{P_{k_1}f_1}_{L^2}\norm{P_{k_2}f_2}_{L^2}$.
  \end{enumerate}
\end{lem}
\begin{proof}
The proof of the first statement follows by standard application of H\"older estimates in the following terms (see e.g.\ \cite[\textcolor{MidnightBlue}{Lemma 5.2}]{IP15-WW}):
\begin{align}
   \mathcal{F}^{-1}\Big(\int_{\R^3}e^{it\Phi}\m(\xi,\eta)\chi(\xi,\eta)\widehat{f_1}(\xi-\eta)&\widehat{f_2}(\eta)d\eta\Big)(x)\\
   &=\int_{\R^3}e^{ix\cdot\xi}\int_{\R^3}e^{it\Phi}\m(\xi,\eta)\chi(\xi,\eta)\widehat{f_1}(\xi-\eta)\widehat{f_2}(\eta)d\eta d\xi\\
   &=\int_{\R^3}\int_{\R^3}\mathcal{F}^{-1}_{\xi,\eta}(\m\chi)(y,z)(e^{it\Lambda}f_1)(x-y)(e^{it\Lambda}f_2)(x-y-z)dy dz.
\end{align}
Now we prove the set-size estimate. Let $g\in L^2$ and assume without loss of generality that $k_{\min}=k$ (otherwise exchange variables $\eta\leftrightarrow\xi-\eta$). Then there holds:
\begin{align}
    \abs{\br{\mathcal{Q}_{\m}(f_1,f_2),g}}&=\abs{\iint e^{it\Phi}\m(\xi,\eta)\chi(\xi,\eta)\widehat{f_1}(\xi-\eta)\widehat{f_2}(\eta)d\eta \; \widehat{g}(\xi)d\xi}\\
    &\lesssim  \norm{\m\chi}_{L^\infty_{\xi,\eta}}\norm{\widehat{g}(\xi)\widehat{f_1}(\xi-\eta)}_{L^2_{\xi,\eta}}\norm{\tilde{\varphi}(2^{-k}\abs{\xi})\tilde{\varphi}(2^{-k_1}\abs{\xi-\eta})\widehat{f_2}(\eta)}_{L^2_{\xi,\eta}}\\
    &\lesssim \norm{\m\chi}_{S^\infty}\norm{g}_{L^2}\norm{f_1}_{L^2}2^{\frac{3}{2}k}\norm{f_2}_{L^2}.   
\end{align}
Note that $\tilde{\varphi}$ is a slightly modified bump function with similar support properties and in the last line we have used $S^\infty\hookrightarrow L^\infty$.
\end{proof}
We prove the multiplier bounds used throughout the paper.
\begin{lem}[Multiplier bounds] \label{lem:multiplier-bounds} We record the following useful symbol bounds.
   
    \begin{enumerate}
        \item For two symbols $m^1,m^2$ there holds : $\norm{m^1m^2}_{S^\infty}\lesssim \norm{m^1}_{S^\infty}\norm{m^2}_{S^\infty}$.
        \item  Let $\m:\R^3\times\R^3\to \C$ be one of the multipliers in \eqref{def:multipliers} and $\chi$ as in \eqref{eqn: chi localizations}. Then there holds: \begin{align}\label{main-multiplier-bound}
        \norm{\m\chi}_{S^\infty}\lesssim 2^{k_1+k_2}. 
    \end{align}
    \end{enumerate}
    
\end{lem}
\begin{proof}
The algebra property in the first statement is a standard result, see e.g.\  
\cite[\textcolor{MidnightBlue}{Lemma 3.2}]{DIP17-Euler-Maxwell-2D}. The second statement follows as an adaptation of \cite[\textcolor{MidnightBlue}{Lemma 3.3}]{DIP17-Euler-Maxwell-2D} to the $3-$dimensional case and to the localizations $\chi$. In particular, by \eqref{def:multipliers} there holds:
\begin{align}
    \norm{\m\chi}_{S^\infty}=\norm{\mathcal{F}^{-1}(\m\chi)}_{L^1_{x,y}}\lesssim \sum_{l=0}^4\left(2^{l(k+k_1)}\norm{\partial_\xi^l\m}_{L^\infty}+2^{l(k_2+k_1)}\norm{\partial_\eta^l\m}_{L^\infty}\right)\lesssim 2^{k_1+k_2}.
\end{align}
The second inequality above follows by an integration by parts (in $\xi,\eta$) argument, see the proof of \cite[\textcolor{MidnightBlue}{Lemma 3.3}]{DIP17-Euler-Maxwell-2D}.
\end{proof}
\subsection{Integration by Parts in Frequency}\label{sec:ibp}
We present several results allowing to perform integration by parts in bilinear terms of the form \eqref{def:bilinear-expressions} and used in the proof of Proposition \ref{prop:finite-speed-of propagation}. We start by a general integration by parts Lemma, for proof see e.g.\ \cite[\textcolor{MidnightBlue}{Lemma 5.4}]{IP14-KG-3D}.
\begin{lem}\label{lemma-IBPgen}
    Let $\epsilon \in (0,1)$, $K>0$, $M \in \N \setminus \lbrace 0 \rbrace$ such that $\eps K \geq 1$ and $F,g \in C^M(\R^d)$. If $F$ is real-valued and satisfies
    \begin{align*}
        \vert \nabla F \vert \geq \mathbf{1}_{\mathrm{supp}(g)}, \quad \forall 2 \leq \vert \alpha \vert \leq M, \qquad \vert {D}^\alpha F \vert \lesssim_M \epsilon^{1-\vert \alpha \vert}.
    \end{align*}
    Then we have
    \begin{align*}
        \abs{\int_{\R^d} e^{iKF}g  } \lesssim\frac{1}{(\epsilon K)^M} \sum_{\vert \alpha \vert \leq M } \epsilon^{\vert \alpha \vert} \norm{{D}^\alpha g}_{L^1}.
    \end{align*}
\end{lem}
We want to apply this general result on the usual bilinear terms. In particular, $g$ will be given in terms of our multipliers, profiles and localization functions and we need iterated bounds on these objects. We prove several results regarding such iterated bounds in the Lemma below and gather the final estimates in Lemma \ref{lem:ibp}.
\begin{lem}[Iterated bounds \uppercase\expandafter{\romannumeral 1}] \label{lem-iterated-bounds-phase-multiplier}
    Let $\m$ be one of the multipliers in \eqref{def:multipliers}, $\Phi$ one of the phases in \eqref{def:phases} and $\chi$ the localization as in \eqref{eqn: chi localizations}. Then the following iterative bounds hold:
    \begin{align}
        \abs{D^{\alpha}_\xi\m\chi}&\lesssim 2^{k_{\max}}2^{-\abs{\alpha}\min\{k,k_1\}},
         \label{lem:it-multiplier-bounds}\\
          \abs{D^{\alpha}_\xi\Phi\chi}&\lesssim \max\left\{2^{(1-|\alpha|)k^-}2^{-(|\alpha|+2)k^+},2^{(1-|\alpha|)k_1^-}2^{-(|\alpha|+2)k_1^+}\right\}.\label{lem:it-phase-bounds}
    \end{align}
\end{lem}
\begin{proof}
 We prove the iterative bound on the multipliers. To begin with, record the useful formula:
\begin{align}
    D^\alpha_\xi \langle\xi\rangle^{-1}=\sum_{l=1}^{\abs{\alpha}}\langle\xi\rangle^{-(2l+1)} P_{l,\alpha}(\xi),
\end{align}
where $P_{l,\alpha}(\xi)$ is a polynomial in $\xi_{i}$, $i=1,2,3$, of degree $2l-\abs{\alpha}$. This is easily proven by induction. The first order derivatives clearly satisfy this since $\partial_{\xi_i}\langle\xi\rangle^{-1}={-\xi_i}\langle \xi \rangle^{-3}$.
We prove the induction step: assuming the formula above holds, we compute for the $l+1$ order derivative:
\begin{align}\label{eqn:inductive-formula-bracket-inverse}
    \partial_{\xi_i}\sum_{l=1}^{\abs{\alpha}}\langle\xi\rangle^{-(2l+1)} P_{l,\alpha}(\xi)&=\sum_{l=1}^{\abs{\alpha}}-(2l+1)\langle\xi\rangle^{-(2l+3)}\xi_iP_{l,\alpha}(\xi)+ \langle\xi\rangle^{-(2l+1)}\partial_iP_{l,\alpha}(\xi)\\
    &=\sum_{l=1}^{\abs{\alpha}+1}\langle\xi\rangle^{-(2l+1)} P_{l,\alpha+1}(\xi).
\end{align}
The last line follows since $\xi_iP_{l,\alpha}(\xi)$ is a polynomial of degree at most $2l-|\alpha|+1$ while in the second term, the polynomial is of degree $2l-|\alpha|-1$. Hence, when shifting the indices $l\rightarrow l+1$, we obtain $P_{l,|\alpha|+1}$ is of order $2l-(|\alpha|+1)$. In particular, this implies the following bound on the support of $\chi$:
\begin{align}\label{eqn:iterative-bound-D_eta-bracket}
    \abs{D_\xi^\alpha\langle \xi\rangle^{-1}}\lesssim_{\alpha} 2^{-(\abs{\alpha}+1)k^+}.
\end{align}
We consider each term in the multipliers \eqref{def:multipliers}. To begin with, we compute
\begin{align}
    \abs{D^{\alpha}_\xi\left( \frac{\xi\cdot\eta}{\abs{\xi}\langle\eta\rangle}\abs{\xi-\eta}\right)}=\frac{\abs{\eta}}{\br{\eta}}\abs{D^\alpha_\xi\left( \frac{\xi\cdot\eta}{\abs{\xi}\abs{\eta}}\abs{\xi-\eta}\right)}
    \lesssim \sum_{|\beta|\leq |\alpha|}\abs{D^\beta_\xi\left( \frac{\xi\cdot\eta}{\abs{\xi}\abs{\eta}}\right)}\abs{D^{\alpha-\beta}_\xi \abs{\xi-\eta}} \lesssim 2^{(1-\abs{\alpha})\min\{k,k_1\}}.
\end{align}
Moreover, using \eqref{eqn:iterative-bound-D_eta-bracket}, we have
\begin{align}
     \abs{D^{\alpha}_\xi\left(\frac{\xi\cdot(\xi-\eta)}{|\xi|\langle\xi-\eta\rangle}\abs{\eta} \right)}&\lesssim \abs{\eta}\sum_{\abs{\beta}\leq \abs{\alpha}}\abs{D^{\beta}\left(\frac{\xi\cdot(\xi-\eta)}{|\xi|\abs{\xi-\eta}}\right)}\sum_{\abs{\gamma}\leq \abs{\alpha-\beta}}\abs{D^\gamma \abs{\xi-\eta}}\abs{D^{\alpha-\beta-\gamma}\br{\xi-\eta}^{-1}}\\
     &\lesssim \abs{\eta}2^{-\abs{\alpha}\min\{k,k_1\}} \\
     &\lesssim 2^{k_{\max}}2^{-\abs{\alpha}\min\{k,k_1\}}.
\end{align}
Finally using \eqref{eqn:iterative-bound-D_eta-bracket}, we obtain
\begin{align}
    \abs{D_\xi^{\alpha}\left(\langle\xi\rangle \frac{(\xi-\eta)\cdot\eta}{\langle{\xi-\eta}\rangle\langle{\eta}\rangle}\right)}&\lesssim \sum_{\abs{\beta}\leq \abs{\alpha}} \abs{D^\beta\br{\xi}\sum_{\abs{\gamma}\leq \abs{\alpha-\beta}}D^{\gamma} \left(\frac{(\xi-\eta)\cdot\eta}{\abs{\xi-\eta}\abs{\eta}}\right)\sum_{\abs{\delta}\leq\abs{\alpha-\beta-\gamma}}\br{\xi-\eta}^{-1}} 
    \\
    & \lesssim 2^{(1-\abs{\alpha}) \min\{k,k_1\}}.
\end{align}
Altogether we get 
\begin{align}
   \abs{ D^\alpha_\xi \m\chi}\lesssim 2^{k_{\max}}2^{-\abs{\alpha}\min\{k_1,k_2\}}.
\end{align}
 \par We prove the iterated bound on the phase in a similar way. We compute 
    \begin{align}
        \nabla_\xi\Phi(\xi,\eta)=\pm \nabla_\xi\Lambda(\xi)\pm \nabla_\xi\Lambda(\xi-\eta)=\pm \frac{\xi}{\abs{\xi}\langle\xi\rangle^{3}}\pm \frac{\xi-\eta}{\abs{\xi-\eta}\langle\xi-\eta\rangle^3}.
    \end{align}
    Moreover, using \eqref{eqn:iterative-bound-D_eta-bracket} we have the bounds
    \begin{align}
        \abs{D^{\alpha}_\xi|\xi|\langle\xi\rangle^{-1}}\lesssim\abs{\sum_{\abs{\beta}\leq \abs{\alpha}}D^{\beta}_\xi\abs{\xi}D^{{\alpha-\beta}}_\xi\br{\xi}^{-1}}\lesssim \sum_{\abs{\beta}\leq \abs{\alpha}}\abs{\xi}^{-1-\abs{\beta}}2^{-(\abs{\alpha}-\abs{\beta}+1)k^+}\lesssim 2^{(1-|\alpha|)k^-}2^{-(|\alpha|+2)k^+}.
    \end{align}
    This, together with the corresponding bound on $D^{\alpha}_\xi\Lambda(\xi-\eta)$ yields the desired claim for the phase. 
\end{proof}

Finally, we apply these results to the relevant objects when integrating by parts using Lemma \ref{lemma-IBPgen}.
\begin{lem}[Iterated bounds  \uppercase\expandafter{\romannumeral 2}] \label{lem:ibp}
Let  $\alpha\in \N^3_0$  be a multi-index of order $\abs{\alpha}$ and $f_i=Q_{j_i,k_i}F_i$ space-and-frequency localized profiles. Let $g(\xi,\eta)=\m(\xi,\eta)\varphi_k(\xi)\widehat{f_1}(\xi-\eta).$ On the support of $\chi$, there holds:
    \begin{align}  \label{lem:ibp-g2}
        \norm{D_{\xi}^\alpha g}_{L^1}\lesssim 2^{k_{\max}^+}(2^{-\min\{k,k_1\}}+2^{j_1})^{\abs{\alpha}}\norm{\varphi_k}_{L^2}\norm{f_1}_{L^2}.
    \end{align}

\end{lem}
\begin{proof}
    The proof follows by the product rule. We compute using Lemma \ref{lem:multiplier-bounds} and the fact that ${\nabla_\xi\varphi_k}\sim 2^{-k}{\Tilde{\varphi}_k}$, where $\Tilde{\varphi}$ is a a bump function with similar support properties as $\varphi$:
    \begin{align}
         \norm{D^\alpha_\xi g}_{L^1}&=\int\abs{D^{\alpha}_\xi\left(\m(\xi,\eta)\varphi_k(\xi)\widehat{f_1}(\xi-\eta)\right)}d\xi\\
         &\lesssim \sum_{\beta+\gamma=\alpha}\int \abs{D^{\beta}_\xi\m(\xi,\eta)D^\gamma_\xi\varphi_k(\xi)D^{\alpha-\beta-\gamma}_\xi\widehat{f_1}(\xi-\eta)}d\xi\\
         &\lesssim \sum_{\beta+\gamma=\alpha}\sup_{\xi,\eta\in\supp}\abs{D^\beta_\xi\m(\xi,\eta)\chi}2^{-\abs{\gamma} k}\int \abs{\tilde{\varphi}_k(\xi)D^{\alpha-\beta-\gamma}_\xi\widehat{f_1}(\xi-\eta)}d\xi\\
         &\lesssim 2^{k_{\max}}\sum_{\beta+\gamma=\alpha}2^{-\abs{\beta}\min\{k,k_1\}}2^{-\abs{\gamma} k}\norm{\tilde{\varphi}_k}_{L^2}\norm{D^{\alpha-\beta-\gamma}_\xi\widehat{f_1}}_{L^2}\\
         &\lesssim 2^{k_{\max}}\sum_{\beta+\gamma=\alpha}2^{-(\abs{\beta}+\abs{\gamma})\min\{k,k_1\}}2^{\abs{\alpha-\beta-\gamma}j_1}\norm{{\varphi}_k}_{L^2}\norm{f_1}_{L^2}\\
         &\lesssim 2^{k_{\max}}(2^{-\min\{k,k_1\}}+2^{j_1})^{\abs{\alpha}}\norm{{\varphi}_k}_{L^2}\norm{f_1}_{L^2}.
    \end{align}
    Note that we have used $\norm{\abs{D_\xi}\widehat{f_1}}_{L^2}\sim\norm{\abs{x}f_1}_{L^2}\sim2^{j_1}\norm{f_1}_{L^2}$ by Plancherel since the profiles are localized in space. This concludes the proof. 
\end{proof}

\appendix 
\section{Nonlinear Poisson Equation}\label{appendix}
In this appendix, we explain how to treat the full model for ions, that is without assuming a linearized law for the density of electrons. We restart from the perturbed system
\begin{equation}\label{eq:perturbedAppendix}
\left\{
      \begin{aligned}
      \partial_t \rho + \mathrm{div}(v)+\mathrm{div}(\rho v) &=0, \\
\partial_t v +(v \cdot \nabla )v &=-\nabla \Phi,  \\
e^{\Phi}-\Delta \Phi&= 1+\rho,
\end{aligned}
    \right. 
\end{equation} 
where we now take into account the full semilinear coupled elliptic equation on the potential $\Phi$.

 In treating the problem \eqref{eq:perturbedAppendix}, there are additional quadratic and cubic contributions that arise compared to the model \eqref{eq:EP-perturbed} considered throughout the paper. As explained is the introduction, the analysis of the quadratic interactions is at the core of the proof. The new quadratic contribution shares the same structural features and thus fits in the dispersive analysis framework of the paper, see \eqref{eq:Duhamel:APPENDIX}, \eqref{def:multipliers-n}. Our goal in this section is treating the cubic contributions and presenting how to adapt the main bootstrap argument from Proposition \ref{prop:bootstrap}.
\par One of the essential difference between \eqref{eq:perturbedAppendix} and the screened Laplace case is that the equation on the potential $\Phi$ in \eqref{eq:perturbedAppendix} is not explicitly solvable in terms of $\rho$, because we cannot write $\Phi=(1-\Delta)^{-1}\rho$. To come back to this case, and motivated by the fact that $e^{\Phi} \sim 1+ \Phi$ for small $\Phi$, we follow the strategy of Guo and Pausader in \cite{Guo-Pausader-ions3d} and write the equation on $\Phi$ as
\begin{align}\label{eq:1st-rewrite-expo}
    \Phi-\Delta\Phi=\rho-\frac{\Phi^2}{2}- E_3(\Phi), \ \ E_3(z)\vcentcolon= e^z-1-z- \frac{z^2}{2},
\end{align}
that is
\begin{align*}
    \Phi=(1-\Delta)^{-1} \rho- \frac{1}{2}(1-\Delta)^{-1} [\Phi^2]-(1-\Delta)^{-1} [E_3(\Phi)].
\end{align*}
Expanding to the next order in the second term, we introduce a remainder term $R(\rho)$ that is implicitly defined by the following identity:
\begin{align}\label{eq:expandPHIbyR}
    \Phi=(1-\Delta)^{-1} \rho-\frac{1}{2}(1-\Delta)^{-1} \Big[[(1-\Delta)^{-1} \rho]^2 \Big]+R(\rho).
\end{align}

We observe that this expansion carries out the linear contribution (similar as the one we had before in the screened Laplace), a new quadratic contribution, and hence a remainder that is expected to be of cubic nature. The main difficulty, similar to the one in \cite{Guo-Pausader-ions3d}, will be that the remainder $R(\rho)$ is only known \textit{via} the original equation \eqref{eq:1st-rewrite-expo}.
\par With the dispersive unknowns $Z_\pm$ as in \eqref{eq:def-Unknown}, we derive the new dispersive formulation in the context of \eqref{eq:perturbedAppendix}. Relying on \eqref{eq:expandPHIbyR}, and performing the same computations as in Proposition \ref{prop-Duhamel}, we obtain the new system.
\begin{prop}[Duhamel formulation]
    The system \eqref{eq:perturbedAppendix} is equivalent to
    \begin{equation}\label{eq:new-dispers-form-APPENDIX}
\begin{split}
       \partial_tZ_\pm+\frac{i}{4}\abs{\nabla}^{-1}\mathrm{div}\left(\abs{\nabla}(Z_++Z_-)\nabla\langle\nabla\rangle^{-1}(Z_+-Z_-)\right)&\\
       \pm \frac{i}{8}\langle \nabla\rangle\abs{\nabla\langle\nabla\rangle^{-1}(Z_+-Z_-)}^2 \pm \frac{i}{8}\langle \nabla \rangle^{-1} \Big[ \vert \nabla \vert \langle \nabla \rangle^{-2} (Z_+ + Z_-)\Big]^2&\\
       \mp i \langle \nabla \rangle R\left( \frac{1}{2}\vert \nabla \vert (Z_+ + Z_-)\right)
       &=\pm i\Lambda Z_\pm.
   \end{split} 
\end{equation}
Here $\Lambda$ is the dispersion relation \eqref{def:dispersionrel} and $R$ is defined implicitly by \eqref{eq:expandPHIbyR}. Moreover, for the profiles $\mathcal{Z}_\pm$, there holds
\begin{align}\label{eq:Duhamel:APPENDIX}
    \zz_{\pm}(t)=\zz_{\pm}(0)+\sum_{\mu\in\{+,-\}}\mathcal{B}_{\m_\pm^{\mu\mu}+\mathfrak{n}_\pm^{\mu \mu}}(\zz_\mu,\zz_\mu)(t)+\mathcal{B}_{\m_\pm^{+-}+\mathfrak{n}_\pm^{+-}}(\zz_+,\zz_-)(t)\pm \int_0^t\mathcal{C}^R_\pm(s) \, \mathrm{d}s,
\end{align}
where 
\begin{align}\label{def:multipliers-n}
    \mathfrak{n}_\pm^{\mu \mu}(\xi , \eta)\vcentcolon=\pm \frac{i}{8} \frac{1}{\br{\xi}} \frac{\abs{\xi-\eta}}{\br{\xi-\eta}^2}\frac{\abs{\eta}}{\br{\eta}^2}, \qquad
    \mathfrak{n}_\pm^{+-}(\xi, \eta)\vcentcolon=\pm \frac{i}{4} \frac{1}{\br{\xi}} \frac{\abs{\xi-\eta}}{\br{\xi-\eta}^2}\frac{\abs{\eta}}{\br{\eta}^2}.
\end{align}
and
\begin{align}\label{eq:remainder-DuhamelNr-APPENDIX}
    \mathcal{C}^R_\pm\vcentcolon= i e^{\mp it \Lambda} \langle \nabla \rangle R\left( \frac{1}{2}\vert \nabla \vert (Z_+ + Z_-)\right).
\end{align}
\end{prop}
\begin{rem}[Quadratic interactions] \label{rem:quadratic-interactions}
   The additional multipliers $\mathfrak{n}_{\pm}$ satisfy the same bounds as $\m_\pm$ in $S^\infty$ (or $L^\infty$). Indeed, by definition \eqref{def:multipliers-n} there holds $\norm{\mathfrak{n}_\pm\chi}_{S^\infty}\lesssim 2^{k_1+k_2}$, c.f.\ Lemma \ref{lem:multiplier-bounds}. Thus,  the main dispersive analysis of Section \ref{sec:X-norm} follows exactly in the same way.
\end{rem}
\par In the rest of this section, we focus on the new cubic contribution arising from the remainder term $R(\rho)$. We aim at deriving a closed equation on $R(\rho)$ in terms of $\rho$ only, and that is amenable to derive suitable estimates. To this end, we follow the approach of Guo and Pausader in \cite{Guo-Pausader-ions3d}. For the sake of readability, let us denote the screened inverse Laplace operator as $$L\vcentcolon=(1-\Delta)^{-1}.$$From the original identity \eqref{eq:1st-rewrite-expo} 
and plugging the expression of $\Phi$ in terms of $\rho$ and $R(\rho)$ into it thanks to \eqref{eq:expandPHIbyR}, we get
\begin{align*}
    \rho&= (1-\Delta)\Big[ L \rho-\frac{1}{2}L \Big((L \rho)^2 \Big)+R(\rho)\Big]+ \frac{1}{2}\left[L\rho-\frac{1}{2}L \Big((L \rho)^2 \Big)+R(\rho) \right]^2  + E_3 \Big( L \rho-\frac{1}{2}L \Big((L \rho)^2 \Big)+R(\rho)\Big).
\end{align*}
 The equation on $R(\rho)$ can now be rewritten as the following closed semilinear elliptic equation with a source depending on $\rho$:
\begin{align}\label{eq:remainderR}
    (1-\Delta)R={-}F(\rho)R-\frac{1}{2} R^2-E_3\big(F(\rho)+R \big)+S(\rho),
\end{align}
where 
\begin{align*}
S(\rho)\vcentcolon= \frac{1}{2}(L\rho)L\Big((L\rho)^2 \Big)-\frac{1}{8}\left( L\Big((L \rho)^2 \Big)\right)^2, \qquad F(\rho) \vcentcolon= L\rho -\frac{1}{2}L\Big((L\rho)^2 \Big).
\end{align*}
In order to close the proof via a bootstrap as in Proposition \ref{prop:bootstrap} we need energy estimates for \eqref{eq:new-dispers-form-APPENDIX} and $X$-norm bounds for the last term in \eqref{eq:Duhamel:APPENDIX} and to that end we use \eqref{eq:remainderR}.
\begin{rem} In order to have a local-well-posedness theory for the unknowns $Z_\pm$, we infer it from the local-well-posedness of $(\rho,v)$ in conjunction with $\psi \in L^2$.  Similarly to Remark \ref{rem-LWP-disp}, in this setting if $\psi_0\in L^2$, then  $\Vert \psi \Vert_{L^\infty(0,T; L^2)}< \infty$ for $T$ such that $(\rho,v)$ exist.
   In order to see this from the equation on $\psi$, we obtain $L^2$ energy estimates 
\begin{align*}
    \frac{\mathrm{d}}{\mathrm{d}t} \norm{\psi}_{L^2} \lesssim \norm{v}_{H^{N_0}} \norm{\psi}_{L^2}+\norm{\Phi}_{L^2}.
\end{align*}
To bound the last term, we multiply the above elliptic equation by $\Phi$ and get
\begin{align*}
    \int_{\R^3} (e^{\Phi}-1) \Phi+\int_{\R^3} \vert \nabla \Phi \vert^2 \leq \norm{\rho}_{L^2}\norm{\Phi}_{L^2}.
    \end{align*}
    Now writing $e^{\Phi(x)}-1=a(x) \Phi(x)$ where $a(x)=\int_0^1 e^{\theta \Phi(x)} \, \mathrm{d}\theta$, we see that $a(x) \geq e^{-\Vert \Phi \Vert_{L^\infty}}$ and we end up after simplification with
    $\norm{\Phi}_{L^2} \leq e^{\Vert \Phi \Vert_{L^\infty}} \norm{\rho}_{L^2}$.
    From \cite{LannesLinaresSaut} 
    we have that $\Vert \Phi \Vert_{L^\infty} \leq  \log(1+\norm{\rho}_{L^\infty})$. This entails $\norm{\Phi}_{L^2} \leq (1+\norm{\rho}_{L^\infty}) \norm{\rho}_{L^2}$. The claim follows Grönwall's lemma from 
    \begin{align*}
        \frac{\mathrm{d}}{\mathrm{d}t} \norm{\psi}_{L^2} \lesssim \norm{v}_{H^{N_0}} \norm{\psi}_{L^2}+(1+\norm{\rho}_{H^{N_0-1}}) \norm{\rho}_{H^{N_0-1}}.
    \end{align*}
\end{rem}

\medskip

We prove the following bootstrap proposition. This together with the energy estimates above yields the proof of the main theorem.
\begin{prop}(Bootstrap) Under the assumptions of Proposition \ref{prop:bootstrap}
\begin{align}\label{eq:boostrapX-norm-Appendix}
    \norm{ \zz_+(t)}_X +\norm{ \zz_-(t)}_{X} &\leq 100\eps, 
\end{align}
the following improved bounds hold:
\begin{align}\label{eq:bootstrapR-improved}
\begin{split}
    \norm{\zz_+(t)}_{X} +\norm{\zz_-(t) }_{X} &\leq 10\eps.
\end{split}
\end{align}
\end{prop}
\begin{proof}
We adapt the energy estimates from Section \ref{sec:energy-estimates} to the unknowns $Z_\pm$ solving the new system \eqref{eq:Duhamel:APPENDIX}. Indeed, using the energies $E_s^2$, $s\geq 0$ defined in the proof of Proposition \ref{prop:Sobolev-estimate},  we see that the new quadratic contribution due to the second term on the right-hand side of \eqref{eq:expandPHIbyR} appears only in the equation on the stream function $\psi$ and obeys the same bounds leaving the blow up criterion intact. 
\par To estimate the contribution from the new cubic term $R(\rho)$ appearing in the equation on $\psi$, we proceed via a fixed-point argument as in \cite[\textcolor{MidnightBlue}{proof of Lemma 3.1}]{Guo-Pausader-ions3d} based on \eqref{eq:remainderR}, and under the bootstrap assumptions \eqref{eq:boostrapX-norm-Appendix} (in particular it is necessary that $\norm{\rho}_{L^\infty}\ll 1$), there holds that
\begin{align}\label{eq:Linfty-decay-R}
     \norm{R}_{H^{s+1}}\lesssim \norm{\rho}_{L^\infty}^2\norm{\rho}_{H^{s-1}}, \ \ \norm{R}_{L^\infty} \lesssim \Vert\rho \Vert_{L^\infty}^3.
\end{align}
This yields that for $s\geq 0$ and $\mathrm{A}$ as in \eqref{eqn:blow-up-crit}, the energies $E_s^2$ satisfy
\begin{align}
    \frac{d}{dt}E_s^2(t)\lesssim (\mathrm{A(t)}+\norm{\rho}_{L^\infty}^2)E_s^2(t).
\end{align} Hence, we obtain the corresponding standard energy estimates:
 \begin{align*}
\sum_{\mu\in \lbrace +,-\rbrace }\norm{Z_\mu(t)}_{H^s}^2\leq C_s \exp\left(c_s\int_0^t \widetilde{\mathrm{A}}(\tau) \, \mathrm{d}\tau\right)\sum_{\mu\in \lbrace +,-\rbrace }\norm{Z_\mu(0)}_{H^s}^2, 
\end{align*}
with 
\begin{align}\label{eq:newBlow-up-Rterm}
    \mathrm{\widetilde{A}}(\tau)\vcentcolon=\mathrm{A}(\tau)+\sum_{\mu\in\{+,-\}}\norm{|\nabla|Z_{\mu}}_{L^\infty}^2.
\end{align}
The bootstrap assumption \eqref{eq:boostrapX-norm-Appendix} then yields the energy bound:
$$\norm{{Z}_\pm(t)}_{H^{N_0}}\lesssim \eps.$$
By \eqref{eq:Duhamel:APPENDIX} we have to bound \eqref{bootstrap-z-in-proof} with an additional cubic term $\int_0^t \norm{\mathcal{C}^R_\pm(s)}_X \, \mathrm{d}s$ on the right-hand side. By Remark \ref{rem:quadratic-interactions}, the quadratic interactions are bounded similarly to the proof of Proposition \ref{prop:bootstrap}. In order to obtain  \eqref{eq:bootstrapR-improved}, it suffices to prove
\begin{align}
    \int_0^t \norm{\mathcal{C}^R_\pm(s)}_X \, \mathrm{d}s\lesssim \eps^3.
\end{align}
This follows by the arguments outlined below, in particular Proposition \ref{Prop-Xestimate-R}.
\end{proof} 
 To prove the main cubic estimate in Proposition \ref{Prop-Xestimate-R} below, we relate the $X$-norm to more suitable weighted Sobolev norms, that  behave well with the elliptic equation \eqref{eq:remainderR}. For any $ \alpha \in \R$ and $s \in \N$, let us define the following weighted Sobolev norms by
\begin{align*}
    \Vert f \Vert_{H^s_\alpha}\vcentcolon=\sum_{\vert \sigma \vert \leq s} \norm{\langle \cdot \rangle^\alpha \partial^\sigma f}_{L^2}, \ \ 
    \Vert f \Vert_{W^{s, \infty}_\alpha}\vcentcolon=\sum_{\vert \sigma \vert \leq s} \norm{\langle \cdot \rangle^\alpha \partial^\sigma f}_{L^\infty}.
\end{align*}
In what follows, we will rely on the following series of lemma concerning weighted Sobolev spaces.
\begin{lem}[$X$-norm and weighted Sobolev spaces]\label{Lem:compX-weightSob}
Let $\beta \in ]0,1/2[, \gamma <0$ and $\alpha \in \R$. For any $s,s' \in \N$ and any function $f=f(x)$, there holds
\begin{align}
    \label{eq:estimSob-to-X}\norm{f}_{H^{s'}_\alpha} &\lesssim \norm{f}_X , \ \ \alpha<1+\beta, \ \ \alpha+\gamma<s'<9+\beta, \\
        \label{eq:estimeX-to-Sob}\norm{f}_X &\lesssim \norm{f}_{H^s_{1+\beta}}  \ \ 9+\beta \leq s.
\end{align}
    Furthermore, for any integer $n \in \N$, one has 
\begin{align}\label{eq:estim-DerLinftyweight}
    \norm{\nabla^n f}_{L^\infty_\alpha} &\lesssim \norm{f}_X, \ \ \alpha \leq 1+\beta, \ \ 1+\beta+\gamma-\frac{3}{2}<n < 8+1+\beta -\frac{3}{2}.
\end{align}
\end{lem}
\begin{proof}
     Let $|\sigma|\leq s'$, then \eqref{eq:estimSob-to-X} follows via the arguments used to prove \eqref{eqn:tech-localization} and summing in $j+k\geq 0$, since there holds
\begin{equation}
    \norm{\br{x}^\alpha\partial^\sigma f}_{L^2} \lesssim \sum_{(j,k)\in \mathcal{J}}\norm{\br{x}^\alpha\varphi_j^{(k)}P_k\partial^\sigma f}_{L^2}\lesssim \sum_{(j,k)\in \mathcal{J}} 2^{\alpha j}2^{\abs{\sigma}k}\norm{\varphi_j^{(k)}\widetilde{P}_kf}_{L^2}\lesssim \norm{f}_X,
\end{equation}
where $\widetilde{P}_{k}$ is the projection associated with a $\widetilde{\varphi}$ with similar support properties as $\varphi$. For the second estimate \eqref{eq:estimeX-to-Sob}, we have
\begin{align*}
    2^{8k^+}2^{(1+\beta)(k+j)}2^{\gamma k^-}\|Q_{jk}f\|_{L^2}
\lesssim 2^{8k^+}2^{(1+\beta)k}2^{\gamma k^-}\|\langle x\rangle^{1+\beta}P_k f\|_{L^2} \lesssim 2^{(9+\beta-s)k^+}\norm{f}_{H^s_{1+\beta}}\lesssim \norm{f}_{H^s_{1+\beta}}.
\end{align*} 
Taking the supremum over $(j,k) \in \mathcal{J}$ yields \eqref{eq:estimeX-to-Sob}. It remains to prove \eqref{eq:estim-DerLinftyweight}. We start from the estimate
\begin{align*}
    \norm{\langle x \rangle^\alpha \nabla^n f}_{L^\infty} &\lesssim \sum_{(j,k) \in \mathcal{J}}2^{\alpha j}2^{nk}\norm{\varphi_k^{(j)}\widetilde{P}_kf}_{L^\infty} 
    \lesssim \sum_{|k-k'|\leq 4}\sum_{(k',j') \in \mathcal{J}} 2^{\alpha j}2^{nk}\norm{\varphi^{(j)}_k\tilde{P}_kQ_{j'k'}f}_{L^\infty}.
\end{align*}
Relying on Young's inequality for integral operators, we obtain as in what precedes \eqref{eqn:tech-localization}
\begin{align}
    \sum_{|k-k'|\leq 4}\sum_{(k',j') \in \mathcal{J}}\norm{\varphi^{(j)}_k\tilde{P}_kQ_{j'k'}f}_{L^\infty}
    &\lesssim \sum_{|k-k'|\leq 4}\sum_{|j-j'|\leq 4}\norm{Q_{j'k'}f}_{L^\infty}+ 2^{-2j}\sum_{\abs{k-k'}\leq 4}\norm{P_{k'}f}_{L^\infty} \\
    & \lesssim   \sum_{|k-k'|\leq 4}\sum_{|j-j'|\leq 4} 2^{\frac{3}{2}k'}\norm{Q_{j'k'}f}_{L^2}+ 2^{-2j}\sum_{\abs{k-k'}\leq 4}2^{\frac{3}{2}k'} \norm{P_{k'}f}_{L^2}
\end{align}
by Bernstein's inequality. We obtain 
 \begin{align*}
        \norm{\varphi_j^{(k)} \tilde P_k f}_{L^\infty} \lesssim 2^{\frac{3}{2}k}2^{-8k^+}2^{-(1+\beta)(j+k)-\gamma k^-}\norm{f}_X.
    \end{align*}
from which \eqref{eq:estim-DerLinftyweight} follows.
\end{proof}
\begin{lem}[Elliptic Estimates]\label{Lem:ellipticReg-screened}
For any function $f=f(x)$ and $s \in \N$, there holds
\begin{align}\label{eq:elliptic-reg-1stgain}
    \norm{Lf}_{W^{s+1, \infty}_\alpha} &\lesssim  \norm{f}_{W^{s, \infty}_\alpha},  \ \  \alpha \in \R^+, \\
    \label{eq:elliptic-reg-full}
    \norm{Lf}_{H^{s+2}_\alpha} &\lesssim  \norm{f}_{H^{s}_\alpha}, \ \vert \alpha \vert < \frac{3}{2}.
\end{align}
\end{lem}
\begin{proof}
    Recall that the Green kernel associated with the screened Laplacian in dimension $3$ is $$G(x)=\frac{1}{4\pi}\frac{e^{-\vert x \vert}}{\vert x \vert}.$$
    So if $u$ is the solution to $(1-\Delta)u=f$, that is $u=Lu=G \star f$, we have by using  $\langle x \rangle^\alpha \leq \langle x-y \rangle^\alpha\langle y \rangle^\alpha$ that
\begin{align*}
    \vert \langle x \rangle^\alpha (1+\nabla)u(x)\vert  \lesssim \int_{\R^3} \langle x-y \rangle^\alpha \vert (1+\nabla)G(x-y) \vert \langle y \rangle^\alpha \vert f(y) \vert \, \mathrm{d}y.
\end{align*}
Since $\int_{\R^3} \langle z \rangle^\alpha \vert G(z) \vert \, \mathrm{d}z$ and $ \int_{\R^3} \langle z \rangle^\alpha \vert \nabla G (z)\vert \, \mathrm{d}z < \infty$ for any $\alpha \in \R^+$ therefore, (iterating by linearity of the equation) , we obtain the estimate \eqref{eq:elliptic-reg-1stgain} as a consequence of Young's inequality for convolution.

For the second estimate, we rely on the standard Calderon-Zygmund theory adapted to weighted spaces (indeed, $\nabla^2 G$ is only $L^1_{\mathrm{loc}}$ because it behaves like the second order derivates of the unscreened Coulomb kernel around $x=0$). We rely on the following result/framework (see e.g.\ \cite[Chapter V]{stein1993harmonic}): if $1 <p<\infty$, a function $w=w(x)>0$ is in the so-called Muckenhoupt class $A_p$ if and only if there exists $C>0$ such that for any ball $B$ there holds
\begin{align*}
    \left(\frac{1}{\vert B \vert}\int_B w \right) \left(\frac{1}{\vert B \vert}\int_B w^{-\frac{1}{p-1}} \right)^{p-1} \leq C.
\end{align*}
We observe that on $\R^3$, the weight $w(x)=\langle x \rangle^{2\alpha} \in A_2$ provided that $\vert \alpha \vert <3/2$.
In this framework, Calderon-Zygmund type operators are bounded on the weighted space $L^p(w(x)\mathrm{d}x)$ with $1<p<\infty$ if $w \in A_p$. This yields the desired estimate \eqref{eq:elliptic-reg-full} and concludes the proof.
\end{proof}
\begin{lem}[Product estimates]\label{Lem:product-weightedSob}
    Let $s \geq 1$ and $\alpha \in \R$. For any $\kappa>0$ and functions $f,g$ there holds
    \begin{align}\label{eq:tame-estimate}
        \norm{fg}_{H^s_\alpha} &\lesssim \norm{f}_{W^{\lfloor \frac{s}{2} \rfloor, \infty}_\kappa}\norm{g}_{H^s_{\alpha-\kappa}}+\norm{f}_{H^s_{\alpha-\kappa}}\norm{g}_{W^{\lfloor \frac{s}{2} \rfloor, \infty}_\kappa}.
        \end{align}
         In particular there holds for any integer $k \geq 1$
    \begin{align}\label{eq:tame-estimate-power}
        \begin{split}
    \norm{f^k}_{H^s_\alpha} &\lesssim \norm{f}_{W^{\lfloor \frac{s}{2} \rfloor, \infty}_\kappa}^{k-1} \norm{f}_{H^s_{\alpha-(k-1)\kappa}},
        \end{split}
    \end{align}
    and
\begin{align}\label{eq:compo-estimate}
     \norm{E_3(f)}_{H^s_\alpha} \lesssim  \norm{f}_{H^s_{\alpha-\kappa}}\norm{f}_{W^{\lfloor \frac{s}{2} \rfloor, \infty}_\kappa}^2 \exp\left(\norm{f}_{W^{\lfloor \frac{s}{2} \rfloor, \infty}_\kappa} \right), \ \ E_3(z)=e^z-1-z-\frac{z^2}{2}.
\end{align}
\end{lem}
\begin{proof}
    The first estimate \eqref{eq:tame-estimate} is consequence of Leibniz rule, splitting of the weight, and splitting of the sum with $L^2-L^\infty$ estimates depending on the number of derivatives. The estimates \eqref{eq:tame-estimate-power} is then obtained as an iterated application of \eqref{eq:tame-estimate}. To prove \eqref{eq:compo-estimate}, we use \eqref{eq:tame-estimate-power} and get \begin{align*}
    \norm{E_3(g)}_{H^s_\alpha} \leq \sum_{k=3}^{\infty} \frac{\norm{g^k}_{H^s_\alpha}}{k!} \lesssim \sum_{k=3}^{\infty} \norm{g}_{H^s_{\alpha-(k-1)\kappa}}\frac{\norm{g}_{W^{\lfloor \frac{s}{2} \rfloor, \infty}_{{\kappa}}}^{k-1}}{k!} \leq \norm{g}_{H^s_{\alpha-\kappa}}\norm{g}_{W^{\lfloor \frac{s}{2} \rfloor, \infty}_\kappa}^2\sum_{k=0}^{\infty} \frac{\norm{g}_{W^{\lfloor \frac{s}{2} \rfloor, \infty}_\kappa}^{k}}{k!},
\end{align*}
which finishes the proof.
\end{proof}
\begin{lem}[Decay Estimates]\label{Lem-decay-deriv-Weighted}
For $\beta \in ]0, 1/5[$, there holds for any $n \in \lbrace 1,2,3,4 \rbrace$
    \begin{align*}
        \norm{\nabla^n Z_\pm(t)}_{L^\infty_\kappa} &\lesssim \log(t)(1+t)^{-(1+\beta-\kappa)} \norm{\mathcal{Z}(t)}_X, \ \ 0 \leq \kappa<1+\beta.
    \end{align*}
    It holds as well for $\vert \nabla \vert$ instead of $\nabla$.
\end{lem}
\begin{proof}
    We start from
    \begin{align*}
        \norm{\nabla^n Z_\pm(t)}_{L^\infty_\kappa} 
    & \lesssim \sum_{(k,j) \in \mathcal{J}} \sum_{\substack{(k',j') \in \mathcal{J} \\ \vert k-k'\vert \leq 4 \\  j \leq j'+\log_2(t)}}  2^{nk} 2^{\kappa j}\norm{  \varphi_j^{(k)} \widetilde{P_k} e^{\pm it\Lambda} Q_{j'k'} \zz_\pm}_{L^\infty} \\
    & \qquad +  \sum_{(k,j) \in \mathcal{J}} \sum_{\substack{(k',j') \in \mathcal{J} \\ \vert k-k'\vert \leq 4 \\  j > j'+\log_2(t)}}  2^{nk} 2^{\kappa j}\norm{  \varphi_j^{(k)} \widetilde{P_k} e^{\pm it\Lambda} \varphi_{j'}^{(k')} P_{k'} \zz_\pm}_{L^\infty}.
    \end{align*}
    On the one hand, we observe that when $j \leq j'+\log_2(t)$ we have 
    \begin{align*} 2^{\kappa j}\norm{  \varphi_j^{(k)} \widetilde{P_k} e^{\pm it\Lambda}Q_{j'k'}  \zz_\pm}_{L^\infty} &\lesssim (t+2^{j'})^\kappa \norm{e^{\pm it\Lambda} Q_{j' k'} \zz_\pm}_{L^\infty}.
    \end{align*}
    On the other hand, let us denote the kernel of $\widetilde{P_k} e^{\pm it\Lambda}$ by $K_{t,k}$, so that $\widetilde{P_k} e^{\pm it\Lambda} [\varphi_{j'}^{(k')} P_{k'}]  \zz_+(x)=K_{t,k} \star (\varphi_{j'}^{(k')} P_{k'}  \zz_+)(x) $. Note that $\varphi_{j'}^{(k')} P_{k'}  \zz_+(y)$ has a support in $\vert y \vert \lesssim 2^{j'}$ therefore if $j>j'+C\log_2(t)$ for $C$ large enough then $\vert x \vert \sim 2^{j} > C(t+2^j{'})$ and $\vert x-y \vert > \vert x \vert/2 \sim 2^j$ so for such $x$ we get 
    $$ \vert \widetilde{P_k} e^{\pm it\Lambda} \varphi_{j'}^{(k')} P_{k'}\zz_\pm(x) \vert \leq \norm{K_{t,k} \textbf{1}_{\vert \cdot \vert>2^j}}_{L^2} \norm{\varphi_{j'}^{(k')} P_{k'} \zz_\pm}_{L^2}. $$
    From a non-stationary phase argument, using that $\vert \nabla \Lambda \vert \lesssim 1$, we have the fact that for $\vert z \vert >2t$ there holds $\vert K_{t,k}(z) \vert \lesssim 2^{3k} (2^k \vert z \vert)^{-N} $ for any $N>0$. Therefore, since $2^j>2t$, we get after an explicit computation that for all $N>0$ large enough $$ \norm{K_{t,k} \textbf{1}_{\vert \cdot \vert>2^j}}_{L^2} \lesssim 2^{\frac{3}{2}k} (2^k 2^j)^{\frac{3}{2}-N}. $$
    As a consequence, when $j>j'+\log_2(t)$, we get 
    \begin{align*}
        2^{nk }2^{\kappa j}\norm{ \widetilde{P_k} e^{\pm it\Lambda} \varphi_{j'}^{(k')} P_{k'}\mathcal{Z}_\pm}_{L^\infty(\vert x \vert \sim 2^j)} &\lesssim 2^{nk}2^{\kappa j} 2^{\frac{3}{2}k} (2^k 2^j)^{\frac{3}{2}-N} \norm{\varphi_{j'}^{(k')} P_{k'}  \zz_+}_{L^2} \lesssim 2^{-\tilde{N}(k+j)} \norm{\varphi_{j'}^{(k')} P_{k'}\mathcal{Z}}_{L^2}
    \end{align*}
    for any $\widetilde{N}>0$ large enough. All in all, we have obtained for all $N>0$ large enough
    \begin{multline}\label{eq:panda}
         \norm{\nabla^n Z_\pm(t)}_{L^\infty_\kappa} \\ \lesssim \sum_{(k,j) \in \mathcal{J}} \sum_{\substack{(k',j') \in \mathcal{J} \\ \vert k-k'\vert \leq 4 \\  j  \leq  j'+\log_2(t)}}  2^{nk} (t+2^{j'})^\kappa \norm{e^{\pm it\Lambda}Q_{j' k'} \zz_\pm}_{L^\infty} 
 + \sum_{(k,j) \in \mathcal{J}} \sum_{\substack{(k',j') \in \mathcal{J} \\ \vert k-k'\vert \leq 4 \\  j > j'+\log_2(t)}} 2^{-N(k+j)} \norm{\varphi_{j'}^{(k')} P_{k'}  \mathcal{Z}}_{L^2}.
    \end{multline}
    For the second term, we have any decay we want 
    and by relying on \eqref{eqn:tech-localization}, we obtain a contribution $t^{-N} \norm{\mathcal{Z}}_X$. 
    %
    Let us treat the first term in \eqref{eq:panda}. We should actually first sum over $j \leq j'+\log_2(t)$ with $j'$ fixed, which gives a finite number of terms. For $j' \leq  \log_2(t)$, we use the first entry in \eqref{eqn:linear-decay-j<m} to get
    \begin{align*}
        \sum_{j' \leq \log_2t, j \leq j'+\log_2(t)}(t+2^{j'})^\kappa \norm{ e^{it \Lambda} Q_{j'k'} \zz_\pm(t)}_{L^\infty} \lesssim \log_2(t)  t^{-(1+\beta-\kappa)}2^{-(3+\beta)k'^+}2^{(-\frac13+\frac{2}{3}\beta)k'^-}\norm{\zz(t)}_X.
    \end{align*}
    For $j'>\log_2t$, we do not use any decay from the semigroup and directly get by Bernstein's inequality that 
    \begin{align*}
        \sum_{j' > \log_2t, j \leq j'+\log_2(t)}(t+2^{j'})^\kappa \norm{ e^{it \Lambda} Q_{j'k'} \zz_\pm(t)}_{L^\infty}   \lesssim t^{-(1+\beta-\kappa)}2^{(\frac12-\beta)k'}2^{-8k'^+}2^{-\gamma k'^-}\norm{\zz(t)}_X.
    \end{align*}
   where we have used that $(t+2^{j'})^\kappa j' 2^{-(1+\beta)j'} \lesssim 2^{-(1+\beta-\kappa)j'} \lesssim t^{-(1+\beta-\kappa)}$.
   The two former estimates are then summable in $k \sim k'$ with a weight $2^{nk}$ for $n=1,2,3,4$, except for the first when $k \geq 0$ and $n=4$. If $n=4$, we therefore proceed differently when $j' \leq \log_2t$ and $k' \geq 0$.
   We perform an interpolation procedure based on the two available bounds
    \begin{align*}
        \norm{ e^{it \Lambda} Q_{j'k'} \zz_\pm(t)}_{L^\infty}\lesssim \min \left\lbrace t^{-3/2}2^{-(3+\beta)k'}2^{(\frac12-\beta)j'},  2^{-(\frac{15}{2}+\beta)k'} 2^{-(1+\beta)j'} \right\rbrace \norm{\mathcal Z(t)}_X
    \end{align*} which gives for any $\theta \in (0,1) $ that
    \begin{align*}
        (t+2^j)^\kappa  \norm{ e^{it \Lambda} Q_{j'k'} \zz_\pm(t)}_{L^\infty}\lesssim t^{\kappa-\frac{3}{2}\theta} 2^{(-\frac{15}{2}-\beta+\frac{9}{2}\theta)k'} 2^{(\frac{3}{2}\theta-(1+\beta))j'}\norm{\mathcal Z(t)}_X.
    \end{align*}
    Multiplying by $2^{4k}$, and then summing in $(k,j)$, we obtain, the claimed estimate modulo the  provided that 
    $$\frac{2}{3}(1+\beta) < \theta < \frac{7+2\beta}{9}<1,$$
that is satisfied since $\beta \leq \frac{1}{5}<\frac{1}{4}$. This gives the claimed estimate and conclude the proof.
\end{proof}

\begin{lem}[Commutator Estimates]\label{lem:commutator-weight}
For $\delta>0$ there exists $C_\delta>0$ such that for any $0<a<\frac{3}{2}-\delta$ and $f=f(x)$, we have for all $t \in \R$
    \begin{align*}
\norm{\br{x}^{a}e^{it\Lambda}f}_{L^2}\lesssim \br{t}^{a}\norm{f}_{L^2}+C_\delta\norm{\langle x \rangle^{a+\delta}f}_{L^2},
\end{align*}
\end{lem}
\begin{proof}
We decompose
\begin{equation}
\norm{\br{x}^{a}e^{it\Lambda}f}_{L^2}^2 \lesssim \br{t}^{2a}\norm{f}_{L^2}^2+\sum_{\substack{(j,k) \in \mathcal{J} \\ j\geq \log(t)+4}}2^{2aj}\sum_{\substack{(j',k') \in \mathcal{J} \\
\vert k -k' \vert \leq 4}}\norm{\tilde\varphi_j^{(k)} \widetilde{P}_k e^{it\Lambda}Q_{j'k'}f }_{L^2}^2,
\end{equation}
 where we have used a rough bound yielding the first term whenever $j \leq \log(t)+4$. When $\abs{j-j'}\leq 4$, we simply use that $2^{aj'}\Vert \tilde\varphi_j^{(k)} \widetilde{P}_k e^{it\Lambda}Q_{j'k'}f \Vert_{L^2} \lesssim 2^{aj'}\Vert \varphi_{j'}^{(k')}P_{k'}f \Vert_{L^2}\lesssim 2^{-\delta j'}\norm{\br{x}^{a+\delta}P_{k'}f}_{L^2}$ for any $\delta>0$. When $\abs{j-j'}>4$ then $\max\{2^j,2^{j'}\}\gg t$ and therefore a non-stationary phase argument (using the fact that $\norm{\nabla \Lambda}_{L^\infty} \lesssim 1$) combined with Schur's test entails $\Vert \tilde\varphi_j^{(k)} \widetilde{P}_k e^{it\Lambda}Q_{j'k'}f \Vert_{L^2} \lesssim 2^{-N\vert j - j' \vert} \Vert \varphi_{j'}^{(k')}P_{k'}f \Vert_{L^2} $. By Minkowski and Young's inequalities, we obtain
 \begin{align*}
\norm{\br{x}^{a}e^{it\Lambda}f}_{L^2}\lesssim \br{t}^{a}\norm{f}_{L^2}+C_\delta \left(\sum_{(j', k') \in J} 2^{2(a+\delta)j} \norm{Q_{j'k'}f}_{L^2}^2 \right)^{1/2} \lesssim \br{t}^{a}\norm{f}_{L^2}+C_\delta\norm{\langle x \rangle^{a+\delta}f}_{L^2}.
 \end{align*}
 \end{proof}

The main estimate of the remainder $R(\rho)$ is the following proposition.
\begin{prop}\label{Prop-Xestimate-R}
 Under the bootstrap assumption of Proposition \ref{prop:bootstrap}, for any $\kappa>0$ and  there holds that
    \begin{align}\label{claim-for-commutators}
        \norm{R(t)}_{H^{s+1}_{1+\beta}} \lesssim \eps^3  (\log t)^2(1+t)^{-(1+\beta-\kappa)},\qquad \norm{R(t)}_{H^{s+1}} \lesssim \eps^3 (1+t)^{-2(1+\beta)}. 
    \end{align}
    In particular, there holds the cubic estimate:
    \begin{align}\label{X-norm-cubic-proof}
        \int_0^t\norm{\mathcal{C}_\pm^R(\tau)}_Xd\tau\lesssim \eps^3.
    \end{align}
    
\end{prop}
\begin{proof}[Proof of Proposition \ref{Prop-Xestimate-R}]
 The main estimate \eqref{X-norm-cubic-proof} follows from \eqref{claim-for-commutators} using Lemma \ref{lem:commutator-weight}. Indeed, using  Lemma \ref{Lem:compX-weightSob}, the commutator estimate Lemma \ref{lem:commutator-weight} and \eqref{claim-for-commutators} there holds with $s\geq 9+\beta$:
\begin{equation}\label{eqn-cubic-bootstrap}
\begin{aligned}
    \int_0^t \norm{\mathcal{C}^R_\pm(\tau)}_X  \mathrm{d}\tau &\lesssim \int_0^t \norm{e^{\mp i\tau \Lambda} \langle \nabla \rangle R(\tau)}_{H^s_{1+\beta}} \mathrm{d}\tau\\
    &\lesssim \int_0^t \langle \tau \rangle^{1+\beta} \norm{ R(\tau)}_{H^{s+1}}\mathrm{d}\tau + \int_0^t  \norm{R(\tau)}_{H^{s+1}_{1+\beta}}\mathrm{d}\tau\\
    &\lesssim \int_0^t \br{\tau}^{1+\beta}\eps^3(1+\tau)^{-2(1+\beta)}d\tau +\int_0^t \eps^3(\log\tau)^2(1+\tau)^{-(1+\beta-\kappa)}d\tau\\
    &\lesssim \eps^3.
\end{aligned}
\end{equation}

We now prove \eqref{claim-for-commutators}  by performing suitable energy estimates on the elliptic equation \eqref{eq:remainderR} for the remainder $R$.  First note that the $H^{s+1}$ bound follows by the fixed point argument \eqref{eq:Linfty-decay-R} and the decay in time of $\rho$ in $L^\infty$.
We prove the weighted estimate on  $\norm{R(t)}_{H^{s+1}_{1+\beta}}$ in \eqref{claim-for-commutators}.
By the elliptic estimate from Lemma \ref{Lem:ellipticReg-screened} it suffices to bound each term on the right-hand side of \eqref{eq:remainderR} in $H^{s-1}_{1+\beta}$.
\par \textbf{Term $S(\rho)$}.
We write $S(\rho)=S_1(\rho)+S_2(\rho)$ with $S_1(\rho)=\frac{1}{2}(L\rho)L((L\rho)^2)$ and $S_2(\rho)=-\frac{1}{8}\left( L(L \rho)^2\right)^2$. We may use now the weighted product estimates Lemma \ref{Lem:product-weightedSob} and elliptic estimates Lemma \ref{Lem:ellipticReg-screened}, and obtain for all $\kappa>0$ 
\begin{align*}
    \norm{S_1(\rho)}_{H^{s-1}_{1+\beta}} 
    & \lesssim \norm{\rho}_{W^{\lfloor \frac{s-1}{2} \rfloor-1, \infty}_\kappa} \norm{(L\rho)^2}_{H^{s-3}_{1+\beta-\kappa}}+\norm{\rho}_{H^{s-3}_{1+\beta-\kappa}} \norm{(L\rho)^2}_{W^{\lfloor \frac{s-1}{2} \rfloor-1, \infty}_\kappa} \\
    & \lesssim \norm{\rho}_{W^{\lfloor \frac{s-1}{2} \rfloor-1, \infty}_\kappa} \norm{\rho}_{W^{\lfloor \frac{s-3}{2} \rfloor-1, \infty}_\kappa}\norm{\rho}_{H^{s-5}_{1+\beta-2\kappa}} 
    +\norm{\rho}_{H^{s-3}_{1+\beta-\kappa}} \norm{\rho}_{W^{\lfloor \frac{s-1}{2} \rfloor-2, \infty}_{\kappa /2}}^2.
\end{align*}
All in all, we get the cubic type estimate
\begin{align*}
    \norm{S_1(\rho)}_{H^{s-1}_{1+\beta}} &\lesssim \norm{\rho}_{W^{\lfloor \frac{s-1}{2} \rfloor-1, \infty}_{\kappa/2}} \norm{\rho}_{W^{\lfloor \frac{s-1}{2} \rfloor-2, \infty}_{\kappa/2}}\norm{\rho}_{H^{s-3}_{1+\beta-\kappa}} \\
    &\lesssim 
    \norm{\vert \nabla \vert Z_\pm}_{W^{\lfloor \frac{s-1}{2} \rfloor-1, \infty}_{\kappa/2}} 
    \norm{\vert \nabla \vert Z_\pm}_{W^{\lfloor \frac{s-1}{2} \rfloor-2, \infty}_{\kappa/2}}
    \norm{ Z_\pm}_{H^{s-4}_{1+\beta-\kappa}}.
\end{align*}

For the two  first factors, we rely on Lemma \ref{Lem-decay-deriv-Weighted}, provided that the constraint $1 \leq \lfloor\frac{s-1}{2} \rfloor, \lfloor \frac{s-1}{2} \rfloor-1 \leq 4 $ is satisfied together with $9+\beta \leq s$, and we therefore choose $s=10$ to obtain decay in time from these two terms. For the last factor, we use the commutator estimate Lemma \ref{lem:commutator-weight} (choosing $\delta=\kappa/10$) and Lemma \ref{Lem:compX-weightSob} with $Z_\pm=e^{\mp it\Lambda}\mathcal{Z}_\pm$:
\begin{align}\label{eqn:solution-in-weighted-norm}
    \norm{Z_\pm}_{H^{s-4}_{1+\beta-\kappa}}\lesssim \br{t}^{1+\beta-\kappa}\norm{\mathcal{Z}_\pm}_{H^{s-4}}+\norm{\mathcal{Z}_\pm}_{H^{s-4}_{1+\beta-\kappa+\kappa/10}}\lesssim \br{t}^{1+\beta-\kappa}\eps+\norm{\mathcal{Z}_\pm}_{X}\lesssim \br{t}^{1+\beta-\kappa}\eps.
\end{align} Thus, we have the desired bound:
\begin{align*}
    \norm{S_1(\rho)}_{H^{s-1}_{1+\beta}} \lesssim \eps^3 (\log t)^2(1+t)^{-(1+\beta-\kappa)}.
\end{align*}
The term $S_2(\rho)$ is quartic, so it satisfies better bounds. Arguing as above we have using Lemma \ref{Lem-decay-deriv-Weighted}, the commutator estimate Lemma \ref{lem:commutator-weight} and \eqref{eqn:solution-in-weighted-norm} that
\begin{align*}
    \norm{S_2(\rho)}_{H^{s-1}_{1+\beta}} 
    & \lesssim \norm{\rho}_{W^{\lfloor \frac{s-1}{2} \rfloor-2, \infty}_\kappa}^2 \norm{\rho}_{W^{\lfloor \frac{s-3}{2} \rfloor-1, \infty}_\kappa}\norm{\rho}_{H^{s-5}_{1+\beta-\kappa}}\lesssim \eps^4 {(\log t)^3(1+t)^{-2(1+\beta-\kappa)}}.
\end{align*}

\textbf{Term $F(\rho)R$}. 
Proceeding as before using the product and elliptic estimates together with the weighted decay Lemma \ref{Lem-decay-deriv-Weighted} and \eqref{eqn:solution-in-weighted-norm}, we have
\begin{align*}
    \norm{F(\rho)R}_{H^{s-1}_{1+\beta}} 
    & \lesssim ( \norm{\rho}_{W^{\lfloor \frac{s-1}{2} \rfloor-1, \infty}_\kappa} +\norm{\rho}_{W^{\lfloor \frac{s-1}{2}\rfloor-2, \infty}_\kappa}^2 ) \norm{R}_{H^{s-1}_{1+\beta-\kappa}} \\
    & \qquad +(\norm{\rho}_{H^{s-3}_{1+\beta-\kappa}} +\norm{\rho}_{W^{\lfloor \frac{s-3}{2} \rfloor-1, \infty}_\kappa}\norm{\rho}_{H^{s-5}_{1+\beta-\kappa}} )\norm{R}_{W^{\lfloor \frac{s-1}{2} \rfloor, \infty}_\kappa}\\
    &\lesssim ( \norm{\vert \nabla \vert Z_\pm}_{W^{\lfloor \frac{s-1}{2} \rfloor-1, \infty}_\kappa} +\norm{\vert \nabla \vert Z_\pm}_{W^{\lfloor \frac{s-1}{2}\rfloor-2, \infty}_\kappa}^2 ) \norm{R}_{H^{s-1}_{1+\beta-\kappa}} \\
    & \qquad +(\norm{Z_\pm}_{H^{s-2}_{1+\beta-\kappa}} +\norm{\vert \nabla \vert Z_\pm}_{W^{\lfloor \frac{s-3}{2} \rfloor-1, \infty}_\kappa}\norm{Z_\pm}_{H^{s-4}_{1+\beta-\kappa}} )\norm{R}_{W^{\lfloor \frac{s-1}{2} \rfloor, \infty}_\kappa}\\
    &\lesssim \eps \log t(1+t)^{-(1+\beta-\kappa)}\norm{R}_{H^{s-1}_{1+\beta}}+\br{t}^{1+\beta-\kappa}\eps {(\log t)^3(1+t)^{-3(1+\beta-\kappa)}\eps^3}\\
    &\lesssim \eps \log t(1+t)^{-(1+\beta-\kappa)}\norm{R}_{H^{s-1}_{1+\beta}}+(\log_2t)^3(1+t)^{-2(1+\beta-\kappa)}\eps^4.
\end{align*}The first term will be absorbed on the left-hand side of the final estimate.
Note that we have also used the cubic decay of $R$ in weighted $L^\infty$ spaces which follows by a fixed point argument as above \eqref{eq:Linfty-decay-R}.

\textbf{Term $R^2/2$}. By \eqref{eq:estim-DerLinftyweight} and the cubic decay, we have: 
\begin{align}
    \norm{R^2}_{H^{s-1}_{1+\beta}} &\lesssim  \norm{R}_{W^{\lfloor \frac{s-1}{2} \rfloor, \infty}_\kappa}  \norm{R}_{H^{s-1}_{1+\beta-\kappa}} \lesssim (1+t)^{-3(1+\beta-\kappa)}\eps^3  \norm{R}_{H^{s-1}_{1+\beta-\kappa}}\lesssim \eps  \norm{R}_{H^{s-1}_{1+\beta-\kappa}}.
\end{align}

\textbf{Term $E_3(F(\rho)+R)$}. For the remaining term, we use the composition estimate \eqref{eq:compo-estimate} and obtain
\begin{align*}
    \norm{E_3(F(\rho)+R)}_{H^{s-1}_{1+\beta}} \leq  \norm{F(\rho)+R}_{H^{s-1}_{1+\beta-\kappa}}\norm{F(\rho)+R}_{W^{\lfloor \frac{s-1}{2} \rfloor, \infty}_\kappa}^2 \exp\left(\norm{F(\rho)+R}_{W^{\lfloor \frac{s-1}{2} \rfloor, \infty}_\kappa} \right).
\end{align*}
As shown before, we have using Lemma \ref{Lem-decay-deriv-Weighted} and \eqref{eqn:solution-in-weighted-norm} that
\begin{align*}
    \norm{F(\rho)}_{H^{s-1}_{1+\beta-\kappa}} &\lesssim \norm{Z_\pm}_{H^{s-2}_{1+\beta-\kappa}} +\norm{\vert \nabla \vert Z_\pm}_{W^{\lfloor \frac{s-3}{2} \rfloor-1, \infty}_\kappa}\norm{Z_\pm}_{H^{s-4}_{1+\beta-\kappa}} 
    {\lesssim \br{t}^{1+\beta-\kappa}\eps};
    \\
    \norm{F(\rho)}_{W^{\lfloor \frac{s-1}{2} \rfloor, \infty}_\kappa} & \lesssim \norm{\vert \nabla \vert Z_\pm}_{W^{\lfloor \frac{s-1}{2} \rfloor-1, \infty}_\kappa} +\norm{\vert \nabla \vert Z_\pm}_{W^{\lfloor \frac{s-1}{2}\rfloor-2, \infty}_\kappa}^2 \lesssim{\log t} (1+t)^{-(1+\beta-\kappa)}\eps.
\end{align*}

From this it follows:
\begin{align}
    \norm{E_3(F(\rho)+R)}_{H^{s-1}_{1+\beta}}& \lesssim ( \br{t}^{1+\beta-\kappa}\eps+\norm{R}_{H^{s-1}_{1+\beta}}) \Bigg((\log t)^2(1+t)^{-2(1+\beta-\kappa)}\eps^2+\norm{R}_{W^{\lfloor \frac{s-1}{2} \rfloor, \infty}_\kappa}^2 \Bigg)\exp(\eps)
    \\
    &\lesssim (\log t)^2(1+t)^{-(1+\beta-\kappa)}\eps^3+(\log t)^2(1+t)^{-2(1+\beta-\kappa)}\eps^2 \norm{R}_{H^{s-1}_{1+\beta}}.
\end{align}
Note that in the last line we have used the weighted cubic decay of $R$ in terms of the solution $Z_\pm$
and Lemma \ref{Lem-decay-deriv-Weighted}. Finally, we can combine all the estimates and obtain
\begin{align}
    \norm{R(t)}_{H^{s-1}_{1+\beta}}\lesssim  \eps^3 {(\log_2t)^2(1+t)^{-(1+\beta-\kappa)}}+\eps \log_2t(1+t)^{-(1+\beta-\kappa)}\norm{R}_{H^{s-1}_{1+\beta}}+ \eps \norm{R}_{H^{s-1}_{1+\beta}},
    \end{align}
    from which it follows
    \begin{align}
    \norm{R(t)}_{H^{s-1}_{1+\beta}}\lesssim (\log_2t)^2(1+t)^{-(1+\beta-\kappa)}\eps^3.
\end{align}
This concludes the proof.

\end{proof}

\section*{Acknowledgements}
L.E.\ would like to acknowledge Aymeric Baradat and Daniel Han-Kwan for numerous fruitful discussions about pressureless Euler systems, and for our collaboration on related problems over the past years. C.J.\ would like to thank Benoît Pausader for inspiring discussions on this and related problems. Finally, C.J.\ and K.W.\ gratefully acknowledge support of the SNSF through grant PCEFP2\_203059.

\bibliographystyle{alpha}

\bibliography{biblio}

\end{document}